\documentclass{article}
\makeatletter
\def\@testdef #1#2#3{%
  \def\reserved@a{#3}\expandafter \ifx \csname #1@#2\endcsname
 \reserved@a  \else
\typeout{^^Jlabel #2 changed:^^J%
\meaning\reserved@a^^J%
\expandafter\meaning\csname #1@#2\endcsname^^J}%
\@tempswatrue \fi}
\def\colorful{0}

\usepackage[preprint,nonatbib]{neurips_2023}

\usepackage[utf8]{inputenc} 
\usepackage[T1]{fontenc}    
\usepackage{hyperref}       
\usepackage{url}            
\usepackage{booktabs}       
\usepackage{amsfonts}       
\usepackage{nicefrac}       
\usepackage{microtype}      
\usepackage{xcolor}

\usepackage[
style=alphabetic,
maxbibnames=99,
giveninits=true]{biblatex}
\usepackage{amssymb, mathtools, amsmath, amsthm, amssymb,color}
\usepackage[ruled, vlined, linesnumbered]{algorithm2e}
\numberwithin{algocf}{section}

\usepackage{bbm}
\usepackage{breqn}

\usepackage[shortlabels]{enumitem}

\DeclareMathOperator{\vect}{vec}
\DeclareMathOperator{\Var}{Var}

\DeclareMathOperator{\rank}{rank}
\DeclareMathOperator{\Cov}{Cov}
\DeclareMathOperator{\tr}{tr}
\DeclareMathOperator{\op}{op}

\DeclarePairedDelimiter\brackets{(}{)}
\newcommand{\bigO}[2][]{\ifthenelse{\equal{#1}{}}{O\!\left(#2\right)}{O\brackets[#1]{#2}}}
\newcommand{\bigtO}[2][]{\ifthenelse{\equal{#1}{}}{\tilde{O}\!\left(#2\right)}{\tilde{O}\brackets[#1]{#2}}}
\newcommand{\bigOmega}[2][]{\ifthenelse{\equal{#1}{}}{\Omega\!\left(#2\right)}{\Omega\brackets[#1]{#2}}}
\newcommand{\bigtOmega}[2][]{\ifthenelse{\equal{#1}{}}{\tilde{\Omega}\!\left(#2\right)}{\tilde{\Omega}\brackets[#1]{#2}}}

\DeclarePairedDelimiter\doubleVertical{\|}{\|}
\newcommand{\norm}[2][]{\ifthenelse{\equal{#1}{}}{\left\|#2\right\|}{\doubleVertical[#1]{#2}}}
\newcommand{\opnorm}[2][]{\ifthenelse{\equal{#1}{}}{\left\|#2\right\|_{\op}}{\doubleVertical[#1]{#2}_{\op}}}
\newcommand{\fnorm}[2][]{\ifthenelse{\equal{#1}{}}{\left\|#2\right\|_{F}}{\doubleVertical[#1]{#2}_{F}}
}

\newcommand{\la}{\leftarrow}

\newcommand{\ra}{\rightarrow}

\newcommand{\Ind}{\mathbbm{1}}

\newcommand{\B}{\mathcal{B}}
\newcommand{\C}{\mathcal{C}}

\newcommand{\G}{\mathcal{G}}

\DeclareMathOperator*{\E}{\mathbb{E}}
\DeclareMathOperator{\normal}{\mathcal{N}}
\newcommand{\eps}{\epsilon}

\newcommand{\N}{\mathbb{N}}
\newcommand{\Prob}{\mathbb{P}}

\newcommand{\R}{\mathbb{R}}
\newcommand{\Z}{\mathbb{Z}}

\newcommand{\Lin}{\mathcal{L}} 

\newcommand{\X}{\mathcal{X}}
\newcommand{\Y}{\mathcal{Y}}
\newcommand{\grad}{\nabla}

\newcommand{\<}{\langle}
\renewcommand{\>}{\rangle} 

\newcounter{thm}
\numberwithin{thm}{section}

\theoremstyle{plain}
\newtheorem{theorem}[thm]{Theorem}
\newtheorem{corollary}[thm]{Corollary}
\newtheorem{lemma}[thm]{Lemma}
\newtheorem{proposition}[thm]{Proposition}
\newtheorem{claim}[thm]{Claim}
\newtheorem{assumption}[thm]{Assumption}
\newtheorem{fact}[thm]{Fact}

\theoremstyle{definition}
\newtheorem{definition}[thm]{Definition}

\theoremstyle{remark} 

\newcommand{\abs}[1]{\left|{#1}\right|}

\newcommand{\aspace}{\mathcal{A}}
\DeclareMathOperator{\dist}{dist}

\DeclareMathOperator*{\argmin}{\arg\min}

\DeclareMathOperator{\poly}{poly}

\makeatletter
\providecommand*{\diff}{
  \@ifnextchar^{\DIfF}{\DIfF^{}}}

\def\DIfF^#1{%
  \mathop{\mathrm{\mathstrut d}}%
  \nolimits^{#1}\gobblespace}
\def\gobblespace{%
  \futurelet\diffarg\opspace}
\def\opspace{%
  \let\DiffSpace\!%
  \ifx\diffarg(%
  \let\DiffSpace\relax
  \else
  \ifx\diffarg[%
  \let\DiffSpace\relax
  \else
  \ifx\diffarg\{%
  \let\DiffSpace\relax
  \fi\fi\fi\DiffSpace}

\providecommand*{\dd}[3][]{%
  \frac{\diff^{#1}#2}{\diff #3^{#1}}}

\renewcommand{\tilde}{\widetilde}
\renewcommand{\hat}{\widehat}

\newcommand{\Crme}{C_{est}}
\newcommand{\inexact}{\tilde}

\newcommand{\cXstar}{\mathcal{X}^\star}
\newcommand{\sigstarr}{{\sigma^{\star}_r}}
\newcommand{\sigstarl}{{\sigma^\star_1}}
\newcommand{\projX}{\mathcal{P}_{\mathcal{X}^\star}}
\newcommand{\flb}{f^*}

\SetCommentSty{mycommfont}

\newcommand{\jd}[1]{{\color{purple}{\textbf{JD:} #1}}}

\newcommand{\syl}[1]{{\color{purple}{\textbf{Shuyao:} #1}}}

\ifnum\colorful=1
\newcommand{\new}[1]{{\color{red} #1}}

\else
\newcommand{\new}[1]{{#1}}

\fi

\begin{document}

\title{Robust Second-Order Nonconvex Optimization and\\ Its Application to Low Rank Matrix Sensing}
\author{Shuyao Li\\
  University of Wisconsin-Madison\\
  \texttt{shuyao.li@wisc.edu} 
  \And
  Yu Cheng \\
  Brown University\\
  \texttt{yu\_cheng@brown.edu} \\
  \AND
  Ilias Diakonikolas \\
  University of Wisconsin-Madison \\
  \texttt{ilias@cs.wisc.edu}
  \And
  Jelena Diakonikolas\\
  University of Wisconsin-Madison \\
  \texttt{jelena@cs.wisc.edu} \\
  \AND
  Rong Ge \\
  Duke University \\
  \texttt{rongge@cs.duke.edu}
  \And
  Stephen Wright \\
  University of Wisconsin-Madison \\
  \texttt{swright@cs.wisc.edu} \\
}
\date{}

\maketitle
\begin{abstract}
Finding an approximate second-order stationary point (SOSP) 
is a well-studied and fundamental problem in stochastic nonconvex optimization with many applications in machine learning.
However, this problem is poorly understood in the presence of outliers, limiting the use of existing nonconvex algorithms in adversarial settings.

In this paper, we study the problem of finding SOSPs in the strong contamination model, 
where a constant fraction of datapoints are arbitrarily corrupted.
We introduce a general framework for efficiently finding an approximate SOSP with \emph{dimension-independent} accuracy guarantees, using $\widetilde{O}({D^2}/{\epsilon})$ samples where $D$ is the ambient dimension and $\epsilon$ is the fraction of corrupted datapoints.

As a concrete application of our framework, we apply it to the problem of low rank matrix sensing, developing efficient and provably robust algorithms that can tolerate corruptions in both the sensing matrices and the measurements.
In addition, we establish a Statistical Query lower bound providing evidence that the quadratic dependence on $D$ in the sample complexity is necessary for computationally efficient algorithms.

  
\end{abstract}

\section{Introduction}\label{sec:introduction}


Learning in the presence of corrupted data is a significant challenge in machine learning (ML) with many applications, 
including ML security~\cite{Barreno2010,BiggioNL12, SteinhardtKL17, diakonikolas2019sever} 
and exploratory data analysis of real datasets, 
e.g., in biological settings~\cite{RP-Gen02, Pas-MG10, Li-Science08, diakonikolas2017being-robust-in}.
The goal in such scenarios is to design efficient learning algorithms that can tolerate a small constant fraction of outliers and achieve error guarantees independent of the dimensionality of the data.
Early work in robust statistics~\cite{HampelEtalBook86, Huber09}
gave sample-efficient robust estimators for various tasks 
(e.g., the Tukey median~\cite{Tukey75} for robust mean estimation), alas with runtimes exponential in the dimension.
A recent line of work in computer science, 
starting with~\cite{diakonikolas2016robust-estimato, LaiRV16}, developed the first computationally efficient 
robust algorithms for several fundamental high-dimensional tasks. 
Since these early works, there has been significant progress in algorithmic aspects of robust high-dimensional statistics (see~\cite{diakonikolas2019recent-advances}~and~\cite{diakonikolas2022algorithmic} for a comprehensive overview).

In this paper, we study the general problem of smooth (with Lipschitz gradient and Hessian) 
stochastic nonconvex optimization $\min_x \bar{f}(x)$ in the outlier-robust setting, 
where \(\bar f(x) := \E_{A \sim \G} f(x, A)\) and $\G$ is a possibly unknown distribution 
of the random parameter $A$.
We will focus on the following standard adversarial contamination model (see, e.g.,~\cite{diakonikolas2016robust-estimato}). 

\begin{definition}[Strong Contamination Model]
  \label{def:corruption}
  Given a parameter \(0 < \eps < 1/2\) and an inlier distribution \(\G\), an algorithm receives samples from \(\G\) with \(\epsilon\)-contamination as follows: The algorithm first speciﬁes the number of samples \(n\), and \(n\) samples are drawn independently from  \(\G\). An adversary is then allowed to inspect these samples, and replace an \(\epsilon\)-fraction of the samples with arbitrary points. This modified set of \(n\) points is said to be $\eps$-corrupted, which is then given to the algorithm.
\end{definition}

The stochastic optimization problem we consider 
is computationally intractable in full 
generality --- even without corruption --- if the goal is to obtain globally optimal solutions.
At a high level, an achievable goal is to design sample and computationally efficient robust algorithms for finding {\em locally} optimal solutions.
Prior work~\cite{prasad2020robust-estimati, diakonikolas2019sever} studied outlier-robust stochastic optimization and obtained efficient algorithms for finding approximate {\em first-order} stationary points.
While first-order guarantees suffice for convex problems, it is known that in many tractable non-convex problems, first-order stationary points may be bad solutions, but all {\em second-order} stationary points (SOSPs) are globally optimal. This motivates us to study the following questions: 

\begin{quote}
\em Can we develop a general framework for finding {\bf\em second-order} stationary points in outlier-robust stochastic optimization?

Can we obtain sample and computationally efficient algorithms for outlier-robust versions of tractable nonconvex problems using this framework?
\end{quote}

In this work, we answer both questions affirmatively.
We introduce a framework for efficiently finding an approximate SOSP when $\eps$-fraction of the functions are corrupted and then use our framework to solve the problem of outlier-robust low rank matrix sensing.

In addition to the gradient being zero, a SOSP requires the Hessian matrix to not have negative eigenvalues.
The second-order optimality condition is important because it rules out suboptimal solutions such as strict saddle points.
It is known that all SOSPs are globally optimal in nonconvex formulations of many important machine learning problems, such as matrix completion~\cite{ge2016matrix}, matrix sensing~\cite{bhojanapalli2016global}, phase retrieval~\cite{sun2016geometric}, phase synchronization~\cite{bandeira2016low}, dictionary learning~\cite{sun2016dictionary}, and tensor decomposition~\cite{ge2015escaping} (see also \cite[Chapter~7]{wright2022high}). However, the properties of SOSPs are highly sensitive to perturbation in the input data.
For example, it is possible to create spurious SOSPs for nonconvex formulations of low rank matrix recovery problems, even for a semi-random adversary that can add additional sensing matrices but cannot corrupt the measurements in matrix sensing~\cite{GaoC23} or an adversary who can only reveal more entries of the ground-truth matrix in matrix completion~\cite{cheng2018non}. Those spurious SOSPs correspond to highly suboptimal solutions.

Finding SOSPs in stochastic nonconvex optimization problems in the presence of arbitrary outliers was largely unaddressed prior to our work. 
Prior works~\cite{prasad2020robust-estimati, diakonikolas2019sever} obtained efficient and robust algorithms for finding {\em first-order} stationary points with dimension-independent accuracy guarantees.
These works relied on the following simple idea: 
Under certain smoothness assumptions,
projected gradient descent with an {\em approximate} gradient oracle efficiently converges to an {\em approximate} first-order stationary point. 
Moreover, in the outlier-robust setting, approximating the gradient at a specific point amounts to a robust mean estimation problem (for the underlying distribution of the gradients), which can be solved by leveraging existing algorithms for robust mean estimation.

Our work is the first to find approximate SOSPs with dimension-independent errors in outlier-robust settings. Note that in standard non-robust settings, approximate SOSPs can be computed using first-order methods such as perturbed gradient descent~\cite{jin2017howto, jin2021on-nonconvex-op}. This strategy might seem extendable to outlier-robust settings through perturbed approximate gradient descent, utilizing robust mean estimation algorithms to approximate gradients.
The approach in~\cite{yin2019defending} follows this idea, but unfortunately their second-order guarantees scale polynomially with dimension, even under very strong distributional assumptions (e.g., subgaussianity). Our lower bound result provides evidence that approximating SOSPs with dimension-independent error is as hard as approximating \emph{full} Hessian, suggesting that solely approximating the gradients is not sufficient.
On a different note, \cite{IPL23} recently employed robust estimators for both gradient and Hessian in solving certain convex stochastic optimization problems, which has a different focus than ours and does not provide SOSPs with the guarantees that we achieve.



\subsection{Our Results and Contributions}
The notation we use in this section is defined in Section~\ref{sec:prelims}.
To state our results, we first formally define our generic nonconvex optimization problem.
Suppose there is a true distribution over functions \(f: \R^{D} \times \aspace \ra \R\), where $f(x, A)$ takes an argument $x \in \R^{D}$ and is parameterized by a random variable \(A \in \aspace\) drawn from a distribution \(\G\). Our goal is to find an $(\eps_g, \eps_H)$-approximate SOSP of the function \(\bar f(x):= \E_{A \sim \G} f(x, A)\).

\begin{definition}[$\eps$-Corrupted Stochastic Optimization]\label{def:robust-stochastic-opt}
The algorithm has access to \(n\) functions $(f_{i})_{i=1}^n$ generated as follows.
First \(n\) random variables $(A_i)_{i=1}^n$ are drawn independently from $\G$.
Then an adversary arbitrarily corrupts an \(\epsilon\) fraction of the $A_i$'s. 
Finally, the \(\epsilon\)-corrupted version of \(f_{i}(\cdot) = f(\cdot, A_{i})\) is sent to the algorithm as input.
The task is to find an approximate SOSP of the ground-truth average function $\bar f(\cdot) := \E_{A\sim \G} f(\cdot, A)$.
\end{definition}

\begin{definition}[Approximate SOSPs]\label{def:SOSP}
  A point $x$ is an {$(\epsilon_g, \epsilon_H)$-approximate second-order stationary point (SOSP)} of \(\bar{f}\) if $\norm{\grad \bar{f}(x)} \le \epsilon_g$ and ${\lambda_\min \left(\grad^2 \bar{f}(x)\right) \ge -\epsilon_H}$.
\end{definition}

We make the following additional assumptions on $f$ and $\G$. 

\begin{assumption}
  \label{assump:general_clean_data}
   There exists a bounded region \(\mathcal{B}\) such that the following conditions hold:
  \begin{enumerate}[leftmargin=2em]
    \item[(i)] There exists a lower bound $\flb > -\infty$ such that for all $x \in \mathcal{B}$, $f(x, A) \ge \flb$ with probability $1$.
    \item[(ii)] There exist parameters $L_{D_g}$, $L_{D_H}$, $B_{D_g}$, and $B_{D_H}$ such that, with high probability over the randomness in $A \sim \G$, letting $g(x) = f(x, A)$, we have that \(g(x)\) is \(L_{D_g}\)-gradient Lipschitz and \(L_{D_H}\)-Hessian Lipschitz over $\mathcal{B}$, and $\norm{\grad g(x)} \le B_{D_g}$ and $\fnorm{\grad^2 g(x)} \le B_{D_H}$ for all $x \in \mathcal{B}$.
    \item[(iii)]  There exist parameters \(\sigma_{g}, \sigma_{H} > 0\) such that for all $x \in \B$,
    
    \(\opnorm{\Cov_{A \sim \G}(\grad f(x,A))}  \le \sigma_{g}^{2}\) and 
    \(\opnorm{\Cov_{A \sim \G}(\vect(\grad^{2} f(x,A)))} 
    \le \sigma_{H}^{2}\).
  \end{enumerate}
\end{assumption}
Note that the radius of $\mathcal{B}$ and  
  the parameters $L_{D_g}$, $L_{D_H}$, $B_{D_g}$, $B_{D_H}$ are all allowed to depend polynomially on $D$ and $\eps$ (but not on $x$ and $A$).

  Our main algorithmic result for $\eps$-corrupted stochastic optimization is summarized in the following theorem.
  A formal version of this theorem is stated as Theorem~\ref{thm:global_convergence-general} in Section~\ref{sec:general_nonconvex}.

\begin{theorem}[Finding an Outlier-Robust SOSP, informal]
  \label{thm:global_convergence-general-intro}
  \new{Suppose $f$ satisfies Assumption~\ref{assump:general_clean_data} in a region \(\mathcal{B}\) with parameters $\sigma_{g}$ and $\sigma_{H}$.}
  Given an arbitrary initial point \(x_{0} \in \B\) and an $\eps$-corrupted set of $n = \bigtOmega{D^{2}/\epsilon}$ functions where $D$ is the ambient dimension, there exists a polynomial-time algorithm that with high probability outputs an \((\bigO{\sigma_{g}\sqrt\epsilon}, \bigO{\sigma_{H}\sqrt\epsilon})\)-approximate SOSP of $\bar f$, provided that all iterates of the algorithm stay inside \(\B\).
\end{theorem}

\new{Although the bounded iterate condition in Theorem~\ref{thm:global_convergence-general-intro} appears restrictive, this assumption holds if the objective function satisfies a ``dissipativity'' property, which is a fairly general phenomenon~\cite{hale2010asymptotic}.
Moreover, adding an $\ell_2$-regularization term enables any Lipschitz function to satisfy the dissipativity property~\cite[{Section 4}]{raginsky17nonconvex}.
As an illustrating example, a simple problem-specific analysis shows that this bounded iterate condition holds for outlier-robust matrix sensing by exploiting the fact that the matrix sensing objective satisfies the dissipativity property.}

In this paper, we consider the problem of outlier-robust symmetric low rank matrix sensing, which we formally define below. We focus on the setting with Gaussian design.

\begin{definition}[Outlier-Robust Matrix Sensing]
\label{def:iid_Gaussian_with_noise}
    There is an unknown rank-$r$ ground-truth matrix \(M^{*} \in \R^{d \times d}\) that can be factored into $U^{*} {U^{*}}^\top$ where $U^* \in \R^{d \times r}$.   
    The (clean) sensing matrices \(\{A_{i}\}_{i \in [n]}\) have i.i.d.\ standard Gaussian entries. The (clean) measurements \(y_{i}\) are obtained as \(y_i = \<A_{i}, M^{*}\> + \zeta_{i}\), where the noise \(\zeta_{i} \sim \mathcal{N}(0,\sigma^{2})\) is independent from all other randomness. We denote the (clean) data generation process by \((A_{i}, y_{i}) \sim \G_{\sigma}\). When $\sigma=0$, we have $\zeta_{i} = 0$ and we write \(\G := \G_0\) for this noiseless (measurement) setting.
    In outlier robust matrix sensing, an adversary can arbitrarily change any $\eps$-fraction of the sensing matrices and the corresponding measurements.
    This corrupted set of $(A_i, y_i)$'s is then given to the algorithm as input, where the goal is to recover $M^*$.
\end{definition}

We highlight that in our setting, both the sensing matrices \(A_{i} \in \R^{d \times d}\) and the measurements \(y_{i} \in \R\) can be corrupted,  \new{presenting a substantially more challenging problem compared to} prior works (e.g.,~\cite{li2020non,li2020nonconvex}) that only allow corruption in \(y_{i}\). 

Let \(\sigstarl\) and \(\sigstarr\) denote the largest and the smallest nonzero singular value of \(M^{*}\) respectively. We assume \(\sigstarr\) and the rank $r$ are given to the algorithm,
and we assume that the algorithm knows a multiplicative upper bound \(\Gamma\) of \(\sigstarl\) such that \(\Gamma \ge 36 \sigstarl\) (\new{a standard assumption in matrix sensing even for non-robust settings~\cite{ge2017no-spurious-loc,jin2017howto}}). Let \(\kappa = \Gamma/\sigstarr\).

Our main algorithmic result for the low rank matrix sensing problem is summarized in the following theorem. For a more detailed statement, see Theorems~\ref{thm:together} and \ref{thm:noisy_combined} in Section~\ref{sec:general_nonconvex}. 
\begin{theorem}[Our Algorithm for Outlier-Robust Matrix Sensing]\label{thm:low-rank-mtrx-sensing-intro}
    Let $M^* \in \R^{d \times d}$ be the rank $r$ ground-truth matrix with smallest nonzero singular value $\sigstarr$.
    Let $\Gamma \ge 36 \opnorm{M^*}$ and let \(\kappa = \Gamma/\sigstarr\).
    There exists an algorithm for outlier-robust matrix sensing, where an \(\epsilon = \bigO[\big]{1/{(\kappa^{3}r^{3})}}\) fraction of samples from \(\G_{\sigma}\) as in Definition~\ref{def:iid_Gaussian_with_noise} gets arbitrarily corrupted, that can output a rank-$r$ matrix $\hat{M}$ such that \(\fnorm[\big]{\hat M - M^{*}} \leq \iota\) with probability at least \(1-\xi\), where \(\iota > 0\) is the error parameter:
    \begin{enumerate}[leftmargin=*]
        \item[1)] If $\sigma \geq r\Gamma$, then \(\iota = \bigO{\sigma \sqrt{\eps}}\);
        \item[2)] If $\sigma \leq r\Gamma$, then \(\iota = \bigO{\kappa \sigma \sqrt{\eps}}\);
        \item[3)] If \(\sigma = 0\) (noiseless), then \(\iota\) can be made arbitrarily small, achieving exact recovery.
    \end{enumerate}
    The algorithm uses \(n = \bigtO[\Big]{\frac{d^{2}r^{2} + d r \log(\Gamma/\xi)}{\epsilon}}\) samples and runs in time $\poly(n, \kappa, \log(\sigstarr/\iota))$.
\end{theorem}

Finally, we complement our algorithmic results for outlier-robust matrix sensing with a Statistical Query (SQ) lower bound, which provides strong evidence that quadratic dependence on $d$ in the sample complexity is unavoidable for efficient algorithms. A detailed statement of this result is provided in Section~\ref{sec:sq-lower-bound}.

\subsection{Our Techniques}
\paragraph{Outlier-robust nonconvex optimization.}
To obtain our algorithmic result in the general nonconvex setting, we leverage existing results on robust mean estimation~\cite{diakonikolas2020outlier}, which we use as a black box to robustly estimate the gradient and the (vectorized) Hessian.
We use these robust estimates as a subroutine in a randomized nonconvex optimization algorithm (described in Appendix~\ref{sec:inexact-randomized-short-proof-appendix}), which can tolerate inexactness in both the gradient and the Hessian.
With high probability, this algorithm outputs an \((\eps_g, \eps_H)\)-approximate SOSP, where \(\eps_g\) and \(\eps_H\) depend on the inexactness of the gradient and Hessian oracles.

We remark that robust estimation of the Hessian is crucial to obtaining our \emph{dimension-independent} approximation error result and is what causes $D^2$ dependence in the sample complexity (which is unavoidable for SQ algorithms as discussed below).
Notably, the only prior work on approximating SOSPs in the outlier-robust setting \cite{yin2019defending} used robust mean estimation only on the gradients and had sample complexity scaling linearly with $D$;
however, they can only output an order $(\sqrt{\eps}, (\eps D)^{1/5})$-SOSP, which is uninformative for many problems of interest, including the matrix sensing problem considered in this paper, due to the dimensional dependence in the approximation error. 

\paragraph{Application to low rank matrix sensing.} Our main contribution on the algorithmic side is showing that our outlier-robust nonconvex optimization framework can be applied to solve outlier-robust matrix sensing with \emph{dimension-independent} approximation error, even achieving exact recovery when the measurements are noiseless.
We obtain this result using the following geometric insights about the problem: We show that the norm of the covariance of the gradient and the Hessian can both be upper bounded by the sum of $\sigma^2$ and a function of the distance to the closest optimal solution. 
We further prove that all iterates stay inside a nice region using a ``dissipativity'' property~\cite{hale2010asymptotic}), which says that the iterate aligns with the direction of the gradient when the iterate's norm is large.
This allows us to invoke Theorem~\ref{thm:global_convergence-general-intro} to obtain an approximate SOSP of the ground-truth objective function.

We show that this approximate SOSP must be close to a global optimal solution.
Additionally, we establish a local regularity condition in a small region around globally optimal solutions (which is similar to strong convexity but holds only locally). 
This local regularity condition bounds below a measure of stationarity, which allows us to prove that gradient descent-type updates contract the distance to the closest global optimum under appropriate stepsize. 
We leverage this local regularity condition to prove that the iterates of the algorithm stay near a global optimum, so that the regularity condition continues to hold, and moreover, the distance between the current solution and the closest global optimum contracts, as long as it is larger than a function of $\sigma$.
Consequently, the distance-dependent component of the gradient and Hessian covariance bound contracts as well, which allows us to obtain more accurate gradient and Hessian estimates. 
While such a statement may seem evident to readers familiar with linear convergence arguments, we note that proving it is quite challenging, due to the circular dependence between the distance from the current solution to global optima, the inexactness in the gradient and Hessian estimates, and the progress made by our algorithm. 

The described distance-contracting argument allows us to control the covariance of the gradient and Hessian, which we utilize to recover $M^{*}$ exactly when \(\sigma = 0\), and recover $M^{*}$ with error roughly \(O(\sigma \sqrt{\epsilon})\) when \(0 \neq \sigma \leq r \Gamma\). We note that the \(\sigma \sqrt{\epsilon}\) appears unavoidable in the \(\sigma \neq 0\) case, due to known limits of robust mean estimation algorithms~\cite{diakonikolas2016statistical}.

\paragraph{SQ lower bound.} 
We exhibit a hard instance of low rank matrix sensing problem to show that quadratic dependence on the dimension in sample complexity is unavoidable for computationally efficient algorithms. 
Our SQ lower bound proceeds by constructing a family of distributions, corresponding to 
corruptions of low rank matrix sensing, that are nearly uncorrelated in a well-defined technical sense~\cite{FeldmanGRVX17}. To 
achieve this, we follow the framework of~\cite{diakonikolas2016statistical} which considered 
a family of distributions that are rotations of a carefully constructed 
one-dimensional distribution. The proof builds on~\cite{diakonikolas2019efficient-algor, diakonikolas2021statistical}, 
using a new univariate moment-matching construction which yields 
a family of corrupted conditional distributions. These
induce a family of joint distributions that are SQ-hard to learn.


\subsection{Roadmap}
Section~\ref{sec:prelims} defines the necessary notation and discusses relevant building blocks of our algorithm and analysis. Section~\ref{sec:general_nonconvex} introduces our framework for finding SOSPs in the outlier-robust setting. Section~\ref{sec:matrix-low-rank} presents how to extend and apply our framework to solve outlier-robust low rank matrix sensing. Section~\ref{sec:sq-lower-bound} proves that our sample complexity has optimal dimensional dependence for SQ algorithms. Most proofs are deferred to the supplementary material due to space limitations.

\section{Preliminaries}\label{sec:prelims}



 For an integer \(n\), we use \([n]\) to denote the ordered set \(\{1, 2, \dots, n\}\). We use \([a_{i}]_{i \in \mathcal{I}}\) to denote the matrix whose columns are vectors \(a_{i}\), where \(\mathcal{I}\) is an ordered set. We use \(\Ind_{E}(x)\) to denote the indicator function that is equal to $1$ if \(x \in E\) and $0$ otherwise. For two functions \(f\) and \(g\), we say \(f = \bigtO{g}\) if \(f = \bigO[\big]{g \log^k(g)}\) for some constant $k$, and we similarly define \(\tilde\Omega\).

 For vectors \(x\) and \(y\), we let \(\<x, y\>\) denote the inner product \(x^{\top} y\) and \(\norm{x}\) denote the \(\ell_{2}\) norm of \(x\). For \(d \in \Z_{+}\), we use \(I_{d}\) to denote the identity matrix of size \(d \times d\). For matrices \(A\) and \(B\), we use \(\opnorm{A}\) and \(\fnorm{A}\) to denote the spectral norm and Frobenius norm of \(A\) respectively. We use \(\lambda_{\max}(A)\) and \(\lambda_{\min}(A)\) to denote the maximum and minimum eigenvalue of \(A\) respectively. We use \(\tr(A)\) to denote the trace of a matrix \(A\).  We use \(\<A, B\> = \tr(A^{\top} B) \) to denote the entry-wise inner product of two matrices of the same dimension.
 We use \(\vect(A) = [a_{1}^{\top}, a_{2}^{\top}, \dots, a_{d}^{\top}]^{\top}\) to denote the canonical flattening of \(A\) into a vector, where \(a_{1}, a_{2}, \dots, a_{d}\) are columns of \(A\).

\begin{definition}[Lipschitz Continuity]\label{def:lipschitz}
  Let \(\X\) and \(\Y\) be normed vector spaces. A function \(h: \X \rightarrow \Y\) is {\(\ell\)-Lipschitz} if
  $\norm{h(x_{1}) - h(x_{2})}_{\Y} \le \ell \norm{x_{1} - x_{2}}_{\X},$ $\forall x_1, x_2.$

  In this paper, when \( \Y \) is a space of matrices, we take \( \norm{\cdot}_{\Y}\) to be the spectral norm \(\opnorm{\cdot}\). When \(\X\) is a space of matrices, we take \(\norm{\cdot}_{\X}\) to be the Frobenius norm \(\fnorm{\cdot}\); this essentially views the function \(h\) as operating on the vectorized matrices endowed with the usual \(\ell_{2}\) norm. When $\X$ or $\Y$ is the Euclidean space, we take the corresponding norm to be the $\ell_2$ norm.
\end{definition}

\paragraph{A Randomized Algorithm with Inexact Gradients and Hessians.}
We now discuss how to solve the unconstrained nonconvex optimization problem
\(\min_{x \in \R^{D}} f(x),\)
where \(f(\cdot)\) is a smooth function with Lipschitz gradients and Lipschitz
Hessians. The goal of this section is to find an approximate SOSP as defined in Definition~\ref{def:SOSP}.





\begin{proposition}[\cite{li2023randomized}]
\label{prop:optimization_inexact_derivatives_general}
    Suppose a function \(f\) is bounded below by $f^* > -\infty,$ has $L_g$-Lipschitz gradient and $L_H$-Lipschitz Hessian, and its
  inexact gradient and Hessian computations \(\inexact g_{t}\) and \(\inexact H_{t}\) satisfy
  \(\norm{\inexact g_{t} - \grad{f}(x_{t})} \le \frac{1}{3} \epsilon_{g}\) and \(\opnorm[\big]{\inexact H_{t} - \grad^{2} f(x_{t})} \leq \frac29 \epsilon_{H}\).
  Then there exists an algorithm (Algorithm~\ref{alg:inexact_randomized}) with the following guarantees:
  \begin{enumerate}[leftmargin=*]
    \item  (Correctness) If Algorithm~\ref{alg:inexact_randomized} terminates and outputs \(x_{n}\), then \(x_{n}\) is a
    \((\frac43\epsilon_g, \frac{4}{3}\epsilon_{H})\)-approximate SOSP.
    \item (Runtime) Algorithm~\ref{alg:inexact_randomized} terminates with probability $1$. Let
    \(C_{\epsilon}:= \min\left(\frac{\epsilon_{g}^{2}}{6L_{g}},\frac{2\epsilon_{H}^{3}}{9L_{H}^{2}}\right)\). With probability at least \(1-\delta\),
    Algorithm~\ref{alg:inexact_randomized} terminates after $k$ iterations for
    \begin{equation}
      \label{eq:high_probability_bound_general}
       k = \bigO[\Big]{ \frac{f(x_{0}) - \flb}{C_{\epsilon}} +  \frac{L_{H}^{2} L_{g}^{2} \epsilon_{g}^{2}}{\epsilon_{H}^{6}} \log\Bigl(\frac1\delta \Bigr)} \; .
    \end{equation}
    \end{enumerate}
\end{proposition}
The constants \(1/3\) and \(2/9\) are chosen for ease of presentation. For all constructions of Hessian oracles in this paper, we take the straightforward relaxation \(\opnorm[\big]{\inexact H_{t} - \grad^{2} f(x_{t})} \le \fnorm[\big]{\inexact H_{t} - \grad^{2} f(x_{t})}\) and upper bound Hessian inexactness using  Frobenius norm.
Proof of a simplified version of Proposition~\ref{prop:optimization_inexact_derivatives_general} with a weaker high probability bound that is sufficient for our purposes is provided in Appendix~\ref{sec:inexact-randomized-short-proof-appendix} for completeness.

\paragraph{Robust Mean Estimation.} %
Recent algorithms in robust mean estimation give dimension-independent error in the presence of outliers under strong contamination model.

We use the following results, see, e.g.,~\cite{diakonikolas2020outlier}, 
where the upper bound \(\sigma\) on the spectral
norm of covariance matrix is unknown to the algorithm.

\begin{proposition}[Robust Mean Estimation]
  \label{prop:robust-mean-estimation-summary}
  Fix any $0<\xi<1$. Let $S$ be a multiset of $n = O((k \log k + \log (1 / \xi))/\epsilon)$ i.i.d.\ samples from a distribution on $\mathbb{R}^k$ with mean $\mu_{S}$ and covariance $\Sigma$. Let $T \subset \R^k$ be an $\epsilon$-corrupted version of $S$ as in Definition~\ref{def:corruption}.
  There exists an algorithm (Algorithm~\ref{alg:robust_mean_estimation}) such that, with probability at least $1-\xi$, on input $\epsilon$ and  $T$ (but not $\opnorm{\Sigma}$) returns a vector $\widehat{\mu}$ in polynomial time so that $\norm{\mu_S-\widehat{\mu}}=\bigO[\big]{ \sqrt{\opnorm{\Sigma}\epsilon}}$.
\end{proposition}

Algorithm~\ref{alg:robust_mean_estimation} is given in Appendix~\ref{sec:robust-mean-estimation-appendix}. Proposition~\ref{prop:robust-mean-estimation-summary} states that for Algorithm~\ref{alg:robust_mean_estimation} to succeed with high probability, 
\(\bigtO{{k}/{\epsilon}}\) i.i.d.\ samples need to be drawn from a \(k\)-dimensional distribution of bounded covariance.  State of the art algorithms for robust mean estimation can be implemented in near-linear time, 
requiring only a logarithmic number of passes on the data, see, e.g,~\cite{diakonikolas22streaming, ChengDG19, dong2019quantum}. Any of these faster algorithms could be used for our purposes.

With the above results, the remaining technical component for applying the robust estimation subroutine (Algorithm~\ref{alg:robust_mean_estimation}) in this paper is handling the dependence across iterations. 

Because we will run \(\mathsf{RobustMeanEstimation}\) in each iteration of our optimization algorithm, the gradients \(\{\grad f_{i}(x_t)\}_{i = 1}^{n}\) and Hessians \(\{\grad^{2} f_{i}(x_t)\}_{i = 1}^{n}\) can no longer be considered as independently drawn from a distribution after the first iteration. Although they are i.i.d.\ for fixed \(x\), the dependence on previous iterations through \(x\) will break the independence assumption. Therefore, we will need a union bound over all \(x_{t}\) to handle dependence across different iterations \(t\). We deal with this technicality in Appendix~\ref{sec:general_nonconvex-appendix}.


\section{General Robust Nonconvex Optimization}
\label{sec:general_nonconvex}
In this section, we establish a general result that uses  Algorithm~\ref{alg:inexact_randomized} to obtain approximate SOSPs in the presence of outliers under strong contamination. The inexact gradient and inexact Hessian oracles are constructed with the robust mean estimation subroutine (Algorithm~\ref{alg:robust_mean_estimation}).
We consider stochastic optimization tasks in Definition~\ref{def:robust-stochastic-opt} satisfying Assumption~\ref{assump:general_clean_data}. 
We construct the inexact gradient and Hessian oracle required by Algorithm~\ref{alg:inexact_randomized} as follows:
\begin{align*}
  \tilde g_{t} & \la \mathbf{RobustMeanEstimation}(\{\grad f_{i}(x_{t})\}_{i=1}^{n}, 4\epsilon) \\
  \tilde H_{t} & \la \mathbf{RobustMeanEstimation}(\{\grad^{2} f_{i}(x_{t})\}_{i=1}^{n}, 4\epsilon)
\end{align*}

Then we have the following guarantee:
\begin{theorem}
  \label{thm:global_convergence-general}
  Suppose we are given \(\epsilon\)-corrupted set of functions  \(\{f_{i}\}_{i = 1}^{n}\) for sample size $n$, generated according to Definition~\ref{def:robust-stochastic-opt}.
  Suppose Assumption~\ref{assump:general_clean_data} holds in a bounded region \(\B \subset \R^D\) of diameter \(\gamma\) with gradient and Hessian covariance bound $\sigma_g$ and $\sigma_H$ respectively, and we have an arbitrary initialization \(x_{0} \in \B\).
  Algorithm~\ref{alg:inexact_randomized} initialized at \(x_{0}\) outputs an \((\epsilon_{g}, \epsilon_{H})\)-approximate SOSP for a sufficiently large sample with  probability at least \(1-\xi\) if the following conditions hold:
  \begin{enumerate}[(I)]
    \item All iterates \(x_{t}\) in Algorithm~\ref{alg:inexact_randomized} stay inside the bounded region \(\B\).
    \item For an absolute constant \(c > 0\), it holds that \(\sigma_{g}\sqrt\epsilon \le c \epsilon_{g}\) and \(\sigma_{H}\sqrt\epsilon \le c\epsilon_{H}\).
  \end{enumerate}
  The algorithm uses \(n = \bigtO[\big]{D^{2}/\eps}\) samples, where \(\bigtO{\cdot}\) hides logarithmic dependence on \(D, \eps, L_{D_g}, L_{D_H}, B_{D_g}, B_{D_H}, \gamma/\sigma_{H}, \gamma/\sigma_{g},\) and \(1/\xi \). The algorithm runs in time polynomial in the above parameters.
\end{theorem}
Note that we are able to obtain dimension-independent errors $\epsilon_g$ and $\epsilon_H$, provided that $\sigma_{g}$ and $\sigma_{H}$ are dimension-independent.

\subsection{Low Rank Matrix Sensing Problems}
\label{sec:matrix-low-rank}
In this section, we study the problem of outlier-robust low rank matrix sensing as formally defined in Definition~\ref{def:iid_Gaussian_with_noise}. We first apply the above framework to obtain an approximate SOSP in Section~\ref{sec:global_convergence}. Then we make use of the approximate SOSP to obtain a solution that is close to the ground-truth matrix $M^*$ in Section~\ref{sec:local_linear_convergence}; this demonstrates the usefulness of approximate SOSPs.

\subsubsection{Main results for Robust Low Rank Matrix Sensing}
The following are the main results that we obtain in this section:
\begin{theorem}[Main Theorem Under Noiseless Measurements]
\label{thm:together}
   Consider the noiseless setting as in Theorem~\ref{thm:low-rank-mtrx-sensing-intro} with $\sigma = 0$. 
   For some sample size \(n = \bigtO[\big]{{(d^{2}r^{2} + d r \log(\Gamma/\xi))}/{\epsilon}}\) and with probability at least \(1-\xi\), there exists an algorithm that outputs a solution that is \(\iota\)-close to \(M^{*}\) in Frobenius norm in $\bigO[\big]{r^{2}\kappa^{3}\log(1/\xi) + \kappa \log(\sigstarr/\iota)} $ calls to the robust mean estimation subroutine (Algorithm~\ref{alg:robust_mean_estimation}).
\end{theorem}

This result achieves exact recovery of \(M^{*}\), despite the strong
contamination of samples. Each iteration involves a subroutine call to robust
mean estimation. Algorithm~\ref{alg:robust_mean_estimation} presented here is
one simple example of robust mean estimation; there are refinements~\cite{diakonikolas22streaming, dong2019quantum}
that run in nearly linear time, so the total computation utilizing those more efficient algorithms indeed requires \(\tilde O\left( r^{2}\kappa^{3} \right)\) passes of data (computed gradients and Hessians).
\begin{theorem}[Main Theorem Under Noisy Measurements] \label{thm:noisy_combined}
 Consider the same setting as in Theorem~\ref{thm:low-rank-mtrx-sensing-intro} with $\sigma \neq 0$. There exists a sample size \(n = \bigtO[\big]{{(d^{2}r^{2} + dr \log(\Gamma/\xi))}/{\epsilon}}\) such that
\begin{itemize}[leftmargin=*]
  \item if \(\sigma \le r \Gamma\), then with probability at least \(1 - \xi\), there exists an algorithm that outputs a solution \(\hat M\) in $ \tilde O( r^{2}\kappa^{3} )$ calls to robust mean estimation routine~\ref{alg:robust_mean_estimation}, with error \(\fnorm[\big]{\hat M - M^{*}} = \bigO{\kappa \sigma\sqrt\epsilon}\);
  \item if  \(\sigma \ge  r \Gamma\), then with probability at least \(1-\xi\), there exists a (different) algorithm that outputs a solution \(\hat M\) in one call to robust mean estimation routine~\ref{alg:robust_mean_estimation}, with error \(\fnorm[\big]{\hat M - M^{*}} = \bigO{\sigma\sqrt\epsilon}\).
\end{itemize}
\end{theorem}
We prove Theorem~\ref{thm:noisy_combined} in Appendix~\ref{sec:sensing_noise}, and instead focus on the noiseless measurements with $\sigma=0$ when we develop our algorithms in this section; the two share many common techniques.
In the remaining part of Section~\ref{sec:matrix-low-rank}, we use $\G_0$ in Definition~\ref{def:iid_Gaussian_with_noise} for the data generation process.

We now describe how we obtain the solution via nonconvex optimization. Consider the following objective function for (uncorrupted) matrix sensing:
\begin{equation}
  \label{eq:obj_m}
  \min_{\substack{M \in \R^{d \times d} \\ \rank(M) = r}} \frac{1}{2} \E_{(A_{i}, y_{i}) \sim \G_0}(\< M, A_{i} \> - y_{i})^{2}.
\end{equation}

We can write \(M = U U^{\top}\) for some \(U \in \R^{d \times r}\) to reparameterize the objective function. 
Let
\begin{equation}
  \label{eq:obj_u_i}
  f_{i}(U) := \frac12 \left(\< U U^{\top}, A_{i} \> - y_{i}\right)^{2}. 
\end{equation}
We can compute
\begin{equation}
    \label{eq:barf}
    \bar f(U) := \E_{(A_{i}, y_{i}) \sim \G_0} f_{i}(U) = \frac12 \Var\<UU^{\top} - M^{*}, A_{i}\>= \frac12\fnorm{ UU^{\top} - M^{*} }^{2}.
\end{equation}
We seek to solve the following optimization problem under the corruption model in Definition~\ref{def:robust-stochastic-opt}: 
\begin{equation}
  \label{eq:obj_u}
  \min_{U \in \R^{d \times r}}  \bar f(U).
\end{equation}
The gradient Lipschitz constant and Hessian Lipschitz constant of \(\bar f\) are given by the following result.
\begin{fact}[\cite{jin2017howto}, Lemma 6]
  \label{fact:matrix_factorization_constants}
  For any \(\Gamma > \sigstarl\), \(\bar f(U)\) has gradient Lipschitz constant \(L_{g} = 16\Gamma\) and Hessian Lipschitz constant \(L_{H} = 24\Gamma^{\frac12}\) inside the region \(\{U: \opnorm{U}^{2} < \Gamma\}\).
\end{fact}

\subsubsection{Global Convergence to an Approximate SOSP}
\label{sec:global_convergence}
In this section, we apply our framework
Theorem~\ref{thm:global_convergence-general} to obtain global convergence from an arbitrary initialization to an approximate SOSP, by providing problem-specific analysis to guarantee that both Assumption~\ref{assump:general_clean_data} and algorithmic assumptions (I) and (II) required by Theorem~\ref{thm:global_convergence-general} are satisfied.
\begin{theorem}[Global Convergence to a SOSP]
  \label{thm:global_convergence}
   Consider the noiseless setting as in Theorem~\ref{thm:low-rank-mtrx-sensing-intro} with $\sigma = 0$ and \(\epsilon = \bigO[\big]{1/{(\kappa^{3}r^{3})}}\).
   Assume we have an arbitrary initialization \(U_{0}\) inside \(\{U: \opnorm{U}^{2} \le \Gamma\}\). 
   There exists a sample size \(n = \tilde O\left({(d^{2}r^{2} + d r \log(\Gamma/\xi))}/{\epsilon}\right)\) such that with  probability at least \(1-\xi\), Algorithm~\ref{alg:inexact_randomized} initialized at \(U_{0}\) outputs a \((\frac{1}{24}\sigstarr^{3/2}, \frac{1}{3}\sigstarr)\)-approximate SOSP using at most \(\bigO{r^{2}\kappa^{3} \log(1/\xi)}\) calls to robust mean estimation subroutine (Algorithm~\ref{alg:robust_mean_estimation}).
\end{theorem}

\begin{proof}[Proof of Theorem~\ref{thm:global_convergence}]
To apply Theorem~\ref{thm:global_convergence-general}, we verify Assumption~\ref{assump:general_clean_data} first. To verify (i), for all $U$ and $A_i$, $f_{i}(U) = \frac12 \left(\< U U^{\top}, A_{i} \> - y_{i}\right)^{2} \ge 0$ , so $\flb = 0$ is a uniform lower bound. We verify (ii) in Appendix~\ref{sec:global-convergence-appendix}: conceptually, by Fact~\ref{fact:matrix_factorization_constants}, $\bar f$ is gradient and Hessian Lipschitz; both gradient and Hessian of $f_i$ are sub-exponential and concentrate around those of $\bar f$. To check (iii), we calculate the gradients and Hessians of \(f_{i}\) in Appendix~\ref{sec:sample_grad_and_hessians} and bound their covariances from above in Appendix~\ref{sec:grad_cov_bound} and~\ref{sec:Hessian_cov_bound}. The result is summarized in the following lemma. Note that the domain of the target function in Algorithm~\ref{alg:inexact_randomized} and Theorem~\ref{thm:global_convergence-general} is the Euclidean space \(\R^{D}\), so we vectorize \(U\) and let \(D = d r\). The gradient becomes a vector in \(\R^{dr}\) and the Hessian becomes a matrix in \(\R^{dr \times dr}\).


\begin{lemma}[Gradient and Hessian Covariance Bounds]
  \label{lemma:gradient_hessian_covariance_bound}
  For all \(U \in \R^{d \times r}\) with \(\opnorm{U}^{2} \le \Gamma\) and $f_i$ defined in Equation~\eqref{eq:obj_u_i}, it holds 
  \begin{align}
    \label{eq:gradient_covariance_bound}
    \opnorm{\Cov(\vect( \grad f_{i}(U)))} & \le 8 \fnorm{U U^{\top} - M^{*}}^{2} \opnorm{U}^{2} \le 32 r^{2}\Gamma^{3} \\
    \label{eq:hessian_covariance_bound}
    \opnorm{\Cov(\vect(H_{i}))} & \le  16r \fnorm{ U U ^{\top} - M^{*}}^{2}  + 128 \opnorm{U}^{4} \le 192 r^{3} \Gamma^{2}
  \end{align}
\end{lemma}





We proceed to verify the algorithmic assumptions in Theorem~\ref{thm:global_convergence-general}. For the assumption (I), we prove the following Lemma in  Appendix~\ref{sec:global-convergence-appendix} to show that all iterates stay inside the bounded region in which we compute the covariance bounds.
\begin{lemma}
  \label{lemma:radius}
  All iterates of Algorithm~\ref{alg:inexact_randomized} stay inside the region \(\{U: \opnorm{U}^{2} \le \Gamma\}\).
\end{lemma}

To verify Theorem~\ref{thm:global_convergence-general} (II), we let \(\epsilon_{g} = \frac{1}{32}\sigstarr^{3/2}, \epsilon_{H} = \frac{1}{4}\sigstarr\) and $\sigma_g = 8 r \Gamma^{1.5}, \sigma_H =16 r^{1.5} \Gamma $. So if we assume \(\epsilon = O(1/{(\kappa^{3}r^{3}}))\), then for the absolute constant $c$ in Theorem~\ref{thm:global_convergence-general} it holds that
\begin{align*}
  \sigma_g\sqrt\epsilon  \le c {\epsilon_{g}} \qquad
  \sigma_H\sqrt\epsilon  \le c {\epsilon_{H}}.
\end{align*}

Hence, Theorem~\ref{thm:global_convergence-general} applies and Algorithm~\ref{alg:inexact_randomized} outputs an $(\epsilon_g, \epsilon_H)$-approximate SOSP 
with high probability in polynomial time. To bound the runtime, since \(\bar f(U_0) = 1/2\fnorm{U_0 U_0^{\top} - M^{*}}^{2} = O(r^{2}\Gamma^{2}) \) for an arbitrary initialization \(U_{0}\) with \(\opnorm{U_{0}}^{2} <  \Gamma\), the initial distance can be bounded by \(O(r^{2}\Gamma^{2})\). Setting \(L_{g} = 16\Gamma, L_{H} = 24\Gamma^{1/2}, \bar f(U_{0}) = O(r^{2}\Gamma^{2}), \flb = 0\) and thus \(C_{\epsilon} = \bigO[\big]{{\sigstarr^{3}}/{\Gamma}}\), Proposition~\ref{prop:optimization_inexact_derivatives_general} implies that Algorithm~\ref{alg:inexact_randomized} outputs a \((\frac{1}{24}\sigstarr^{3/2}, \frac{1}{3}\sigstarr)\)-approximate second order stationary point \(U_{SOSP}\) in  \(\bigO[\big]{r^{2}\kappa^{3}\log(1/\xi)}\) steps with high probability.
\end{proof}

\subsubsection{Local Linear Convergence}\label{sec:local_linear_convergence}
In this section, we describe a local search algorithm that takes a \((\frac{1}{24}\sigstarr^{3/2}, \frac{1}{3}\sigstarr)\)-approximate second-order stationary point as its initialization and achieves exact recovery even in the presence of outliers. 

\begin{algorithm}
  \caption{Local Inexact Gradient Descent}\label{alg:local_linear_convergence}
  \KwData{The initialization \(U_{SOSP}\) is a \((\frac{1}{24}\sigstarr^{3/2}, \frac{1}{3}\sigstarr)\)-approximate SOSP, corruption fraction is \(\epsilon\), corrupted samples are \(\{(A_{i}, y_{i})\}_{i = 1}^{n}\), target distance to optima is \(\iota\)}
  \KwResult{\(U\) that is \(\iota\)-close in Frobenius norm to some global minimum}
  \(\eta = 1/\Gamma\), 
  \(U_{0} = U_{SOSP}\)\\
  \For{t = 0, 1, \dots}{
    \(\tilde g_{t} := \mathbf{RobustMeanEstimation}(\{\grad f_{i}(U_{t})\}_{i=1}^{n}, 4\epsilon)\) \label{line:robust_mean_estimation_gradient}\\
    \(U_{t+1} \leftarrow U_{t} - \eta \tilde g_{t}\)
  }
\end{algorithm}

\begin{theorem}[Local Linear Convergence]
  \label{thm:local_linear_convergence}
    Consider the same noiseless setting as in Theorem~\ref{thm:low-rank-mtrx-sensing-intro}.
    Assume we already found a \((\frac{1}{24}\sigstarr^{3/2}, \frac{1}{3}\sigstarr)\)-approximate SOSP \(U_{SOSP}\) of $\bar f$. Then there exists a sample size \(n = \tilde O \left({d r \log(1/\xi)}/{\epsilon}\right)\) such that with probability at least \(1-\xi\), Algorithm~\ref{alg:local_linear_convergence} initialized at \(U_{SOSP}\) outputs a solution that is \(\iota\)-close to
 some global minimum in Frobenius norm after \(\bigO[\big]{\kappa\log( \sigstarr/{\iota} )}\) calls to robust mean estimation subroutine (Algorithm~\ref{alg:robust_mean_estimation}). Moreover, all iterates \(U_{t}\) are \(\frac13 \sigstarr^{1/2}\)-close to some global minimum in Frobenius norm.
\end{theorem}

\begin{proof}[Proof Sketch]
  First we use known properties of \(\bar f = \E_{(A_{i}, y_{i}) \sim \G_0} f_{i}\) from the literature~\cite{jin2017howto} to show approximate SOSPs of $\bar f$ --- in particular our initialization $U_{SOSP}$ --- are in a small neighborhood of the global minima of $\bar f$. In that neighborhood, it was also known that \(\bar f\) satisfies some local regularity conditions that enable gradient descent's linear convergence. 
  
  However, the algorithm only has access to the {\em inexact} gradient from the robust mean estimation subroutine (line~\ref{line:robust_mean_estimation_gradient} in Algorithm~\ref{alg:local_linear_convergence}) and, therefore, we need to establish linear convergence of {\em inexact gradient descent}. We achieve this with the following iterative argument: As the iterate gets closer to global optima, the covariance bound of sample gradient in Equation~\eqref{eq:gradient_covariance_bound} gets closer to 0. Because the accuracy of the robust mean estimation scales with the covariance bound (see Proposition~\ref{prop:robust-mean-estimation-summary}), a more accurate estimate of population gradient \(\grad \bar f\) can be obtained via the robust estimation subroutine (Algorithm~\ref{alg:robust_mean_estimation}). This, in turn, allows for an improved inexact gradient descent step, driving the algorithm towards an iterate that is even closer to global optima.

  See Appendix~\ref{sec:local-linear-convergence-appendix} for the complete proof.
\end{proof}

\section{Statistical Query Lower Bound for Low Rank Matrix Sensing}
\label{sec:sq-lower-bound}

Our general algorithm (Theorem~\ref{thm:global_convergence-general}) leads
to an efficient algorithm for robust low rank matrix sensing with sample complexity
$O(d^2r^2/\eps)$ (Theorem~\ref{thm:noisy_combined}). Interestingly, the sample
complexity of the underlying robust estimation problem --- ignoring computational
considerations --- is $\Theta(dr/\eps)$. The information-theoretic upper bound 
of $O(dr/\eps)$ can be achieved by an exponential (in the dimension) 
time algorithm (generalizing the Tukey median to our regression setting); 
see, e.g., Theorem 3.5 in~\cite{gao2017robust}.

Given this discrepancy, it is natural to ask whether the sample complexity achieved by 
our algorithm can be improved via a different computationally efficient method. In this section, we provide evidence that this may not be possible. In more detail, 
we establish a near-optimal information-computation tradeoff
for the problem, within the class of Statistical Query (SQ) algorithms. 
To formally state our lower bound, we require basic background on
SQ algorithms. 

\paragraph{Basics on SQ Model.} 
SQ algorithms are a class of algorithms
that, instead of access to samples from some distribution $\mathcal{P}$, 
are allowed to query expectations of bounded functions over $\mathcal{P}$. 

\begin{definition}[SQ Algorithms and \(\mathsf{STAT}\) Oracle~\cite{Kearns:98}]
Let \(\mathcal{P}\) be a distribution on \(\R^{d^{2} + 1}\).  A Statistical Query (SQ) is a bounded function \(q: \R^{d^{2} + 1} \ra [-1, 1]\). For \(\tau > 0\), the \(\mathsf{STAT}(\tau)\) oracle responds to the query \(q\) with a value \(v\) such that \(\abs{v - \E_{X \sim \mathcal{P}}[q(X)]} \le \tau\). An SQ algorithm is an algorithm whose objective is to learn some information about an unknown distribution \(\mathcal{P}\) by making adaptive calls to the corresponding \(\mathsf{STAT}(\tau)\) oracle.
\end{definition}

In this section, we consider $\mathcal{P}$ as the unknown corrupted distribution where $(A_i, y_i)$ are drawn. The SQ algorithm tries to learn the ground truth matrix $M^*$ from this corrupted distribution; the goal of the lower bound result is to show that this is hard.

The SQ model has the capability to implement a diverse set of algorithmic techniques in 
machine learning such as spectral techniques, moment and tensor methods, local search (e.g., 
Expectation Maximization), and several others~\cite{FeldmanGRVX17}. 
A lower bound on the SQ complexity of a problem 
provides evidence of hardness for the problem. \cite{brennan2020statistical} established that (under certain assumptions) 
an SQ lower bound also implies a qualitatively similar lower bound 
in the low-degree polynomial testing model. 
This connection can be used to show a similar lower bound for low-degree polynomials.  

Our main result here is a near-optimal SQ lower bound for 
robust low rank matrix sensing that applies even for rank $r=1$,
i.e., when the ground truth matrix is 
\(M^{*} = u u^{\top}\) for some \(u \in \R^{d}\).
\new{The choice of rank $r = 1$ yields the strongest possible lower bound in our setting because it is the easiest parameter regime: Recall that the sample complexity of our algorithm is $\widetilde O(d^2 r^2)$ as in Theorems~\ref{thm:together} and~\ref{thm:noisy_combined}, and the main message of our SQ lower bound is to provide evidence that the $d^2$ factor is necessary for computationally efficient algorithms \emph{even if} $r = 1$.}


\begin{theorem}[SQ Lower Bound for Robust Rank-One Matrix Sensing]
  \label{thm:sq_lower_bound}
  Let \(\epsilon \in (0,1/2)\) be the fraction of corruptions and let \(c \in (0,1/2)\). 
  Assume the dimension \(d \in \N\) is sufficiently large. 
  Consider the \(\epsilon\)-corrupted rank-one matrix sensing problem with ground-truth matrix \(M^{*} = u u^{\top}\) and noise \(\sigma^{2} = O(1)\).
  Any SQ algorithm that outputs \(\hat u\) with \(\norm{\hat u - u} = O(\epsilon^{1/4})\) either requires \(2^{\Omega(d^{c})}/ d^{2-4c}\)  queries or 
  makes at least one query to 
  $\mathsf{STAT}\Big(e^{\bigO{1/\sqrt\epsilon}}/\bigO{d^{1-2c}}\Big)$. 
\end{theorem}

In other words, we show that, when provided with SQ access to an $\eps$-corrupted distribution, approximating \(u\) is impossible unless employing a statistical query of higher precision
than what can be achieved with a strictly sub-quadratic number (e.g.,~\(d^{1.99}\)) of samples. 
Note that the SQ oracle \(\mathsf{STAT}\bigl(e^{\bigO[\big]{1/\sqrt\epsilon}}/\bigO[big]{d^{1-2c}}\bigr)\) 
can be simulated with \(O(d^{2 - 4c})/e^{O(1/\epsilon)}\) samples, and this bound is tight in general. 
Informally speaking, this theorem implies that improving the sample complexity 
from \(d^{2}\) to \(d^{2 - 4c}\) requires exponentially many queries. 
This result can be viewed as a near-optimal information-computation tradeoff
for the problem, within the class of SQ algorithms. 

The proof follows a similar analysis as in~\cite{diakonikolas2019efficient-algor, diakonikolas2021statistical}, using one-dimensional moment matching to construct a family of corrupted conditional distributions, which induce a family of corrupted joint distributions that are SQ-hard to learn. We provide the details of the proof in Appendix~\ref{sec:sq-lower-bound-appendix}. \new{Apart from the formal proof, in Appendix~\ref{sec:count-why-simple} we also informally discuss the intuition for why some simple algorithms that require \(O(d)\) samples do not provide dimension-independent error guarantees.}

\section*{Acknowledgements}
\label{sec:ack}
Shuyao Li was supported in part by NSF Awards DMS-2023239, NSF Award CCF-2007757 and the U.\ S.\ Office of Naval Research under award number N00014-22-1-2348.
Yu Cheng was supported in part by NSF Award CCF-2307106. 
Ilias Diakonikolas was supported in part by NSF Medium Award CCF-2107079, NSF Award CCF-1652862 (CAREER), a Sloan Research Fellowship, and a DARPA Learning with Less Labels (LwLL) grant.
Jelena Diakonikolas was supported in part by NSF Award CCF-2007757 and by the U.\ S.\ Office of Naval Research under award number N00014-22-1-2348.
Rong Ge was supported in part by NSF Award DMS-2031849, CCF-1845171 (CAREER) and a Sloan Research Fellowship.
Stephen Wright was supported in part by NSF Awards DMS-2023239 and CCF-2224213 and AFOSR via subcontract UTA20-001224 from UT-Austin. 

\newpage

\printbibliography

\newpage

\appendix

\section*{Supplementary Material}
Supplementary material is organized as follows. In Appendix~\ref{sec:prelim-appendix}, we provide useful auxiliary facts and relevant technical results from previous works. Appendix~\ref{sec:general_nonconvex-appendix} proves our result for general robust nonconvex optimization (Theorem~\ref{thm:global_convergence-general}). Appendix~\ref{sec:low-rank-appendix} provides omitted computation and proofs for robust low rank matrix sensing (Section~\ref{sec:matrix-low-rank}). Appendix~\ref{sec:sq-lower-bound-appendix} proves our SQ lower bound (Section~\ref{sec:sq-lower-bound}) for the sample complexity of efficient algorithms for the outlier-robust low rank matrix sensing problem, \new{and Appendix~\ref{sec:count-why-simple} discusses the intuition why some simple algorithms that violate our SQ lower bound fail.}

\section{Technical Preliminaries}\label{sec:prelim-appendix}
\subsection{Notation and Auxiliary Facts}\label{sec:notation-appendix}
\new{In the appendix, we use \(\bigtO{\cdot}\) notation to suppress, for conciseness, logarithmic dependences on all defined quantities, even if they do not appear inside \(\bigtO{\cdot}\). More precise statements can be found in the corresponding main parts of the paper.
}
For matrix space $\R^{d \times r}$ and a set $\cXstar \subset \R^{d \times r}$, we use $\projX(\cdot)$ to denote the Frobenius projection onto \(\cXstar\), i.e.,\ $\projX(U) = \argmin_{Z \in \cXstar} \fnorm{U - Z}^2 $. We use $\dist(U, \cXstar)$ to denote $\fnorm{U - \projX(U)}$.

\begin{fact}\label{fact:fnorm}
   For matrices \(A, B\) with compatible dimensions,
   \begin{align*}
     \fnorm{AB} &\le \opnorm{A} \fnorm{B}\\
     \fnorm{AB} &\le \fnorm{A} \opnorm{B}
   \end{align*}
 \end{fact}
For two matrices $A$ and $B$, let \(A \otimes B\)  denote the Kronecker product of \(A\) and \(B\).
\begin{fact}
  \label{fact:tensor}
  For matrices \(A, B, C\) with compatible dimensions,
  $\operatorname{vec}({A B C})=\left({C}^{\top} \otimes {A}\right) \operatorname{vec}({B})$.
\end{fact}
\begin{fact}
  \label{fact:kron_operator_norm}
  \(\opnorm{A \otimes B} = \opnorm{A}\opnorm{B}\)
\end{fact}

We will frequently use the following fact about the mean and the variance of the quadratic form for zero-mean multivariate normal distributions.

\begin{fact}
  \label{fact:Gaussian_quadratic_form}
  Let \(X \sim \mathcal N(0, \Sigma)\) be a \(k\)-dimensional random variable  and let \(G\) be a \(k \times k\) real matrix; then
  \begin{align*}
    \E[X^{\top} G X] &= \tr(G\Sigma) = \< G, \Sigma\> \\
    \Var[X^{\top} G X] &=  \tr(G\Sigma (G + G^{\top})\Sigma)
  \end{align*}
\end{fact}

\subsection{Simplified Proof for the Optimization Algorithm with Inexact Gradients and Hessians}
\label{sec:inexact-randomized-short-proof-appendix}

\begin{algorithm}[H]
\caption{Nonconvex minimization with inexact gradients and Hessians~\cite{li2023randomized}}\label{alg:inexact_randomized}
  \For{$t = 1, 2, \dots$}{
    Obtain \(\inexact g_{t}\) from the inexact gradient oracle \\

  \uIf{\(\|\inexact g_{t}\| > \epsilon_{g}\)} {
    \(x_{t+1} = x_{t} - \frac1L \inexact g_{t}\)
  }
  \Else{
    Obtain \(\inexact H_{t}\) from the inexact Hessian oracle \\
    Compute smallest eigenvalue and its corresponding eigenvector \((\inexact\lambda_{t}, \inexact p_{t})\) \\
    \uIf{\(\inexact\lambda_{t} < -\epsilon_{H}\)}{
    \(\sigma_{t} = \pm 1\) with probability \(\frac12\) \\
    \(x_{t+1} = x_{t} + \frac{2 \epsilon_{H}}{L_{H}} \sigma_t \inexact p_{t}\)
  }
  \Else{
    return \(x_{t}\)
  }
  }
}
\end{algorithm}
We stated the following guarantee (Proposition~\ref{prop:optimization_inexact_derivatives_general}) in Section~\ref{sec:prelims}:
\begin{proposition}
\label{prop:optimization_inexact_derivatives_general-appendix}
    Suppose a function \(f\) is bounded below by $f^* > -\infty,$ has $L_g$-Lipschitz gradient and $L_H$-Lipschitz Hessian, and its
  inexact gradient and Hessian computation \(\inexact g_{t}\) and \(\inexact H_{t}\) satisfy
  \(\norm{\inexact g_{t} - \grad{f}(x_{t})} \le \frac{1}{3} \epsilon_{g}\) and \(\opnorm[\big]{\inexact H_{t} - \grad^{2} f(x_{t})} \leq \frac29 \epsilon_{H}\).
  Then there exists an algorithm (Algorithm~\ref{alg:inexact_randomized}) with the following guarantees:
  \begin{enumerate}[leftmargin=*]
    \item  (Correctness) If Algorithm~\ref{alg:inexact_randomized} terminates and outputs \(x_{n}\), then \(x_{n}\) is a
    \((\frac43\epsilon_g, \frac{4}{3}\epsilon_{H})\)-approximate second-order
    stationary point.
    \item (Runtime) Algorithm~\ref{alg:inexact_randomized} terminates with probability $1$. Let
    \(C_{\epsilon}:= \min\left(\frac{\epsilon_{g}^{2}}{6L_{g}},\frac{2\epsilon_{H}^{3}}{9L_{H}^{2}}\right)\). With probability at least \(1-\delta\),
    Algorithm~\ref{alg:inexact_randomized} terminates after $k$ iterations for
    \begin{equation*}
       k = \bigO[\Big]{ \frac{f(x_{0}) - \flb}{C_{\epsilon}} +  \frac{L_{H}^{2} L_{g}^{2} \epsilon_{g}^{2}}{\epsilon_{H}^{6}} \log\Bigl(\frac1\delta \Bigr)} \; .
    \end{equation*}
    \end{enumerate}
  \end{proposition}

  This section proves the correctness and a slightly weaker runtime guarantee in terms of the dependence on \(\delta\). We prove a \(O(1/\delta)\) dependence instead of \(O(\log(1/\delta))\), i.e.,
  with probability at least \(1-\delta\),
    Algorithm~\ref{alg:inexact_randomized} terminates after $k$ iterations for
    \begin{equation}
      \label{eq:high_probability_bound_general-appendix}
       k = \bigO[\Big]{ \frac{f(x_{0}) - \flb}{\delta C_{\epsilon}}}\;.
    \end{equation}

    \begin{proof}[Proof of correctness and a weaker runtime~\cite{li2023randomized}]
      We prove the correctness first. From the stopping criteria, we have
\(\norm{\inexact g_{k}} \le \epsilon_{g}\) and \(\inexact \lambda_{k} \ge -\epsilon_{H}\). Thus,
\[
  \norm{\grad f(x_{k})}
 \le \norm{g_{k}} + \frac13 \epsilon_{g} = \frac43 \epsilon_{g}.
\]
To bound \(\lambda_{\min}(\grad^{2} f(x_{k}))\), write \(U_{k} = \inexact H_{k} - \grad^{2} f(x_{k})\). By Hessian inexactness condition, \(\lambda_{\min}(-U_{k}) \ge -\opnorm{-U_{k}} \ge -\frac29 \epsilon_{H}\).

We use Weyl's theorem to conclude that
\[
  \lambda_{\min}(\grad^{2}f(x_{k}) ) \ge \lambda_{\min}(\inexact H_{k}) + \lambda_{\min}(-U_{k}) \ge -\epsilon_{H}-\frac29 \epsilon_{H} > -\frac{4}{3}\epsilon_{H},
\]
as required.

Now we analyze the runtime. We proceed by first establishing an expectation bound and then use Markov's inequality.

Write \(v_{t} :=  \inexact g_{t} - \grad f(x_{t})\). When the algorithm takes a gradient step, we have \(\| \inexact g_{t}\| > \epsilon_{g} \).
By gradient inexactness condition, we have \(\|v_{t}\| \le \frac13\| \inexact g_{t}\|\), so that  \(\|\grad f(x_{t}) \| \ge \frac23 \| \inexact g_{t}\|\).
From Taylor's Theorem, we have
\begin{align}
  f(x_{t+1}) & = f\left(x_{t} - \frac1{L_{g}}  \inexact g_t\right)  \nonumber \\
             & \leq f(x_{t}) - \frac1{L_{g}} \grad f(x_{t})^{\top} \inexact g_t + \frac{L_{g}}2 \cdot \frac1{L_{g}^{2}} \| \inexact g_t\|^{2} \nonumber \\
             & = f(x_{t}) - \frac1{2L_{g}}\left(\|\grad f(x_{t}) \|^{2} - \| v_{t}\|^{2} \right) \nonumber \\
             & \le f(x_{t}) - \frac1{2L_{g}} \left( \frac49\| \inexact g_{t}\|^{2} - \frac19\| \inexact g_{t}\|^{2} \right) \nonumber \\
             & \le f(x_{t}) - \frac1{6L_{g}}\| \inexact g_{t}\|^{2} \nonumber \\
             & \le f(x_{t}) - \frac1{6L_{g}}\epsilon_{g}^{2}, \label{eq:first-order-decrease}
\end{align}

For the negative curvature step, recall that \(\opnorm[\big]{\inexact H_{t} - \grad^{2} f(x_{t})} \leq \frac29 \epsilon_{H}\). It follows from Taylor's theorem that
\begin{align*}
  f(x_{t+1})
  &= f(x_{t} + \frac{2\epsilon_H}{L_{H}} \sigma_t \inexact{p}_{t})\\
  & \le f(x_{t}) + 2\frac{\epsilon_H}{L_{H}}\grad f(x_{t})^{\top}\sigma_t \inexact{p}_{t} + \frac12\frac{4\epsilon_H^{2}}{L_{H}^{2}} \inexact{p}_{t}^{\top} \grad^{2}f(x_{t})\inexact{p}_{t} + \frac{L_{H}}6  \frac{8\epsilon_H^{3}}{L_{H}^{3}} \\
  & = f(x_{t}) + \frac{2\epsilon_H^{2}}{L_{H}^{2}}\left(\inexact{p}_{t}^{\top} \grad^{2}f(x_{t})\inexact{p}_{t} + \frac23 \epsilon_H \right) + 2\frac{\epsilon_H}{L_{H}}\grad f(x_{t})^{\top}\sigma_t \inexact{p}_{t}
\end{align*}
When the algorithm takes a negative curvature step,  we have \(\inexact\lambda_{t} < -\epsilon_H < 0 \), so by Hessian inexactness condition,  we have
\(\| \inexact H_{t} - \grad^{2} f(x_{t})\| \le \frac29 \epsilon_{H}  \le \frac29 |\inexact\lambda_{t}|  \). It follows from the definition of operator norm and Cauchy-Schwarz that
\(|\inexact{p}_{t}^{\top} \inexact H_{t}\inexact{p}_{t} - \inexact{p}_{t}^{\top}\grad^{2} f(x_{t})\inexact{p}_{t}| \le \| \inexact H_{t}\inexact{p}_{t} - \grad^{2} f(x_{t})\inexact{p}_{t}\|\le \frac{2|\inexact\lambda_{t}|}9\), so
\[
\inexact{p}_{t}^{\top}\grad^{2} f(x_{t})\inexact{p}_{t} \le \frac29|\inexact\lambda_{t}| + {p}_{t}^{\top} \inexact H_{t}\inexact{p}_{t} = \frac29 |\inexact\lambda_{t}| + (-|\inexact\lambda_{t}|) = -\frac79 |\inexact\lambda_{t}|.
\]
We thus have
\begin{align}
  f(x_{t+1})
  & \le f(x_{t}) + \frac{2\epsilon_H^{2}}{L_H^{2}}\left(\inexact{p}_{t}^{\top} \grad^{2}f(x_{t})\inexact{p}_{t} + \frac23 \epsilon_H \right) + 2\frac{\epsilon_H}{L_H}\grad f(x_{t})^{\top}\sigma_t \inexact{p}_{t}  \nonumber \\
  & \le f(x_{t}) + \frac{2\epsilon_H^{2}}{L_H^{2}}\left(-\frac79 |\inexact\lambda_{t}| + \frac23 \alpha_{t}\right) + 2\frac{\epsilon_H}{L_H}\grad f(x_{t})^{\top}\sigma_t \inexact{p}_{t} \nonumber \\
  & \le  f(x_{t})  - \frac{2\epsilon_{H}^{3}}{9L_{H}^{2}} + 2\frac{\epsilon_H}{L_H}\grad f(x_{t})^{\top}\sigma_t \inexact{p}_{t}. \label{eq:NC-second-order-decrease}
\end{align}

We have the following result for expected stopping time of Algorithm~\ref{alg:inexact_randomized}.
Here the expectation is taken with respect to the random variables $\sigma_t$ used at the negative curvature iterations.
For purposes of this and later results, we define
\begin{equation} \label{eq:Ceps}
C_{\epsilon}:= \min\left(\frac{\epsilon_{g}^{2}}{6L_{g}},\frac{2\epsilon_{H}^{3}}{9L_{H}^{2}}\right).
\end{equation}

Assuming Lemma~\ref{lemma:expected_complexity}, the runtime guarantee in Proposition~\ref{prop:optimization_inexact_derivatives_general-appendix} follows from Markov inequality.
\end{proof}

\begin{lemma}
  \label{lemma:expected_complexity}
   Consider the same setting as in Proposition~\ref{prop:optimization_inexact_derivatives_general-appendix}.
    Let \(T\) denote the iteration at which  Algorithm~\ref{alg:inexact_randomized} terminates. Then \(T < \infty\) almost surely and
    \begin{equation}
      \label{eq:expectation_bound}
      \E T \le \frac{f(x_{0}) - \flb}{C_{\epsilon}},
    \end{equation}
    where $C_{\epsilon}$ is defined in \eqref{eq:Ceps}.
  \end{lemma}

  The proof of this result is given below.
  It constructs a supermartingale\footnote{A supermartingale with respect to filtration \(\{\G_{1}, \G_{2}, \dotsc\}\) is a sequence of random variables $\{Y_1,Y_2, \dotsc \}$ such that for all \(k \in \Z_{+}\), (i) \(\E \abs{Y_{t}} < \infty\), (ii) \(Y_{t}\) is \(\G_{t}\)-measurable, and (iii) $\E (Y_{t+1} \, | \G_{t})  \le Y_t$.} based on the function value and uses a supermartingale convergence theorem and optional stopping theorem to obtain the final result.
  A similar proof technique is used in \cite{bergou2021a-subsampling-l} but for a line-search algorithm.
  We collect several relevant facts about supermartingales before proving the result.

First, we need to ensure the relevant supermartingale is well defined even after the algorithm terminates, so that it is possible to let the index \(t\) of the supermartingale go to \(\infty\).
  \begin{fact}[{\cite[Theorem 4.2.9]{durrett2019probability}}]
    \label{fact:stopped_martingale}
If \(T\) is a stopping time and \(X_{t}\) is a supermartingale, then \(X_{\min(T, t)}\) is a supermartingale.
\end{fact}
The following supermatingale convergence theorem will be used to ensure the function value converges, so that the algorithm terminates with probability 1.
\begin{fact}[Supermartingale Convergence Theorem, {\cite[Theorem 4.2.12]{durrett2019probability}}]
  \label{fact:martingale_convergence}
If \(X_{t} \ge 0\) is a supermartingale, then as \(t \ra \infty\), there exists a random variable \(X\) such that \(X_{t} \ra X \; a.s.\) and \(\E X \le \E X_{0}\).
\end{fact}

Finally, we will use the optional stopping theorem to derive the expected iteration complexity. Note that we use a version of the optional stopping theorem specific to nonnegative supermartingales that does not require uniform integrability.
\begin{fact}[Optional Stopping Theorem, {\cite[Theorem 4.8.4]{durrett2019probability}}]
  \label{fact:optional_stopping}
  If \(X_{n}\) is a nonnegative supermartingale and \(N \le \infty\) is a stopping time, then \(\E X_{0} \ge \E X_{N}\). 
\end{fact}

\begin{proof}[Proof of Lemma~\ref{lemma:expected_complexity}]
We first construct a supermartingale based on function values. Since \(\E \left[\sigma_t \right] = 0\), linearity of expectation implies that
\(\E \left[2\frac{\alpha_{t}}{M}\grad f(x_{t})^{\top}\sigma_t \inexact{p}_{t}\middle\vert x_{t}\right] = 0\).
We therefore have from \eqref{eq:NC-second-order-decrease} that
\begin{align*}
  \E\left[f(x_{t+1}) \middle\vert x_{t}\right]
  & \le f(x_{t})- \frac{2\epsilon_{H}^{3}}{9L_{H}^{2}}.
\end{align*}
By combining with the (deterministic) first-order decrease estimate \eqref{eq:first-order-decrease}, we have
\[\E\left[f(x_{t+1}) \middle\vert x_{t}\right] \le f(x_{t})- \min\left(\frac{\epsilon_{g}^{2}}{6L_{g}},\frac{2\epsilon_{H}^{3}}{9 L_{H}^{2}}\right) = f(x_t) - C_{\epsilon}.   \]
Consider the stochastic process \(M_{t}:= f(x_{t}) + t C_{\epsilon}\). We have
\begin{align*}
  \E\left[M_{t+1}\middle\vert x_{t} \right]
 &= \E\left[  f(x_{t+1}) + (t+1) C_{\epsilon} \middle\vert x_{t} \right] \\
  & \le \E\left[  f(x_{t})- C_{\epsilon} + (t+1) C_{\epsilon} \middle\vert x_{t} \right]\\
  & =  \E\left[  f(x_{t}) + t C_{\epsilon} \middle\vert x_{t} \right] = M_{t}.
\end{align*}

We need to select a filtration to define the supermartingale \(M_{t}\).
We view \(x_{t}\) as random variables defined with respect to \(\sigma_{i}, i \le k\).
Since \(M_{t}\) is expressed as a function of \(x_{t}\) only, we define the filtration \(\{\G_{t}\}\) to be the filtration generated by \(x_{t}\), and it naturally holds that \(M_{t}\) is \(\G_{t}\)-measurable for all \(k\) and \(\E\left[M_{t+1}\middle\vert \G_{t} \right] = \E\left[M_{t+1}\middle\vert x_{t} \right]\).
Hence, \(\{M_{t}\}\) is a supermartingale with respect to filtration \(\{\G_{t}\}\).

Let \(T\) denote the iteration at which our algorithm stops.
Since the decision to stop at iteration \(t\) depends only on \(x_{t}\), we have \(\{T = t \} \in \G_{t}\), which implies \(T\) is a stopping time.

We will use the supermartingale convergence theorem (Fact~\ref{fact:martingale_convergence}) to show that \(T < +\infty\) almost surely, since the function value cannot decrease indefinitely as it is bounded by \(\flb\) from below. To apply Fact~\ref{fact:martingale_convergence}, we need to let \(t \ra \infty\), so we need to transform \(\{M_{t}\}\) to obtain a supermartingale \(\{Y_{t}\}\)that is well defined even after the algorithm terminates.

It follows from Fact~\ref{fact:stopped_martingale} that \(Y_{t} :=  M_{\min(t,T)}\) is also a supermartingale. Since \(Y_{t} \ge \flb\), it follows from the supermartingale convergence theorem (Fact~\ref{fact:martingale_convergence}) applied  to \(Y_{t} - \flb\) that \(Y_{t} \rightarrow Y_{\infty}\) almost surely for some random variable \(Y_{\infty}\) with \(\E Y_{\infty} \le \E Y_{0} = \E M_{0} = f(x_{0}) < \infty\). Hence \(\Prob[Y_{\infty} = +\infty] = 0\). On the other hand, as \(t \rightarrow \infty\), we have \(t C_{\epsilon}  \rightarrow \infty\), so \( T = +\infty \implies Y_{t} = M_{t} \ge \flb + t C_{\epsilon} \ra \infty \implies Y_{\infty} = +\infty\). Therefore we have \(\Prob[T < +\infty]= 1\).

We can then apply the optional stopping theorem (Fact~\ref{fact:optional_stopping}) to \(Y_{t} - \flb\). It follows that \[ \flb +  \E T \cdot C_{\epsilon}  \le \E f(x_{T}) +  \E T \cdot C_{\epsilon}  =\E\left[ M_{T} \right] = \E\left[ Y_{T} \right] \le \E [Y_{0}] = \E\left[ M_{0} \right] = f(x_{0}), \] where the first equality uses \(T < +\infty\) almost surely and the last inequality is Fact~\ref{fact:optional_stopping}.
By reorganizing this bound, we obtain \(\E T \le  \frac{f(x_{0}) - \flb}{C_{\epsilon}} \), as desired.
\end{proof}

\subsection{Robust Mean Estimation---Omitted background}\label{sec:robust-mean-estimation-appendix}
A class of algorithms that robustly estimate quantities in the presence of outliers under strong contamination model (Definition~\ref{def:corruption}) are based on the stability of samples:
\begin{definition}[Definition of Stability~\cite{diakonikolas2017being-robust-in}]\label{def:stability}
  Fix \(0 < \epsilon < 1/2\) and \(\delta > \epsilon\). A finite set
  \(S \subset \R^{k}\) is {\((\epsilon, \delta)\)-stable} with respect to mean
  \(\mu \in \R^{k}\) and \(\sigma^{2}\) if for every \(S'\subset S\) with
  \(|S'| \ge (1 - \epsilon) |S|\), the following conditions hold: (i)
  \(\norm{\mu_{S'} - \mu} \le \sigma\delta\), and (ii)
  \(\opnorm{\bar\Sigma_{S'} - \sigma^{2} I} \le \sigma^{2}\delta^{2}\epsilon\),
  where \(\mu_{S'}\) and \(\bar\Sigma_{S'}\) denote the empirical mean and
  empirical covariance over the set \(S'\) respectively.
\end{definition}

\begin{algorithm}[ht]
  \caption{\(\mathsf{RobustMeanEstimation}\) with unknown covariance bound}
  \label{alg:robust_mean_estimation}
  \KwData{\(0 < \epsilon < 1/2\) and \(T\) is an \(\epsilon\)-corrupted set}
  \KwResult{\(\hat \mu\) with \(\norm{\hat\mu - \mu_{S}} = \bigO[\big]{\sqrt{\opnorm{\Sigma} \epsilon}}\)}
  Initialize a weight function \(w : T \rightarrow \R_{\ge 0}\) with \(w(x)= 1/ |T|\) for all \(x \in T\)\;
  \While{\(\norm{w}_{1} \ge 1 - 2\epsilon\)}{
    \(\mu(w) := \frac1{\norm{w}_{1}}\sum_{x\in T}w(x) x\)\;
    \(\Sigma(w) :=  \frac1{\norm{w}_{1}}\sum_{x\in T}w(x)(x - \mu(w))(x-\mu(w))^{\top}\)\;
    Compute the largest eigenvector \(v\) of \(\Sigma(w)\) \;
    \(g(x) := |v^{\top}(x - \mu(w))|^{2}\)\;
    Find the largest threshold \(t\) such that \(\Sigma_{x\in T: g(x) \ge t}w(x) \ge \epsilon\)\;
    \(f(x) := g(x)\Ind{\{g(x) \ge t\}}\) \;
    \(w(x) \leftarrow w(x)\left( 1 - \frac{f(x)}{\max_{y \in T: w(y) \neq 0} f(y)}\right)\) \;
  }
  \Return\(\mu(w)\)
\end{algorithm}

On input the corrupted version of a stable set, Algorithm~\ref{alg:robust_mean_estimation} has the following guarantee.
\begin{proposition}[{Robust Mean Estimation with Stability~\cite[Theorem A.3]{diakonikolas2020outlier}}]
  \label{prop:robust_mean_estimation_with_stability}
  Let $T \subset \mathbb{R}^k$ be an $\epsilon$-corrupted version of a set $S$, where $S$ is $(C \epsilon, \delta)$-stable with respect to $\mu_S$ and $\sigma^2$, and where $C>0$ is a sufficiently large constant. Algorithm~\ref{alg:robust_mean_estimation} on input $\epsilon$ and  $T$ (but not $\sigma$ or $\delta$) returns (deterministically) a vector $\hat{\mu}$ in polynomial time so that $\norm{\mu_S-\hat{\mu}}=O(\sigma \delta)$.
\end{proposition}

Proposition~\ref{prop:robust_mean_estimation_with_stability} requires the uncorrupted samples to be stable. With a large enough sample size, most independent and identically distributed samples from a bounded covariance distribution are stable. The remaining samples can be treated as corruptions.
\begin{proposition}[{Sample Complexity for Stability~\cite[Theorem 1.4]{diakonikolas2020outlier}}]
  \label{prop:bounded_variance_stability_iid}
  Fix any $0<\xi<1$. Let $S$ be a multiset of $n$ independent and identically distributed samples from a distribution on $\mathbb{R}^k$ with mean $\mu$ and covariance $\Sigma$. Then, with probability at least $1-\xi$, there exists a sample size \(n = \bigO{\frac{k \log k + \log (1 / \xi)}{\epsilon}}\) and a subset $S^{\prime} \subseteq S$ such that $\left|S^{\prime}\right| \geq\left(1-\epsilon\right) n$ and $S^{\prime}$ is $\left(2 \epsilon, \delta\right)$-stable with respect to $\mu$ and $\opnorm{\Sigma}$, where \(\delta = \bigO{\sqrt\epsilon}\).
\end{proposition}

\begin{proof}[Proof of Proposition~\ref{prop:robust-mean-estimation-summary}]
  Proposition~\ref{prop:bounded_variance_stability_iid} implies that for i.i.d.\ samples from a \(k\)-dimensional bounded covariance distribution to contain a stable set of more than $(1-\epsilon)$-fraction of samples with high probability, we need \(\bigtO{\frac{k}{\epsilon}}\) samples. We refer to the remaining $\eps$-fraction of samples as unstable samples. Since the adversary corrupts an $\epsilon$-fraction of clean samples \(S\), the input  set \(T\) can be considered as a $2\epsilon$-corrupted version of a stable set, if we view the unstable samples as corruptions. Therefore, Proposition~\ref{prop:robust_mean_estimation_with_stability} applies and gives the desired error guarantee.
\end{proof}

\section{Omitted Proofs in General Robust Nonconvex Optimization}\label{sec:general_nonconvex-appendix}
\paragraph{Sample Size for Successful Robust Mean Estimation}
For each iteration \(t\) in Algorithm~\ref{alg:inexact_randomized} and Algorithm~\ref{alg:local_linear_convergence}, we would like to robustly estimate the mean of a set of corrupted points \(\{\vect{(\grad f(x_{t}))}\}_{i = 1}^{n}\) and/or \(\{\vect{(\grad^{2} f(x_{t}))}\}_{i = 1}^{n}\)  where the inliers are drawn from a distribution with bounded covariance \(\Sigma \preceq \sigma^{2} I\). For fixed \(x_{t}\), inliers are drawn  independently and identically distributed (i.i.d.), but the dependence across all iterations \(t\) is allowed.

\begin{theorem}
  \label{thm:sample_complexity_stability}
  Let \(\X \subset \R^{l}\) be a closed and bounded set with radius at most \(\gamma\) (i.e., \(\norm{x} \le \gamma, \; \forall x \in \X\)). Let \(p^{*}\) be a distribution over functions \(h: \X \ra \R^{k}\). Let $\xi \in (0,1)$ be the failure probability. Suppose \(h\) is \(L\)-Lipschitz and uniformly bounded, i.e.,\ there exists \(B > 0\) such that \(\norm{h(x)} \le B \; \forall x \in \X \) almost surely.
  Assume further that for each \(x \in \X\) we have \(\opnorm{\Cov_{h \sim p^{*}}(h(x))} \le \sigma^{2}\).
  Let \(S\) be a multiset of \(n\) i.i.d.\ samples \(\{h_{i}\}_{i=1}^{n}\) from \(p^{*}\). Then there exists a sample size
  \[ n = O\left( \frac{k \log k + l \log\left(\frac{\gamma L B }{\sigma\epsilon\xi}\right) }{\epsilon}\right)\]
  such that with probability \(1-\xi\), it holds that for each \(x \in \X\), there exists a subset \(S' \subset S\) with
  \(|S'| \ge (1-2\epsilon)n\) (potentially different subsets \(S' \) for
  different \(x\)) such that \(\{h_{i}(x)\}_{i \in S'}\) is
  \((2\epsilon, \delta)\)-stable with respect to \(\mu\) and \(\sigma^{2}\), where  \(\delta = \bigO{\sqrt\epsilon}\).
\end{theorem}

To prove this theorem, we work with an easier version of stability condition specialized to samples from a bounded covariance distribution.

\begin{claim}[Claim 2.1 in~\cite{diakonikolas2020outlier}]
  \label{claim:stability_bounded_covariance}
  (Stability for bounded covariance) Let $R \subset \mathbb{R}^k$ be a finite multiset such that $\norm{\mu_R-\mu}\leq \sigma\delta$, and $\opnorm{\bar{\Sigma}_R-\sigma^{2}I} \leq \sigma^{2}\delta^2 / \epsilon$ for some $0 \leq \epsilon \leq \delta$. Then $R$ is $\left(\Theta(\epsilon), \delta^{\prime}\right)$ stable with respect to $\mu$ and $\sigma^2$, where $\delta^{\prime}=O(\delta+\sqrt{\epsilon})$.
\end{claim}
\begin{proof}[Proof of Theorem~\ref{thm:sample_complexity_stability}]
  Given Claim~\ref{claim:stability_bounded_covariance}, it suffices to show that with probability \(1-\xi\), for each \(x \in \X\) there exists a subset \(S' \subset S\) with \( |S'| \ge (1 - 2\epsilon) n \) such that
  \begin{align}
    \label{eq:mean_stability}
    \norm{\E_{i \in S'}h_i(x) - \E_{h \sim p^{*}}h(x)} &\le \bigO{\sigma\delta} \\
    \label{eq:covariance_stability}
    \opnorm{\E_{i \in S'}(h_i(x) -  \E_{h \sim p^{*}}h(x))(h_i(x) - \E_{h \sim p^{*}}h(x))^{\top} - \sigma^{2} I} &\le \bigO{\sigma^{2}\delta^{2}/\epsilon}
  \end{align}
  By Proposition~\ref{prop:bounded_variance_stability_iid}, for each \(x \in \X\), with probability \(1 - \left(\frac{\xi\gamma B L}{\sigma \sqrt{\epsilon}}\right )^{l} \ge 1 - \xi/\left(\frac{\gamma B L}{\sigma \sqrt{\epsilon}}\right )^{l}\), there exists a subset \(S' \subset S\) with \( |S'| \ge (1 - 2\epsilon) n \) such that~\eqref{eq:mean_stability} and~\eqref{eq:covariance_stability} hold.

  We proceed with a net argument. Up to a multiplicative error that can be suppressed by \(\bigO{\cdot}\), if Equation~\eqref{eq:mean_stability} holds for some \(x \in \X\), it also holds for all other \(x'\) in a ball of radius \(\sigma\delta/L\) because \(h(\cdot)\) is \(L\)-Lipschitz. Similarly, \eqref{eq:covariance_stability} is equivalent to
  \[|\E_{i \in S'}((h_i(x) - \E_{h \sim p^{*}}h(x))^{\top}v)^{2} - \sigma^{2} v^{\top}v| \le \bigO{\sigma^{2}\delta^{2}/\epsilon}\qquad \forall v\in\R^{k}.\]
  Since \(h(\cdot)\) is \(L\)-Lipschitz and uniformly bounded by \(B\), if~\eqref{eq:covariance_stability} holds for some \(x \in \X\), it also holds for all \(x'\) in a ball of radius \(\frac{\sigma^{2}\delta^{2}}{\epsilon B L}\). Therefore, it suffices for Equation~\eqref{eq:mean_stability} and~\eqref{eq:covariance_stability} to hold for a \(\tau\)-net of \(\X\), where for \(\delta = \bigO{\sqrt\epsilon}\) we have \(\tau = \min\left(\frac{\sigma\delta}{L}, \frac{\sigma^{2}\delta^{2}}{\epsilon B L}\right) = \Omega\left( \frac{\sigma\sqrt\epsilon }{B L}\right )\). An \( \Omega\left( \frac{\sigma\sqrt\epsilon }{B L}\right )\)-net of \(\gamma\)-radius ball in \(\R^{l}\) has size \(\bigO[\Big]{\left(\frac{\gamma B L}{\sigma \sqrt{\epsilon}}\right )^{l}}\). Taking a union bound over this net completes the proof.
\end{proof}


\paragraph{Proof of Main Theorem}
We establish a slightly enhanced version of Theorem~\ref{thm:global_convergence-general} by including the last sentence as an additional component to the original theorem.
\begin{theorem}
  \label{thm:global_convergence-general-appendix}
  Suppose we are given \(\epsilon\)-corrupted set of functions  \(\{f_{i}\}_{i = 1}^{n}\) for sample size $n$, generated according to Definition~\ref{def:robust-stochastic-opt}.
  Suppose Assumption~\ref{assump:general_clean_data} holds in a bounded region \(\B \subset \R^D\) of radius \(\Gamma\) with gradient and Hessian covariance bound $\sigma_g$ and $\sigma_H$ respectively, and we have an arbitrary initialization \(x_{0} \in \B\).
  For some \(n = \bigtO[\big]{D^{2}/\epsilon}\), Algorithm~\ref{alg:inexact_randomized} initialized at \(x_{0}\) outputs an \((\epsilon_{g}, \epsilon_{H})\)-approximate SOSP in polynomial time with high probability if the following conditions hold:
  \begin{enumerate}[(I)]
    \item All iterates \(x_{t}\) in Algorithm~\ref{alg:inexact_randomized} stay inside the bounded region \(\B\).
    \item For an absolute constant \(c > 0\), it holds that \(\sigma_{g}\sqrt\epsilon \le c \epsilon_{g}\) and \(\sigma_{H}\sqrt\epsilon \le c\epsilon_{H}\).
  \end{enumerate}

  Moreover, there exists an absolute constant \(\Crme\) such that for each iteration \(t\), the gradient oracle \(\inexact g_{t}\) and Hessian oracle \(\inexact H_{t}\) satisfy
  \(\norm{\inexact g_{t} - \grad{f}(x_{t})} \le \Crme \sigma_{g}\sqrt\eps\) and \(\fnorm[\big]{\inexact H_{t} - \grad^{2} f(x_{t})} \leq\Crme \sigma_{H}\sqrt\epsilon\).
\end{theorem}

And we recall that the we construct the gradient and Hessian oracles in Algorithm~\ref{alg:inexact_randomized} in the following way:
\begin{align*}
  \tilde g_{t} & \la \mathbf{RobustMeanEstimation}(\{\grad f_{i}(x_{t})\}_{i=1}^{n}, 4\epsilon) \\
  \tilde H_{t} & \la \mathbf{RobustMeanEstimation}(\{\grad^{2} f_{i}(x_{t})\}_{i=1}^{n}, 4\epsilon)
\end{align*}

If we ignore that the dependence between $\{\grad f_i(x_t)\}$ introduced via $x_t$, this theorem follows directly from Proposition~\ref{prop:robust-mean-estimation-summary}. Here we fix this technicality via a union bound over all $x_t$ with Theorem~\ref{thm:sample_complexity_stability}.

\begin{proof}

By Proposition~\ref{prop:optimization_inexact_derivatives_general}, it suffices to check that gradient and Hessian inexactness condition is satisfied. That is, for all iterations \(t\), it holds that
\begin{align}
  \label{eq:gradient_inexactness-general-appendix}
  \norm{\inexact g_{t} - \grad{f}(x_{t})} &\le \frac{1}{3} \epsilon_{g}, \\
  \label{eq:Hessian_inexactness-general-appendix}
  \norm[\big]{\vect(\inexact H_{t}) - \vect(\grad^{2} f(x_{t}))} &\leq \frac29 \epsilon_{H}.
\end{align}
Here we use the fact that \(\norm[\big]{\vect(\inexact H_{t}) - \vect(\grad^{2} f(x_{t}))}= \fnorm[\big]{\inexact H_{t} - \grad^{2} f(x_{t})} \ge \opnorm[\big]{\inexact H_{t} - \grad^{2} f(x_{t})}\).

We proceed to apply Theorem~\ref{thm:sample_complexity_stability} and consider the function \(h_{g}(x, A) = \vect(\grad_{x} f(x, A))\) and \(h_{H}(x) = \vect(\grad_{xx}^{2} f(x, A))\). We consider \(h_{H}: \R^{D} \ra \R^{D^{2}}\) first. By assumption (I), all iterates never leave the bounded region \(\mathcal{B}\) with high probability; we condition on this event in the remaining analysis and it follows that \(\opnorm{\Cov_{A \sim \G}(\vect(\grad^{2} f(x_{t},A)))} \le \sigma_{H}^{2}\) for all iterations \(t\).

Assumption~\ref{assump:general_clean_data} (ii) posits that with high probability \(h_{H}(\cdot, A)\) is \(L_{D_H}\)-Lipschitz and its \(\ell_{2}\)-norm is bounded above by \(B_{D_H}\). We further condition on this event. Let \(\xi \in (0,1)\).  Since \(\opnorm{\Cov_{A \in \G}(h_{H}(x, A))} \le \sigma_{H}\) by Assumption~\ref{assump:general_clean_data} (iii), there exists a sample size
\[ n = O\left( \frac{D^{2} \log D + D \log\left(\frac{\gamma L_{D_{H}} B_{D_{H}} }{\sigma_{H}\epsilon\xi}\right) }{\epsilon}\right),\]
such that with probability \(1-\xi\), it holds that for each \(x \in \X\), there
exists a \((2\epsilon, \bigO{\sqrt\epsilon})\)-stable subset of size at least
\((1-2\epsilon)n\) in the sense of Definition~\ref{def:stability} (potentially
different subsets \(S' \) for different \(x\)). Therefore, conditioning on the
event that for all \(x\) there exists a stable subset,
Proposition~\ref{prop:robust_mean_estimation_with_stability} implies that there exists an absolute constant \(\Crme\) such that \(\norm[\big]{\vect(\inexact H_{t}) - \vect(\grad^{2} f(x_{t}))} \le \Crme {\sigma_{H}\sqrt\epsilon}\) for all \(t\). By assumption (II) with a sufficiently small  constant \(c\), we have \(\Crme {\sigma_{H}\sqrt\epsilon} \le 2\epsilon_{H}/ 9\), and therefore \(\norm[\big]{\vect(\inexact H_{t}) - \vect(\grad^{2} f(x_{t}))} \leq \frac29 \epsilon_{H}.\)

Now we established that for Equation~\eqref{eq:Hessian_inexactness-general-appendix} to hold for all iterations \(t\) with high probability, the sample complexity is required to be at least \(n = \bigtO{D^{2}/\epsilon}\). A similar argument implies that \(\bigtO{D/\epsilon}\) samples are needed for Equation~\eqref{eq:gradient_inexactness-general-appendix} to hold for all iterations \(t\) with high probability, which is dominated by \( \bigtO{D^{2}/\epsilon}\).
\end{proof}

\section{Omitted Proofs in Low Rank Matrix Sensing}\label{sec:low-rank-appendix}
This section provides omitted computation and proofs for robust low rank matrix sensing (Section~\ref{sec:matrix-low-rank}), establishing Theorem~\ref{thm:low-rank-mtrx-sensing-intro}. Appendix~\ref{sec:computation-appendix} computes the gradient and Hessian of $f_i$ defined in Equation~\eqref{eq:obj_u_i} and their covariance upper bounds under Definition~\ref{def:iid_Gaussian_with_noise} with noiseless measurements ($\sigma = 0$). Appendix~\ref{sec:global-convergence-appendix} proves the global convergence of our general robust nonconvex optimization framework (Theorem~\ref{thm:global_convergence-general}) applied to the noiseless robust low rank matrix sensing problem, which leads to an approximate SOSP, and Appendix~\ref{sec:local-linear-convergence-appendix} proves the local linear convergence to a global minimum from this approximate SOSP. Finally, Appendix~\ref{sec:sensing_noise} discusses the case of noisy measurements ($\sigma \neq 0$) under Definition~\ref{def:iid_Gaussian_with_noise}.

\subsection{Omitted Computation in Low Rank Matrix Sensing}\label{sec:computation-appendix}

\subsubsection{Sample Gradient and Sample Hessian}\label{sec:sample_grad_and_hessians}

\paragraph{Summary of Results} Consider the noiseless setting as in Theorem~\ref{thm:low-rank-mtrx-sensing-intro} with $\sigma = 0$. We now compute the gradient and the Hessian of \(f_{i}: \R^{d \times r} \rightarrow \R \). We firstly summarize the results. The gradient of \(f_{i}\) at point \(U\), denoted by \(\grad f_{i}(U)\), is a matrix in \(\R^{d \times r}\) given by
\[\grad f_{i}(U) = \< U U^{\top} - M^{*}, A_{i} \>  (A_{i} + A_{i}^{\top}) U.\]

The Hessian of \(f_{i}\) at point \(U\), denoted by \((H_{i})_{U}\), is a linear operator \((H_{i})_{U}: \R^{d \times r} \rightarrow \R^{d \times r}\). The operator \((H_{i})_{U}\) acting on a matrix \(Y \in \R^{d \times r}\) gives:
\[(H_{i})_{U}(Y) = \< U U ^{\top} - M^{*}, A_{i} \> (A_{i} + A_{i}^{\top}) Y + \< Y, (A_{i} + A_{i}^{\top})U \> (A_{i} + A_{i}^{\top})U.\]

We want to identify this linear map with a matrix \(H_{i}(U)\) of dimension
\(\R^{dr \times dr}\). We vectorize both the domain and the codomain of
\((H_{i})_{U}\) so that it can be represented by a matrix.

Let \(I_{r}\) and \( I_{d}\) denote the identity matrices of dimension \(r \times r\) and \( d\times d\) respectively.
\[H_{i}(U) = \< U U ^{\top} - M^{*}, A_{i} \> I_{r} \otimes (A_{i} + A_{i}^{\top})  + \vect \left((A_{i} + A_{i}^{\top})U  \right) \vect\left((A_{i} + A_{i}^{\top})U\right)^{\top}. \]

In this paper, we sometimes abuse the notation and use \(\grad^{2}f_{i}(U)\) to refer to either \(H_{i}(U)\) or \((H_{i})_{U}\), but its precise meaning should be clear from its domain.
\paragraph{Computation of Gradients and Hessians} Let \((D^{k} f_{i})_{U}\) denote the \(k\)-th order  derivative of \(f_{i}\) at point \(U \in  \R^{d \times r}\) and \(\Lin(X, Y)\) denote the space of all linear mappings from \(X\) to \(Y\). We consider higher-order derivatives as linear maps:
\begin{align*}
  (D f_{i})_{U} &: \R^{d \times r} \rightarrow \R \\
  (D^{2} f_{i})_{U} &:  \R^{d \times r} \rightarrow  \Lin(\R^{d \times r}, \R)
\end{align*}

We identify \((D f_{i})_{U}\) with the matrix \(\grad f_{i}(U)\) such that \[(D f_{i})_{U}(Z) = \< \grad f_{i}(U), Z \>\] and identify \((D^{2} f_{i})_{U}\) with a linear operator \((H_{i})_{U}:  \R^{d \times r} \rightarrow  \R^{d \times r} \) such that \[ (D^2 f_{i})_{U}(Y)(Z) = \< (H_{i})_{U}(Y), Z \>. \]

Since \(f_{i}\) is differentiable, applying its derivative at \(U\) to a matrix \(Z\) gives the corresponding directional derivative at \(U\) for the direction \(Z\).

\begin{align*}
  (Df_{i})_{U}(Z) & = \dd{}{t} f_{i}(U+t Z)\Bigr|_{t=0} \\
              & = \dd{}{t} \frac{1}{2} (\< (U+t Z) (U+t Z)^{\top}, A_{i} \> - y_{i})^{2} \Bigr|_{t=0} \\
              & =\left(\< U U^{\top}, A_{i} \> - y_{i}\right) \left(\dd{}{t} \< U U^{\top} + t^{2} Z Z^{\top} + t(Z U^{\top} + U Z^{\top}), A_{i} \> \right)\Bigr|_{t=0} \\
              & = \left(\< U U^{\top}, A_{i} \> - y_{i}\right)  \< \dd{}{t}\left(t^{2} Z Z^{\top} + t(Z U^{\top} + U Z^{\top})  \right), A_{i}  \> \Bigr|_{t=0} \\
              & \stackrel{(a)}{=} \left(\< U U^{\top}, A_{i} \> - y_{i}\right)  \< Z U^{\top} + U Z^{\top}, A_{i}  \> \\
              & =\left(\< U U^{\top}, A_{i} \> - y_{i}\right)  \< A_{i}, Z U^{\top} \> + \< A_{i}, U Z^{\top}  \>\\
              & =\left(\< U U^{\top}, A_{i} \> - y_{i}\right)  \< A_{i}U, Z  \> + \< U^{\top} A_{i}, Z^{\top}  \> \\
              & \stackrel{(b)}{=} \left(\< U U^{\top}, A_{i} \> - y_{i}\right)  \< (A_{i} + A_{i}^{\top}) U, Z  \>.
\end{align*}

From (a) we conclude \((Df_{i})_{U}\) is the functional \(Z \mapsto\left(\< U U^{\top}, A_{i} \> - y_{i}\right)  \< Z U^{\top} + U Z^{\top}, A_{i}  \>\); it would be easier to calculate the second derivative from this form. From (b), using $y_i = \< A_i, M^{*} \>,$ we obtain the following closed-form expression for the gradient \(\grad f_{i}(U)\):
\[ \grad f_{i}(U) = \< U U^{\top} - M^{*}, A_{i} \>  (A_{i} + A_{i}^{\top}) U.\]

To calculate the second derivative of \(f_{i}\) at point \(U\), we study the variation of its first derivative \(U \mapsto (D f_{i})_{U}\) at direction \(Y\). Recall that the second derivative lives in the space of linear functionals \(\Lin(\R^{d \times r}, \R)\) and we use \(Z \in \R^{d \times r}\) to denote the input of this linear functional.
\begin{align*}
  (D^{2}f_{i})_{U}(Y)
  & = \dd{}{t} (D f_{i})_{U+t Y} \Bigr|_{t = 0} \\
  & = \dd{}{t} \left\{ Z \mapsto \left(\< (U + t Y) (U + t Y)^{\top}, A_{i} \> - y_{i}\right)  \< Z (U + t Y)^{\top} + (U + t Y) Z^{\top}, A_{i}  \> \right\} \Bigr|_{t=0} \\
  & =   \left(\< U U ^{\top}, A_{i} \> - y_{i}\right) \dd{}{t} \left\{ Z \mapsto   \< Z (U + t Y)^{\top} + (U + t Y) Z^{\top}, A_{i}  \> \right\} \Bigr|_{t=0} + \\
  & \qquad \qquad \dd{}{t} \left\{\< (U + t Y) (U + t Y)^{\top}, A_{i} \> - y_{i}\right\}  \Bigr|_{t=0} \left\{Z \mapsto \< Z U^{\top} + U  Z^{\top}, A_{i}  \> \right\} \\
  & =   \left(\< U U ^{\top}, A_{i} \> - y_{i}\right) \{Z \mapsto \< Z Y^{\top} + Y Z ^{\top},A_{i} \> \} + \\
  & \qquad \qquad \< U Y^{\top} + Y U^{\top}, A_{i} \>  \left\{Z \mapsto \< Z U^{\top} + U  Z^{\top}, A_{i}  \> \right\} \\
  & = Z \mapsto   \left(\< U U ^{\top}, A_{i} \> - y_{i}\right) \< Z Y^{\top} + Y Z ^{\top},A_{i} \> + \< U Y^{\top} + Y U^{\top}, A_{i} \> \< Z U^{\top} + U  Z^{\top}, A_{i}  \> \\
  & =  Z \mapsto (\< U U ^{\top}, A_{i} \> - y_{i}) \<(A_{i} + A_{i}^{\top}) Y, Z\> + \< Y, (A_{i} + A_{i}^{\top})U \>  \<(A_{i} + A_{i}^{\top})U, Z\>.
\end{align*}
Recall that we identify \((D^{2} f_{i})_{U}\) with a linear operator
\((H_{i})_{U}: \R^{d \times r} \rightarrow \R^{d \times r} \) such
that \[ (D^2 f_{i})_{U}(Y)(Z) = \< (H_{i})_{U}(Y), Z \>. \] Note that
\((H_{i})_{U}: \R^{d \times r} \rightarrow \R^{d \times r}\) viewed as a
fourth-order tensor is difficult to write down in a closed form, so we vectorize
both the domain and the codomain of \((H_{i})_{U}\) so that it can be
represented by a matrix.

The operator \((H_{i})_{U}\) acting on matrix \(Y \in \R^{d \times r}\) gives
\[(H_{i})_{U}(Y) = \< U U ^{\top} - M^{*}, A_{i} \> (A_{i} + A_{i}^{\top}) Y + \< Y, (A_{i} + A_{i}^{\top})U \>  (A_{i} + A_{i}^{\top})U.\]

We identify this linear map with a matrix \(H_{i}(U)\) of dimension \(\R^{dr \times dr}\), acting on the vectorized version of \(Y\). As a shorthand, write \(y = \vect(Y)\) and \(B_{i} = A_{i} + A_{i}^{\top}\).

Let \(I_{r}\) and \( I_{d}\) denote the identity matrix of dimension \(r \times r\) and \( d\times d\) respectively. Then
\begin{align*}
  H_{i}(U)\, y
  & =  \vect \left( \< U U ^{\top} - M^{*}, A_{i} \> (A_{i} + A_{i}^{\top}) Y + \< Y, (A_{i} + A_{i}^{\top})U \>  (A_{i} + A_{i}^{\top})U  \right) \\
  & = \< U U ^{\top} - M^{*}, A_{i} \> \vect \left(  B_{i} Y I_{r}\right) + \vect \left(B_{i}U  \right) \vect(B_{i}U)^{\top}\vect(Y) \\
  & = \< U U ^{\top} - M^{*}, A_{i} \> I_{r} \otimes B_{i} \vect \left( Y \right) + \vect \left(B_{i}U  \right) \vect(B_{i}U)^{\top}\vect(Y) \\
  & = \< U U ^{\top} - M^{*}, A_{i} \> (I_{r} \otimes B_{i}) y + \vect \left(B_{i}U  \right) \vect(B_{i}U)^{\top} y.
\end{align*}

Hence \[H_{i}(U) = \< U U ^{\top} - M^{*}, A_{i} \> I_{r} \otimes B_{i}  + \vect \left(B_{i}U  \right) \vect(B_{i}U)^{\top}. \]

\subsubsection{Gradient Covariance Bound}\label{sec:grad_cov_bound}

\begin{lemma}[Equation~\eqref{eq:gradient_covariance_bound}]
Consider the noiseless setting as in Theorem~\ref{thm:low-rank-mtrx-sensing-intro} with $\sigma = 0$. For all \(U \in \R^{d \times r}\) with \(\opnorm{U}^{2} \le \Gamma\), it holds that
  \begin{equation}
    \opnorm{\Cov(\vect( \grad f_{i}(U)))} \le 8 \fnorm{U U^{\top} - M^{*}}^{2} \opnorm{U}^{2} \le 32 r^{2}\Gamma^{3}.
  \end{equation}
\end{lemma}

Recall that \(\Gamma \ge 36 \sigstarl\), so we have the following trivial bound that is used frequently
\[ \opnorm{M^{*}} = \sigstarl \le \Gamma / 36 < \Gamma. \]

\begin{proof}[Proof of Equation~\eqref{eq:gradient_covariance_bound}]
Recall that  \[\grad f_{i}(U) = \< U U^{\top} - M^{*}, A_{i} \> (A_{i} + A_{i}^{\top}) U. \]

We frequently encounter the following calculations.
\begin{lemma}
  \label{lemma:Gaussian_quadratic_form_rank1}
  Let \(P,Q \in \R^{d \times d}\) be given and let \(A_{i}\) have i.i.d.\ standard Gaussian entries. Then
  \begin{align}
    \E[\<P, A_{i}\> \<Q, A_{i}\>] &= \<P, Q\>, \\
    \Var[\<P, A_{i}\> \<Q, A_{i}\>] &\le  2 \fnorm{P}^{2}\fnorm{Q}^{2}.
  \end{align}
\end{lemma}
\begin{proof}
  By Definition~\ref{def:iid_Gaussian_with_noise}, \(X:=\vect(A_{i})\) is a standard Gaussian vector with identity covariance \(\Sigma = I_{d^2}\).  Let \(a = \vect(P), b = \vect(Q)\), and \(G = a b^{\top}\). Fact~\ref{fact:Gaussian_quadratic_form} implies that
  \begin{align*}
    \E \<P, A_{i}\> \<Q, A_{i}\>
    & = \< ab^{\top}, I_{dr} \>  = \<a, b\> = \<P, Q\>\\
    \Var[\<P, A_{i}\> \<Q, A_{i}\>]
    & =  \tr(G\Sigma (G + G^{\top})\Sigma) = \tr(G(G+G^{\top})) \\
    & = \tr(ab^{\top}(ab^{\top} + ba^{\top})) \\
    & = b^{\top} a \tr(a b^{\top}) +  b^{\top} b \tr(a a^{\top}) \\
    & = \<a, b\>^{2} + \norm{a}^{2}\norm{b}^{2} \\
    & \le 2 \norm{a}^{2}\norm{b}^{2} \text{\qquad (Cauchy-Schwarz)} \\
    & =  2 \fnorm{P}^{2}\fnorm{Q}^{2}. \qedhere
  \end{align*}
\end{proof}

To compute the mean and the variance of the gradient \(\grad f_{i}(U)\), it is more convenient to work with \(\< \grad f_{i}(U), Z \>\) for some \(Z \in \R^{d \times r}\) with \(\fnorm{Z} = 1\).

\begin{align*}
  \< \grad f_{i}(U), Z \>
  & = \< U U^{\top} - M^{*}, A_{i}\> \< (A_{i} + A_{i}^{\top})U, Z\> \\
  & = \< U U^{\top} - M^{*}, A_{i}\> \< A_{i} + A_{i}^{\top}, Z U^{\top}\> \\
  & = \< U U^{\top} - M^{*}, A_{i}\> \< Z U^{\top} + U Z^{\top}, A_{i}\>.
\end{align*}

Lemma~\ref{lemma:Gaussian_quadratic_form_rank1} implies that the expectation
\begin{align*}
  \E \< \grad f_{i}(U), Z \>
  & = \< U U^{\top} - M^{*}, Z U^{\top} + U Z^{\top}\> \\
  & = \< \left((U U^{\top} - M^{*}) + (U U^{\top} - M^{*})^{\top}\right)U, Z\> \\
  & = \< 2\left(U U^{\top} - M^{*}\right)U, Z\>.
\end{align*}

By linearity of expectation, we conclude
\begin{equation}
  \label{eq:expected_gradient}
   \E\grad f_{i}(U)= 2(U U^{\top} - M^{*})U.
\end{equation}

The variance bound can also be obtained via Lemma~\ref{lemma:Gaussian_quadratic_form_rank1}
\begin{align*}
  \Var \< \grad f_{i}(U), Z \>
  & \le 2  \fnorm{U U^{\top} - M^{*}}^{2} \fnorm{Z U^{\top} + U Z^{\top}}^{2} \le 8 \|U U^{\top} - M^{*}\|_{F}^{2} \opnorm{U}^{2} \|Z\|_{F}^{2},
\end{align*}
where the last inequality comes from Fact~\ref{fact:fnorm}.

Since \(\Var \< \grad f_{i}(U), Z \> = \Var[\vect( \grad f_{i}(U))^{\top} \vect(Z)] = \vect(Z)^{\top} \Cov(\vect( \grad f_{i}(U))) \vect(Z)\), we have
\begin{align*}
  \opnorm{\Cov(\vect( \grad f_{i}(U)))}
  & \le 8 \fnorm{U U^{\top} - M^{*}}^{2} \opnorm{U}^{2} \\
  & \le 8 \Gamma \left(\fnorm{U U^{\top}} + \fnorm{M^{*}}\right)^{2} \\
  & \le 8 \Gamma (r \Gamma + r \sigstarl)^{2} = 32 r^{2} \Gamma^{3}. \qedhere
\end{align*}
\end{proof}

\subsubsection{Hessian Covariance Bound}
\label{sec:Hessian_cov_bound}




\begin{lemma}[Equation~\eqref{eq:hessian_covariance_bound}]
  Consider the noiseless setting as in Theorem~\ref{thm:low-rank-mtrx-sensing-intro} with $\sigma = 0$. For all \(U \in \R^{d \times r}\) with \(\opnorm{U}^{2} \le \Gamma\), it holds that
  \begin{equation}
    \opnorm{\Cov(\vect(H_{i}))} \le  16r \fnorm{ U U ^{\top} - M^{*}}^{2}  + 128 \opnorm{U}^{4} \le 192 r^{3} \Gamma^{2}.
  \end{equation}
\end{lemma}

\begin{proof}

As a shorthand, we write \(B_{i} = A_{i} + A_{i}^{\top}\). Recall that
\[H_{i}(U) = \< U U ^{\top} - M^{*}, A_{i} \> I_{r} \otimes B_{i} + \vect \left( B_{i} U  \right) \vect\left(B_{i}U\right)^{\top}.\]

Let \(W\) be a \(dr \times dr\) matrix that has Frobenius norm 1. We work with \(\<H_{i}(U), W\>\) as we did in the gradient covariance bound calculation.

By Young's inequality for \(L^{2}\) random variables (i.e., \(\norm{\mathcal A + \mathcal B}^{2}_{L^{2}} \le 2 \norm{\mathcal A}^{2}_{L^{2}} + 2 \norm{\mathcal B}^{2}_{L^{2}}\)),
\begin{align*}
  \Var \< H_{i}(U), W\>
  & \le 2\Var\underbrace{[\< U U ^{\top} - M^{*}, A_{i} \> \<I_{r} \otimes B_{i}, W\>]}_{\mathcal{A}} + 2\Var\underbrace{\left[ \<\vect \left(B_{i}U  \right) \vect(B_{i}U)^{\top}, W\> \right]}_{\mathcal{B}}.
\end{align*}

We first bound the term \(\mathcal{A}\). Note \(I_{r} \otimes B_{i}\) consists of \(r\) copies of \(B_{i}\) in the diagnal \[
I_{r} \otimes B_{i} =
\left[
\begin{array}{c|c|c}
  B_{i} & \cdots & O \\
  \hline
  \vdots & \ddots & \vdots \\
  \hline
  O & \cdots &B_{i}
\end{array}
\right].
\]

We partition \(W\) into \(r^{2}\) submatrices of dimension \(d \times d\).
\[
W=
\left[
\begin{array}{c|c|c}
  W_{11} & \cdots & W_{1r} \\
  \hline
  \vdots & \ddots & \vdots \\
  \hline
  W_{r1} & \cdots & W_{rr}
\end{array}
\right].
\]

Then
\begin{align*}
  & \Var[\< U U ^{\top} - M^{*}, A_{i} \> \<I_{r} \otimes B_{i}, W\>] \\
  & = \Var[\< U U ^{\top} - M^{*}, A_{i} \> \< B_{i}, \sum_{i=1}^{r}W_{ii}\>]\\
  &  = \Var[\< U U ^{\top} - M^{*}, A_{i} \> \< A_{i}, \sum_{i=1}^{r} (W_{ii} + W_{ii}^{\top})\>] \\
  & \le 2 \fnorm{ U U ^{\top} - M^{*}}^{2} \fnorm{\sum_{i=1}^{r} (W_{ii} + W_{ii}^{\top})}^{2},
\end{align*}
where the last inequality uses Lemma~\ref{lemma:Gaussian_quadratic_form_rank1}.

Write \(a_{i} = \vect(W_{ii}) \in \R^{d^{2}}, b_{i} = \vect(W_{ii}^{\top}), c_{i} = a_{i} + b_{i}\). Then \(\norm{a_{i}} = \norm{b_{i}} = \fnorm{W_{ii}}\) and
\begin{align*}
  \fnorm{\sum_{i=1}^{r} (W_{ii} + W_{ii}^{\top})}^{2}
  & = \norm{\sum_{i=1}^{r}c_{i}}^{2} \\
  & = \sum_{j=1}^{d^{2}} \left(\sum_{i=1}^{r}c_{ij}\right)^{2} \le \sum_{j=1}^{d^{2}} r \sum_{i=1}^{r} c_{ij}^{2} \qquad \text{(Cauchy-Schwarz)}\\
  & = r \sum_{i=1}^{r} \norm{c_{i}}^{2} = r \sum_{i=1}^{r} \norm{a_{i} + b_{i}}^{2} \\
  & \le 2r \left(\sum_{i=1}^{r}(\norm{a_{i}}^{2} + \norm{b_{i}}^{2}) \right) \\
  & = 4 r \sum_{i=1}^{r}\fnorm{W_{ii}}^{2} \le 4 r \fnorm{W}^{2} = 4 r.
\end{align*}

Therefore,
\begin{align*}
  \Var[\< U U ^{\top} - M^{*}, A_{i} \> \<I_{r} \otimes B_{i}, W\>] & \le 2 \fnorm{ U U ^{\top} - M^{*}}^{2} \fnorm{\sum_{i=1}^{r} (W_{ii} + W_{ii}^{\top})}^{2} \\
  & \le 8r \fnorm{ U U ^{\top} - M^{*}}^{2} \fnorm{W}^{2} = 8r \fnorm{ U U ^{\top} - M^{*}}^{2}.
\end{align*}

To bound the term \(\mathcal{B}\), we apply Fact~\ref{fact:Gaussian_quadratic_form} with \(\Sigma = \Cov(\vect(B_{i}U)) \in \R^{dr \times dr}\):
\begin{align*}
  & \Var\left[ \<\vect \left(B_{i}U  \right) \vect(B_{i}U)^{\top}, W\> \right] \\
  & = \Var\left[  \vect(B_{i}U)^{\top} W \vect (B_{i}U) \right] \\
  & =  \tr(W \Sigma (W + W^{\top}) \Sigma) \\
  & =  \< W \Sigma, \Sigma W\> +  \< W \Sigma, \Sigma W^{\top}\> \\
  & \le  \fnorm{W \Sigma} \fnorm{\Sigma W} +   \fnorm{W \Sigma} \fnorm{\Sigma W^{\top}} \\
  & \le  \fnorm{W}\opnorm{\Sigma} \opnorm{\Sigma}\fnorm{W} +  \fnorm{W} \opnorm{\Sigma} \opnorm{\Sigma} \fnorm{W} \\
  & = 2 \fnorm{W}^{2} \opnorm{\Sigma}^{2} = 2 \opnorm{\Sigma}^{2},
\end{align*}
where the last inequality comes from Fact~\ref{fact:fnorm}.

To bound \(\opnorm\Sigma ^{2}\), we use Fact~\ref{fact:tensor} again. Since \(\vect(B_{i}U) = \vect(I_{d}B_{i}U) = (U^{\top} \otimes I_{d}) \vect(B_{i})\), we have
\begin{equation}
  \label{eq:covariance_BU}
  \Cov(\vect{(B_{i}U)}) =  (U^{\top} \otimes I_{d}) \Cov(\vect(B_{i})) (U \otimes I_{d}).
\end{equation}

Let \(P \in \R^{d^{2} \times d^{2}}\) be the permutation that maps \(\vect(A_{i})\) to \(\vect(A_{i}^{\top})\). Then
\begin{align*}
  \Cov(\vect(B_{i}))
  & = \Cov(\vect(A_{i}) + \vect(A_{i}^{\top}))   \\
  & \le 2 I_{d^{2}} + 2\Cov(\vect(A_{i}) , \vect(A_{i}^{\top})) & \text{(Young's inequality)}\\
  & = 2 I_{d^{2}} + 2\Cov(\vect(A_{i}) , P \vect(A_{i})) \\
  & = 2 I_{d^{2}} + 2 P\Cov(\vect(A_{i}) , \vect(A_{i})) \\
  & = 2 I_{d^{2}} + 2P.
\end{align*}

Since \(P\) is a permutation, we have \(\opnorm{\Cov(\vect(B_{i}))} \le 4\).

It follows from Equation~\eqref{eq:covariance_BU} and Fact~\ref{fact:kron_operator_norm} that \(\opnorm{\Cov(\vect(B_{i}U))} \le 4\opnorm{U}^{2}\) and therefore
\[ \Var\left[ \<\vect \left(B_{i}U  \right) \vect(B_{i}U)^{\top}, W\> \right] \le 64 \opnorm{U}^{4} \fnorm{W}^{2} = 64 \opnorm{U}^{4}. \]

We conclude that for all \(W \in \R^{dr \times dr}\), \[\Var \< H_{i}, W\> \le 16r \fnorm{ U U ^{\top} - M^{*}}^{2}  + 128 \opnorm{U}^{4},\] hence
\begin{equation*}
  \opnorm{\Cov(\vect(H_{i}))} \le  16r \fnorm{ U U ^{\top} - M^{*}}^{2}  + 128 \opnorm{U}^{4} \le 192 r^{3} \Gamma^{2},
\end{equation*}
where the last inequality uses \(r \ge 1\) and \(\opnorm{U}^{2} \le \Gamma\).
\end{proof}

\subsection{Global Convergence in Low Rank Matrix Sensing---Omitted Proofs}
\label{sec:global-convergence-appendix}
We firstly verify Assumption~\ref{assump:general_clean_data} (ii). We recall the entire assumption below:
\begin{assumption}[Assumption~\ref{assump:general_clean_data}]
   There exists a bounded region \(\mathcal{B}\) such that the function \(f\)  satisfies: 
  \begin{enumerate}[leftmargin=*]
    \item[(i)] There exists a lower bound $\flb > -\infty$ such that for all $x \in \mathcal{B}$, $f(x, A) \ge \flb$ with probability $1$.
    \item[(ii)] There exist parameters \(L_{D_g}, L_{D_H}, B_{D_g}, B_{D_H}\) such that with high probability over the randomness in $A$, \(f(\cdot, A)\) is \(L_{D_g}\)-gradient Lipschitz and \(L_{D_H}\)-Hessian Lipschitz, and the \(\ell_{2}\)-norm of gradient and the Frobenius norm of Hessian of \(f(\cdot, A)\) are upper bounded by $B_{D_g}$ and $B_{D_H}$ respectively.
    \item[(iii)]  There exist parameters \(\sigma_{g}, \sigma_{H} > 0\) such that

    \(\opnorm{\Cov_{A \sim \G}(\grad f(x,A))}  \le \sigma_{g}\) and   
    \(\opnorm{\Cov_{A \sim \G}(\vect(\grad^{2} f(x,A)))} 
    \le \sigma_{H}\).
  \end{enumerate}
\end{assumption}

\begin{proof}[Verifying Assumption~\ref{assump:general_clean_data}(ii)]
We take the bounded region $\mathcal{B}$ to be the region $\{U: \opnorm{U}^2 \le \Gamma\}$. Let $f_{i}(U) = \frac12 \left(\< U U^{\top}, A_{i} \> - y_{i}\right)^{2}$ be the cost function corresponding to clean samples. We need to verify that $\grad f_i(U)$ and $\grad^2 f_i(U)$ are \(\poly(dr\Gamma/\epsilon)\)-Lipschitz and bounded within $\mathcal{B}$ with high probability so that the violated samples constitute at most an $\eps$-fraction.

We discuss gradient-Lipschitzness as an example, focusing on the constant $L_{D_g}$. The rest of conditions can be checked similarly.  We drop the subscript $L := L_{D_g}$ in the following analysis for conciseness of the notation. 

Since \(\grad f_{i}(X) - \grad f_{i}(Y) = \< X X^{\top} - M^{*}, A_{i} \> (A_{i} + A_{i}^{\top}) (X - Y) + \<X X^{\top} - Y Y^{\top}, A_{i}\> (A_{i}+A_{i}^{\top})Y\), it suffices that \(\fnorm{A_{i}}^{2} \le 2d^2 + 3/\eps\) for \(\grad f_{i}(\cdot)\) to be \(\poly(dr\Gamma/\epsilon)\)-Lipschitz. Since \(\fnorm{A_{i}}^{2}\) follows chi-square distribution with degree of freedom \(d^{2}\), by Laurent-Massart bound, we have 
  \[\Prob\{\fnorm{A_{i}}^{2} - d^2 \ge d^2 + 3/\epsilon \} \le \exp(-1/\epsilon).\]

This completes the verification of Assumption~\ref{assump:general_clean_data}(ii) for $L_{D_g}$. We proceed to discuss why this high probability result implies that at most an $\eps$-fraction of clean samples violates gradient Lipschitzness. By Chernoff's inequality, the probability that more than \(\epsilon\)-fraction of uncorrupted \(A_{i}\)'s fail to satisfy \(\fnorm{A_{i}}^{2} \le 2 d^2 + 3/ \eps\) is less than
    \[\exp\left(-n\left( (1-\epsilon) \log\frac{1-\epsilon}{1-\exp(-1/\epsilon)} + \epsilon \log\frac{\epsilon}{\exp(-1/\epsilon)}\right)\right) = \bigO{\frac1\eps \exp(-n)},\] which is a small if \(n = \tilde O\left(\frac1{\epsilon} \right)\) is sufficiently large.
\end{proof}

\begin{corollary}
  \label{cor:robust_mean_estimation_guarantee}
  Consider the noiseless setting as in Theorem~\ref{thm:low-rank-mtrx-sensing-intro} with $\sigma = 0$. There exists a sample size \(n = \tilde O\left( \frac{d^{2}r^{2}}{\epsilon}\right)\) such that with $n$ samples, all subroutine calls to \(\mathsf{RobustMeanEstimation}\)  (Algorithm~\ref{alg:robust_mean_estimation}) to estimate population gradients \(\grad \bar f(U_{t})\) and population Hessians \(\grad^{2} \bar f(U_{t})\) succeed with high probability. In light of Equations~\eqref{eq:gradient_covariance_bound} and~\eqref{eq:hessian_covariance_bound}, this means that there exists a large enough constant \(\Crme\) such that with high probability, for all iterations \(t\), we have
  \begin{align}
    \label{eq:inexact_gradient_accuracy}
    \fnorm{\tilde g_{t} - \grad \bar f(U_{t})}  \le 2 \Crme \fnorm{U_{t} U_{t}^{\top} - M^{*}} \opnorm{U_{t}}  \\
    \fnorm{\tilde H_{t} - \grad^{2}\bar f(U_{t})} \le  4r^{1/2} \fnorm{ U_{t} U_{t} ^{\top} - M^{*}}  + 16 \opnorm{U_{t}}^{2}.
  \end{align}
\end{corollary}
\begin{proof}
  This follows directly from the last sentence of Theorem~\ref{thm:global_convergence-general-appendix}, with \(D = d^{2}r^{2}\) and \(\sigma_{g}, \sigma_{H}\) given by Equations~\eqref{eq:gradient_covariance_bound} and~\eqref{eq:hessian_covariance_bound}.
\end{proof}

We next prove Lemma~\ref{lemma:radius}, establishing that all iterates stay inside a bounded region in which the covariance bounds are valid. The analysis utilizes a ``dissipativity'' property~\cite{hale2010asymptotic}), which says that the iterate aligns with the direction of the gradient when the iterate's norm is large. For the gradient step, the gradient will reduce the norm of the iterate. For the negative curvature step, dissipativity property provides a lower bound on the gradient when the iterate's norm is large, but we show that Algorithm~\ref{alg:inexact_randomized} only takes a negative curvature step when this lower bound is violated, therefore the iterate's norm must be small and negative curvature steps with a fixed and small stepsize cannot increase the iterate's norm by too much.

\begin{proof}[Proof of Lemma~\ref{lemma:radius}]
  Recall Algorithm~\ref{alg:inexact_randomized} consists of inexact gradient descent steps of size \(1/L_{g} = \frac{1}{16\Gamma}\) and randomized inexact negative curvature steps of size
  \begin{equation}
    \label{eq:negative_curvature_stepsize}
    \frac{2\epsilon_{H}}{L_{H}} = \frac{\sigstarr}{36\Gamma} \le \frac{\Gamma^{1/2}}{36}.
  \end{equation}

  We proceed to use induction to prove the following for Algorithm~\ref{alg:inexact_randomized}:

  \begin{enumerate}
    \item Suppose at step \(\tau\) we run the negative curvature step to update the iterate from \(U_{\tau}\) to \(U_{\tau + 1}\), then \(\opnorm{U_{\tau}} \le \frac12 \Gamma^{1/2}\)
    \item \(\opnorm{U_{t}} \le \Gamma^{1/2}\) for all \(t \ge 0\).
  \end{enumerate}

  First we consider the inexact gradient steps. We denote the gradient inexactness by \(e_{t} = \grad\bar f(U_{t}) - \tilde g_{t}\). Recall that \[\fnorm{e_{t}} \le  4 r \Crme \sigstarl \opnorm{U_{t}}\sqrt{\epsilon}\] according to Equation~\eqref{eq:inexact_gradient_accuracy} in Lemma~\ref{cor:robust_mean_estimation_guarantee} whenever \(\opnorm{U_{t}} \le \Gamma^{1/2}\), which is true according to the induction hypothesis. Therefore, we have
  \begin{align*}
    \opnorm{U_{t+1}}
    & = \opnorm{U_{t} - \frac1{16\Gamma} \tilde g_{t}} = \opnorm{U_{t} - \frac{1}{16\Gamma} (2 (U_{t} U_{t} - M^{*}) U_{t} - e_{t})} \\
    & \le \opnorm{U_{t} - \frac1{8\Gamma} U_{t} U_{t}^{\top} U_{t}} + \frac1{8\Gamma}\opnorm{M^{*}U_{t} + \frac12 e_{t}}\\
    & \le \max_{i}\{ \sigma_{i}(U_{t}) - \frac1{8\Gamma} \sigma_{i}^{3}(U_{t})\} + \frac1{8\Gamma}(\sigstarl + 2 r \Crme \sigstarl \sqrt{\epsilon}) \opnorm{U_{t}}
  \end{align*}
  Since the function \(t \mapsto t - \frac1{8\Gamma} t^{3}\) is increasing in \([0, \sqrt{\frac{8\Gamma}{3}}]\) and \(\opnorm{U_{t}} \le \Gamma^{1/2} \le \sqrt{\frac{8\Gamma}{3}}\) by induction hypothesis, the maximum is taken when \(i = 1\), hence
  \begin{align}
    \label{eq:iterate_inside_main}
    \opnorm{U_{t+1}}
    & \le \opnorm{U_{t}} - \frac{1}{8\Gamma} \opnorm{U_{t}}^{3} + \frac1{8\Gamma}(\sigstarl + 2 r \Crme \sigstarl \sqrt{\epsilon}) \opnorm{U_{t}} \nonumber \\
    & \le \opnorm{U_{t}} - \frac{1}{8\Gamma} (\opnorm{U_{t}}^{2} - \sigstarl - 2 r \Crme \sigstarl \sqrt{\epsilon})\opnorm{U_{t}} \nonumber \\
    & \le  \opnorm{U_{t}} - \frac{1}{8\Gamma} (\opnorm{U_{t}}^{2} - 2\sigstarl)\opnorm{U_{t}},
  \end{align}
where the last inequality uses \(\epsilon = \bigO{\frac1{\kappa^{3}r^{3}}} = \bigO{\frac1{r^{2}}}\). We split into two cases now:

  \textbf{Case} \(\opnorm{U_t} > \frac12 \Gamma^{1/2}\). Recall that \(\Gamma \ge 36 \sigstarl\), hence
  \begin{align*}
    \opnorm{U_{t+1}}
    & \le  \opnorm{U_{t}} - \frac{1}{8\Gamma} (\opnorm{U_{t}}^{2} - 2\sigstarl)\opnorm{U_{t}} \\
    & \le  \opnorm{U_{t}} - \frac{1}{8\Gamma} (\frac{\Gamma}{4} - 2\sigstarl)\frac{\Gamma^{1/2}}{2} \\
    & \le \opnorm{U_{t}} - \frac{1}{8\Gamma} (\frac{\Gamma}{4} - \frac{\Gamma}{18})\frac{\Gamma^{1/2}}{2} \\
    & \le \opnorm{U_{t}} - \frac{1}{96}\Gamma^{1/2},
  \end{align*}
  where the second inequality is because \(t \mapsto (t^{2} - 2\sigstarl)t\) is increasing on \(\left[ \sqrt{\frac{2\sigstarl}{3}}, +\infty \right]\) and \(\opnorm{U} \ge \frac12 \Gamma^{1/2} \ge 3 \sigstarl\). This case captures the dissipativity condition satisfied by the matrix sensing problem, which says that the iterate aligns with the direction of the gradient when the iterate’s norm is large, so descending along the gradient decreases the norm of the iterate.

  \textbf{Case} \(\opnorm{U_t} \le \frac12 \Gamma^{1/2}\). From Equation~\eqref{eq:iterate_inside_main} we know that if \(\opnorm{U_{t}}^{2} \ge 2 \opnorm{M^{*}}\), it always holds that \(\opnorm{U_{t+1}} \le \opnorm{U_{t}} \le \frac12 \Gamma^{1/2}\). For \(\opnorm{U_{t}}^{2} \le 2 \opnorm{M^{*}}\), we have \(\opnorm{U_{t}} \le \sqrt 2 \sigstarl^{1/2}\), and therefore
  \begin{align*}
    \opnorm{U_{t+1}}
    & \le  \opnorm{U_{t}} - \frac{1}{8\Gamma} (\opnorm{U_{t}}^{2} - 2\sigstarl)\opnorm{U_{t}} \\
    & \le  \opnorm{U_{t}} + \frac{1}{8\Gamma} (\sqrt 2 \sigstarl^{1/2} + \opnorm{U_{t}} )((\sqrt 2\sigstarl^{1/2} -  \opnorm{U_{t}} ))\opnorm{U_{t}} \\
    & \le  \opnorm{U_{t}} + \frac{2\sqrt 2 \sigstarl^{1/2} }{8\Gamma} (\sqrt 2\sigstarl^{1/2} -  \opnorm{U_{t}} )\sqrt 2 \sigstarl^{1/2} \\
    & = \opnorm{U_{t}} + \frac{ \sigstarl}{2\Gamma} (\sqrt 2\sigstarl -  \opnorm{U_{t}} ) \\
    & \le \sqrt 2 \sigstarl^{1/2}  \text{\qquad (because the above expression is increasing in \(\opnorm{U_{t}}\))}\\
    & \le \frac{\sqrt 2}{6} \Gamma^{1/2} < \Gamma^{1/2},
  \end{align*}
  where in the second last inequality we use \(\sigstarl \le \Gamma / 36\).

  In summary, for all \(t \in \N\),
  \begin{align}
    \label{eq:iterate_inside}
    \opnorm{U_{t}} > \frac12 \Gamma^{1/2} & \implies \opnorm{U_{t} - \frac1{16\Gamma}\inexact g_{t}} \le \opnorm{U_{t}} - \frac1{96}\Gamma^{1/2} \\
    \opnorm{U_{t}} \le \frac12 \Gamma^{1/2} & \implies \opnorm{U_{t} - \frac1{16\Gamma} \inexact g_{t}} \le \frac12\Gamma^{1/2}
  \end{align}
  In particular, for  gradient step regime \(U_{t+1} := U_{t} - \frac1{16\Gamma} \tilde g_{t}\), we know the iterate \(U_{t+1}\) stays inside the region \(\{U: \opnorm{U}^{2} \le \Gamma\}\) provided that \(U_{t}\) is inside the region.

  We proceed to analyze  negative curvature steps, which only happen if the inexact gradient is small \(\fnorm{\tilde g_{t}} \le \epsilon_{g} = \frac1{32}\sigstarr^{3/2}\). Note that it follows from~Equation~\eqref{eq:iterate_inside} that
  \begin{align*}
    \fnorm{\tilde g_{t}} \le  \frac1{32}\sigstarr^{3/2}
    & \implies \opnorm{\tilde g_{t}} \le  \frac1{32}\sigstarr^{3/2} \le \frac1{32}\Gamma^{3/2} \\
    & \implies \opnorm{U_{t} - \frac1{16\Gamma} \tilde g_{t}} \ge \opnorm{U_{t}} - \frac1{512} \Gamma^{1/2} > \opnorm{U_{t}} - \frac1{96}\Gamma^{1/2}\\
    & \implies \opnorm{U_{t}} \le \frac12 \Gamma^{1/2}.
  \end{align*}
  Therefore, if at step \(\tau\) we run the inexact negative curvature step to update the iterate from \(U_{\tau}\) to \(U_{\tau + 1}\), then \(\opnorm{U_{\tau}} \le \frac12 \Gamma^{1/2}\). Recall in Equation~\eqref{eq:negative_curvature_stepsize} the negative curvature stepsize \(\frac{2 \epsilon_{H}}{L_{H}}  \le \frac{\Gamma^{1/2}}{36} < \frac{\Gamma^{1/2}}{2}\), so for the negative curvature update \[U_{\tau+1} = U_{\tau} + \frac{2 \epsilon_{H}}{L_{H}} \sigma_{\tau} \inexact p_{\tau}\] where \(\sigma_{\tau} = \pm 1\) with probability 1/2 and \(p_{\tau}\) is a vector with \(\fnorm{p_{\tau}} = 1\), we have
  \begin{align*}
    \opnorm{U_{\tau+1}} &\le \opnorm{U_{\tau}} + \opnorm{\frac{2 \epsilon_{H}}{L_H} \sigma_{\tau} \inexact p_{\tau}} \\
    & \le  \opnorm{U_{\tau}} + \fnorm{\frac{2 \epsilon_{H}}{L_H} \sigma_{\tau} \inexact p_{\tau}} < \frac12 \Gamma^{1/2} + \frac12 \Gamma^{1/2} = \Gamma^{1/2}.
  \end{align*}
  Finally, the initialization satisfies \(\opnorm{U_{0}}^{2} \le \Gamma\) by assumption, so the induction is complete.
\end{proof}

\subsection{Local Linear Convergence in Low Rank Matrix Sensing---Omitted Proofs}\label{sec:local-linear-convergence-appendix}
Recall that Theorem~\ref{thm:global_convergence} gives us a \((\frac{1}{24}\sigstarr^{3/2}, \frac{1}{3}\sigstarr)\)-approximate SOSP as the initialization for Algorithm~\ref{alg:local_linear_convergence}. We now explain why approximate SOSP is useful for the local search in the low rank matrix sensing problems.
\begin{definition}[Strict Saddle Property]
\label{def:strict_saddle}
  Function $f(\cdot)$ is a $(\epsilon_{g}, \epsilon_{H}, \zeta)$-{strict saddle} function if for any $U \in \R^{r \times d}$, at least one of following properties holds:
\begin{enumerate}[label=(\alph*)]
\item $\norm{\grad f(U)} > \epsilon_{g}$.
\item $\lambda_{\min}(\grad^{2} f(U)) < -\epsilon_{H}$.
\item $U$ is $\zeta$-close to $\cXstar$ --- the set of local minima; namely, $\dist(U, \cXstar) \leq \zeta$.
\end{enumerate}
\end{definition}
Approximate SOSPs are close to the set of local minima for strict saddle functions:
\begin{proposition}
\label{prop:strict_saddle}
    Let $f$ be a $(\epsilon_{g}, \epsilon_{H}, \zeta)$-{strict saddle} function and  $\cXstar$ be the set of its local minima. If $U_{SOSP}$ is a \((\epsilon_{g}, \epsilon_{H})\)-approximate SOSP of $f$, then $\dist(U_{SOSP}, \cXstar) \leq \zeta$.
\end{proposition}
\begin{proof}
    Since $U_{SOSP}$ is a \((\epsilon_{g}, \epsilon_{H})\)-approximate SOSP of $f$, we have $\norm{\grad f(U)} \le \epsilon_{g}$ and $\lambda_{\min}(\grad^{2} f(U)) \ge -\epsilon_{H}$, i.e.,\ (a) and (b) in Definition~\ref{def:strict_saddle} fail to hold. Therefore, (c) holds, i.e.,\ $\dist(U, \cXstar) \leq \zeta$.
\end{proof}

Some regularity conditions only hold in a local neighborhood of $\cXstar$. We focus on the following definition.
\begin{definition}[Local Regularity Condition]
  \label{def:regularity}
In a $\zeta$-neighborhood of the set of local minima $\cXstar$ (that is, $\dist(U, \cXstar) \leq \zeta$),
the function $f(\cdot)$ satisfies an $(\alpha, \beta)$-{regularity condition} if for any $U$ in this neighborhood:
\begin{equation}\label{eq:regularity}
\< \grad f(U), U - \projX(U) \> \ge \frac{\alpha}{2}\norm{U- \projX(U)}^2 + \frac{1}{2\beta} \norm{\grad f(U)}^2.
\end{equation}
\end{definition}
Local regularity condition is a weaker condition than strong convexity, but both
conditions would allow a first-order algorithm to obtain local linear
convergence. 

Recall that \(\bar f(U) = \frac12\fnorm{ UU^{\top} - M^{*} }^{2}\). The global
minima of \(\bar f\) solve the so-called symmetric low rank matrix factorization
problem and the corresponding function value is zero. Let \(TDT^{\top}\) be the
singular value decomposition of the real symmetric matrix \(M^{*}\). We observe
\(U^{*} = TD^{1/2}\) is a global optimum of \(\bar f\).

The function
\(\bar f\) satisfies the local regularity condition and the strict saddle
property, and all of its local minima are global minima~\cite{jin2017howto}.

\begin{fact}[\cite{jin2017howto}, Lemma 7]
\label{fact:matrix_factorization}
For \(\bar f\) defined in \eqref{eq:barf}, all local minima are global minima. The set of global minima is characterized by $\cXstar = \{U^* R |  R R^{\top}= R^{\top} R = I \}$. Furthermore, $\bar f(U)$ satisfies:
\begin{enumerate}
\item $(\frac{1}{24}(\sigstarr)^{3/2}, \frac{1}{3}\sigstarr, \frac{1}{3}(\sigstarr)^{1/2})$-strict saddle property, and
\item $(\frac{2}{3}\sigstarr, 10\sigstarl)$-regularity condition in the $\frac{1}{3}(\sigstarr)^{1/2}$-neighborhood of $\cXstar$ in Frobenius norm.
\end{enumerate}
\end{fact}

Similar to Corollary~\ref{cor:robust_mean_estimation_guarantee}, we need a theorem to say that all calls to \(\mathsf{RobustMeanEstimation}\) (Algorithm~\ref{alg:robust_mean_estimation}) are successful. For the local search (Algorithm~\ref{alg:local_linear_convergence}), only
inexact gradient oracles are required. Since gradients have dimension \(d r\),
the sample complexity needed for all robust estimates of gradient to be accurate with high probability is \(\tilde O\left( \frac{dr}{\epsilon}\right)\).

\begin{corollary}\label{cor:robust_mean_estimation_guarantee_gradient_only}
  Consider the noiseless setting as in Theorem~\ref{thm:low-rank-mtrx-sensing-intro} with $\sigma = 0$. There exists a sample size \(n = \tilde O\left( \frac{dr}{\epsilon}\right)\) such that with $n$ samples, all
  subroutine calls to \(\mathsf{RobustMeanEstimation}\)
  (Algorithm~\ref{alg:robust_mean_estimation}) to estimate gradients
  \(\grad \bar f(U_{t})\) succeed with high probability. In light of Equation~\eqref{eq:gradient_covariance_bound}, this means that there exists a large
  enough constant \(\Crme\) such that with high probability which is suppresed by \(\tilde O(\cdot)\), for all iterations \(t\), we have
  \begin{align}
    \fnorm{\tilde g_{t} - \grad \bar f(U_{t})}  \le 2 \Crme \fnorm{U_{t} U_{t}^{\top} - M^{*}} \opnorm{U_{t}}.
  \end{align}
\end{corollary}
\begin{proof}
  The proof follows exactly the same line of argument as Corollary~\ref{cor:robust_mean_estimation_guarantee} and Theorem~\ref{thm:global_convergence-general-appendix}, which discusses the sample complexity required for robust mean estimation of both gradients and Hessians to be successful. Since the Hessian has dimension $d^2 r^2$, the sample complexity in Corollary~\ref{cor:robust_mean_estimation_guarantee} is dominated by $\bigtOmega{d^2r^2/\epsilon}$. Here the gradient has dimension $d r$, so we have the required sample complexity $\bigtOmega{dr/\epsilon}$.
\end{proof}



\begin{proof}[Proof of Theorem~\ref{thm:local_linear_convergence}]
  Let \(e_{t}  = \tilde g_{t} - \grad \bar f(U_{t})\) and we rewrite the inexact gradient descent update as \[U_{t+1} = U_{t} - \eta (\grad \bar f(U_{t}) + e_{t} ).\]

  It follows from
  Corollary~\ref{cor:robust_mean_estimation_guarantee_gradient_only} that, with
  high probability and for some large constant \(\Crme\), all calls to
  \(\mathbf{RobustMeanEstimation}\) are successful, i.e.,\ for all iterations
  $t$,
  \begin{align}
    \label{eq:gradient_robust_mean_error}
    \fnorm{e_{t}}  \le 2 \Crme \Gamma^{1/2} \fnorm{U_{t} U_{t}^{\top} - M^{*}} \sqrt{\epsilon}.
  \end{align}
  We condition on this event for the rest of the analysis. 
  
  It is straightforward to check that \(\frac13 \sigstarr^{1/2}\)-neighborhood of \(\cXstar\) in Frobenius norm is inside the region \(\{U: \opnorm{U}^{2} \le \Gamma\}\). By strict saddle property in Fact~\ref{fact:matrix_factorization}, we know from Proposition~\ref{prop:strict_saddle} that \(U_{0} = U_{SOSP}\) is \(\frac13 \sigstarr^{1/2}\)-close to \(\cXstar\) in Frobenius norm. We assume for the sake of induction that  \(U_{t}\) is \(\frac13 \sigstarr^{1/2}\)-close to \(\cXstar\).

  Let \(\projX(U)\) be the Frobenius projection of \(U \in \R^{d \times r}\) onto \(\cXstar\). By Fact~\ref{fact:matrix_factorization}, for a given \(U\in \R^{d \times r}\), there exists a rotation \(R_{U}\) so that \(\projX(U) = U^{*} R_{U}\). Therefore, \(M^{*} = U^{*} {U^{*}}^{\top} = U^{*}R_{U} R_{U}^{\top}{U^{*}}^{\top} = \projX(U) \projX(U)^{\top}\). We have
  \begin{align*}
    U U^{\top} -  M^{*}
    & = U U^{\top} - U  \projX(U)^{\top} +  U \projX(U)^{\top} -  \projX(U)  \projX(U)^{\top} \\
    & = U (U - \projX(U) )^{\top} + (U - \projX(U) ) \projX(U)^{\top}.
  \end{align*}
  
  Since \(\opnorm{U_{t}} \le \Gamma^{1/2}\) and \(\opnorm{\projX(U)} \le \Gamma^{1/2}\), it follows from~\eqref{eq:gradient_robust_mean_error} that
  \[ \fnorm{e_{t}} \le 2 \Crme \Gamma^{1/2} \sqrt{\epsilon} \fnorm{U_{t} U_{t}^{\top} -  M^{*} } \le 4 \Crme \Gamma \sqrt{\epsilon} \fnorm{U_t - \projX(U_t)}.\]

  We proceed to derive the exponential decrease of the distance to global minima.
  \begin{align*}
    & \quad \fnorm{U_{t+1} - \projX(U_{t+1})}^{2}  \\
    &  \le \fnorm{U_{t+1} - \projX(U_{t})}^{2} \\
    & = \fnorm{U_{t} - \eta (\grad \bar f(U_{t}) + e_{t} ) - \projX(U_{t}) }^{2} \\
    & = \fnorm{U_{t} - \projX(U_{t})}^{2} + \eta^{2} \fnorm{\grad \bar f(U_{t})+e_{t}}^{2} - 2\eta \< \grad \bar f(U_{t}) + e_{t} , U_{t} - \projX(U_{t}) \> \\
    & \le (1 - \eta \alpha) \fnorm{U_{t} - \projX(U_{t})}^{2} + \eta^{2} \fnorm{\grad \bar f (U_{t}) + e_{t}}^{2} - \frac{\eta}{\beta} \fnorm{\grad \bar f(U_{t})}^{2} - 2\eta \< e_{t}, U_{t} - \projX(U_{t}) \>,
  \end{align*}
  where the last inequality comes from the regularity condition in Fact~\ref{fact:matrix_factorization} with \(\alpha = \frac23\sigstarr, \beta = 10\sigstarl\).
  
  The cross term can be bounded by
  \begin{align*}
    -2 \eta \< e_{t}, U_{t} - \projX(U_{t}) \> \le 2 \eta \fnorm{e_{t}} \fnorm{U_{t} - \projX(U_{t})} \le 8\Crme\eta\Gamma\sqrt\epsilon \fnorm{U_{t} - \projX(U_{t})}^{2},
  \end{align*}
  where the last inequality comes from Equation~\ref{eq:gradient_robust_mean_error}.
  
  By Young's inequality,
  \begin{align*}
     \eta^{2} \fnorm{\grad \bar f (U_{t}) + e_{t}}^{2} &\le 2 \eta^{2} \fnorm{\grad \bar f(U_{t})}^{2} + 2 \eta^{2} \fnorm{e_{t}}^{2} \\
     &\le 2 \eta^{2} \fnorm{\grad \bar f(U_{t})}^{2} +  32 \eta^{2} \Crme^{2} \Gamma^{2} \epsilon \fnorm{U_{t} - \projX(U_{t})}^{2}.
  \end{align*}
  
  Recall that \(\Gamma \ge 36 \sigstarl\). Choosing \(\eta =  \frac{1}{\Gamma}  \le \frac1{36\sigstarl} < \frac1{20\sigstarl} = \frac1{2\beta} \), it follows that
  \begin{align}
    \label{eq:local_linear_convergence}
    &\fnorm{U_{t+1} - \projX(U_{t+1})}^{2} \nonumber \\
    & \le (1 - \eta \alpha) \fnorm{U_{t} - \projX(U_{t})}^{2} + \eta^{2} \fnorm{\grad \bar f (U_{t}) + e_{t}}^{2} - \frac{\eta}{\beta} \fnorm{\grad \bar f(U_{t})}^{2} - 2\eta \< e_{t}, U_{t} - \projX(U_{t}) \> \nonumber \\
    & \le (1 - \eta \alpha + 8\Crme\eta\Gamma\sqrt\epsilon +  32 \eta^{2} \Crme^{2} \Gamma^{2} \epsilon ) \fnorm{U_{t} - \projX(U_{t})}^{2}+ (2 \eta^{2}  - \frac{\eta}{\beta}) \fnorm{\grad \bar f(U_{t})}^{2} \nonumber \\
    & \le \left(1 - \frac {\sigstarr} {30\Gamma} + \frac25 \Crme\sqrt{\epsilon} +  32(\Crme/20)^{2}\epsilon\right) \fnorm{U_{t} - \projX(U_{t})}^{2} \nonumber \\
    &  \le \left(1 - \left(O\left(\frac{1}{\kappa} \right) - O(\sqrt \epsilon)\right)\right)\fnorm{U_{t} - \projX(U_{t})}^{2}.
  \end{align}
  
  One consequence of this calculation is \(\fnorm{U_{t+1} - \projX(U_{t+1})} \le \fnorm{U_{t} - \projX(U_{t})}\), so \(U_{t+1}\) is also \(\frac13 \sigstarr^{1/2}\)-close to \(\cXstar\) in Frobenius norm, completing the induction. The other consequence is local linear convergence of \(U_{t}\) to \(\cXstar\) with rate \(\left(1 - O\left(\frac1\kappa\right)\right)\), assuming \(\epsilon \kappa^{2}\) is sufficiently small. Since the initial distance is bounded by \(O(\sigstarr^{1/2})\), converging to a point that is \(\iota\)-close to \(\cXstar\) in Frobenius norm requires the following number of iterations:
  \[ O\left(\frac{\log(\iota/\sigstarr^{1/2})}{\log(1- O(1/\kappa))}\right) = O\left(\kappa\log\left( \frac\sigstarr{\iota}\right)\right).  \qedhere\]
  \end{proof}

\subsection{Low Rank Matrix Sensing with Noise}
\label{sec:sensing_noise}

In Section~\ref{sec:matrix-low-rank}, we focused on the case that \(\sigma = 0\) as in Theorem~\ref{thm:together}. Now we consider the case that \(\sigma \neq 0\) and prove Theorem~\ref{thm:noisy_combined}.

Recall that $y_i \sim \normal(\<A_{i}, M^{*}\>, \sigma^2)$ in Definition~\ref{def:iid_Gaussian_with_noise}. Since \(\<A_{i}, M^{*}\>\) follows Gaussian distribution with mean 0 and variance \(\fnorm{M^{*}}^{2}\), we optionally assume \(\sigma = \bigO{\fnorm{M^{*}}} = \bigO{r \Gamma}\) to keep the signal-to-noise ratio in constant order. For the ease of presentation, we make the following assumption:
\begin{assumption}
\label{assump:noise_level}
  Assume \(\sigma \le r \Gamma\).
\end{assumption}

We present different algorithms depending on whether Assumption~\ref{assump:noise_level} holds.

As in the noiseless case, we define \(f_{i}(U) = \frac12\left( \< U U^{\top}, A_{i}\> - y_{i}\right )^{2}\) and \[\bar f(U) = \E_{(A_{i}, y_{i}) \sim \G_\sigma}f_{i}(U) = \frac12\fnorm{U U^{\top} - M^{*}}^{2} + \frac12 \sigma^{2}.\]
Note that the \(\frac12 \sigma^{2}\) term in \(\bar f(U)\) has no effect on its minimum, gradients, or Hessians, and Fact~\ref{fact:matrix_factorization} and~\ref{fact:matrix_factorization_constants} still apply in verbatim.

We prove the following result when Assumption~\ref{assump:noise_level} holds:
\begin{theorem}
  Consider the same setting as in Theorem~\ref{thm:low-rank-mtrx-sensing-intro} with $0 \neq \sigma \le r \Gamma$ (Assumption~\ref{assump:noise_level} holds). There exists a sample size \(n = \bigtO[\big]{{d^{2}r^{2}}/{\epsilon}}\) such that with high probability, there exists an algorithm that outputs a solution \(\hat M\) in $ \tilde O( r^{2}\kappa^{3} )$ calls to robust mean estimation routine~\ref{alg:robust_mean_estimation}, with error \(\fnorm[\big]{\hat M - M^{*}} = \bigO{\kappa \sigma\sqrt\epsilon}\).
\end{theorem}
\begin{proof}
We compute the gradient and Hessian of \(f_{i}\):
\begin{align*}
   \grad f_{i}(U) &= \left(\< U U^{\top} - M^{*}, A_{i} \> - \zeta_{i}\right)  (A_{i} + A_{i}^{\top}) U \\
  H_{i}(U) &= \left(\< U U ^{\top} - M^{*}, A_{i} \> - \zeta_{i}\right) I_{r} \otimes B_{i}  + \vect \left(B_{i}U  \right) \vect(B_{i}U)^{\top}.
\end{align*}

We also need to bound the covariance of sample gradients and sample Hessians in the bounded region \(\{U \in \R^{d \times r}: \opnorm{U}^{2} \le \Gamma\}\), similar to Equations~\eqref{eq:gradient_covariance_bound} and~\eqref{eq:hessian_covariance_bound}.
\begin{align}
  \label{eq:gradient_covariance_bound_noise}
  \opnorm{\Cov(\vect( \grad f_{i}(U)))} &\le 4 \left(\fnorm{U U^{\top} - M^{*}}^{2} + \sigma^{2}\right) \opnorm{U}^{2} = \bigO{r^{2}\Gamma^{3}}. \\
  \label{eq:hessian_covariance_bound_noise}
  \opnorm{\Cov(\vect(H_{i}))} &\le  16r \left(\fnorm{ U U ^{\top} - M^{*} }^{2} + \sigma^{2}\right)  + 128 \opnorm{U}^{4} = \bigO{r^{3}\Gamma^{2}}.
\end{align}

Since global convergence analysis in Section~\ref{sec:global_convergence} only requires \(\opnorm{\Cov(\vect( \grad f_{i}(U)))} = \bigO{r^{2}\Gamma^{3}}\) and \(\opnorm{\Cov(\vect(H_{i}))} = \bigO{r^{3}\Gamma^{2}}\), Theorem~\ref{thm:global_convergence} also holds in the noisy setting (up to a constant factor). This means that we could obtain a \((\frac1{24}\sigstarr^{3/2}, \frac13 \sigstarr)\)-approximate SOSP of \(\bar f\), which is in \(\frac13 \sigstarr^{1/2}\)-neighborhood of \(\cXstar\) in Frobenius norm.

We now focus on the local convergence analysis. We use the same algorithm (Algorithm~\ref{alg:local_linear_convergence}), which uses \(\mathsf{RobustMeanEstimation}\) subroutine in each iteration \(t\) to obtain an inexact gradient \(\tilde g_{t}\) for \(\bar f(U)\). We rewrite the inexact gradient descent update as \[U_{t+1} = U_{t} - \eta (\grad \bar f(U_{t}) + e_{t} ).\]
We condition the remaining analysis on the event that all calls to \(\mathbf{RobustMeanEstimation}\) are successful as in Corollary~\ref{cor:robust_mean_estimation_guarantee_gradient_only}, which happens with high probability and implies
\begin{align*}
  \fnorm{e_{t}}
  & \le 2 \Crme \left(\fnorm{U_{t} U_{t}^{\top} - M^{*}} + \sigma\right) \opnorm{U_{t}} \sqrt\epsilon \\
  &\le 2 \Crme \left(\fnorm{U_{t} U_{t}^{\top} - M^{*}} + \sigma\right) \Gamma^{1/2} \sqrt\epsilon.
\end{align*}

It is straightforward to check that \(\frac13 \sigstarr^{1/2}\)-neighborhood of \(\cXstar\) in Frobenius norm is inside the region \(\{U: \opnorm{U}^{2} \le \Gamma\}\). By strict saddle property in Fact~\ref{fact:matrix_factorization}, we know \(U_{0} = U_{SOSP}\) is \(\frac13 \sigstarr^{1/2}\)-close to \(\cXstar\) in Frobenius norm. We assume for the sake of induction that  \(U_{t}\) is \(\frac13 \sigstarr^{1/2}\)-close to \(\cXstar\). We will show that \(\fnorm{U_{t+1} - \projX(U_{t+1})} \) either shrinks or is at the order of \(\bigO{\frac{\sigma\epsilon}{\Gamma^{1/2}}}\), which is also \(\frac13 \sigstarr^{1/2}\)-close to \(\cXstar\) in Frobenius norm for sufficiently small \(\epsilon = \bigO{\frac1{r\kappa}}\).

As before, we define \(\projX(\cdot)\) to be the projection onto \(\cXstar\) and bound the Frobenius distance to \(\cXstar\) as follows:
\begin{align*}
  & \quad \fnorm{U_{t+1} - \projX(U_{t+1})}^{2}  \\
  &  \le \fnorm{U_{t+1} - \projX(U_{t})}^{2} \\
  & = \fnorm{U_{t} - \eta (\grad \bar f(U_{t}) + e_{t} ) - \projX(U_{t}) }^{2} \\
  & = \fnorm{U_{t} - \projX(U_{t})}^{2} + \eta^{2} \fnorm{\grad \bar f(U_{t})+e_{t}}^{2} - 2\eta \< \grad \bar f(U_{t}) + e_{t} , U_{t} - \projX(U_{t}) \> \\
  & \le (1 - \eta \alpha) \fnorm{U_{t} - \projX(U_{t})}^{2} + \eta^{2} \fnorm{\grad \bar f (U_{t}) + e_{t}}^{2} - \frac{\eta}{\beta} \fnorm{\grad \bar f(U_{t})}^{2} - 2\eta \< e_{t}, U_{t} - \projX(U_{t}) \>,
\end{align*}
where the last inequality comes from the regularity condition in Fact~\ref{fact:matrix_factorization} with \(\alpha = \frac23\sigstarr, \beta = 10\sigstarl\).

Using Cauchy-Schwarz inequality to bound \(\fnorm{\grad \bar f (U_{t}) + e_{t}}^{2}\) and \(\< e_{t}, U_{t} - \projX(U_{t}) \>\) and choosing  \(\eta =  \frac{1}{20\Gamma}  \le \frac1{20\sigstarl} = \frac1{2\beta}\) to kill the \(\fnorm{\grad \bar f (U_{t})}^{2}\) term, we obtain
\begin{align*}
  &\fnorm{U_{t+1} - \projX(U_{t+1})}^{2} \nonumber \\
  & \le (1 - \eta \alpha) \fnorm{U_{t} - \projX(U_{t})}^{2} + 2\eta^{2} \fnorm{\grad \bar f (U_{t})}^{2} + 2\fnorm{e_{t}}^{2} - \frac{\eta}{\beta} \fnorm{\grad \bar f(U_{t})}^{2} - 2\eta \< e_{t}, U_{t} - \projX(U_{t}) \> \nonumber \\
  & \le (1 - \eta \alpha) \fnorm{U_{t} - \projX(U_{t})}^{2}  + 2\eta^{2}\fnorm{e_{t}}^{2} - 2\eta \< e_{t}, U_{t} - \projX(U_{t}) \> + (2 \eta^{2}  - \frac{\eta}{\beta}) \fnorm{\grad \bar f(U_{t})}^{2}\\
  & \le (1 - \eta \alpha) \fnorm{U_{t} - \projX(U_{t})}^{2}  + 2\eta^{2}\fnorm{e_{t}}^{2} + 2\eta \fnorm{e_{t}} \fnorm{U_{t} - \projX(U_{t})}.
\end{align*}

Now we consider the following two possible cases:

\textbf{Case 1:}  \(\fnorm{U_{t} U_{t}^{\top} - M^{*}} > \sigma\). Then
\begin{align*}
  \fnorm{e_{t}}
  & \le  2 \Crme \left(\fnorm{U_{t} U_{t}^{\top} - M^{*}} + \sigma\right) \Gamma^{1/2} \sqrt\epsilon \\
  & < 4 \Crme \fnorm{U_{t} U_{t}^{\top} - M^{*}} \Gamma^{1/2} \sqrt\epsilon.
\end{align*}
And it follows similar to~\eqref{eq:local_linear_convergence} that
\begin{align*}
  &\fnorm{U_{t+1} - \projX(U_{t+1})}^{2}  \\
    & \le (1 - \eta \alpha + 8\Crme\eta\Gamma\sqrt\epsilon +  32 \eta^{2} \Crme^{2} \Gamma^{2} \epsilon ) \fnorm{U_{t} - \projX(U_{t})}^{2} \\
  & \le \left(1 - \frac {\sigstarr} {30\Gamma} + \frac25 \Crme\sqrt{\epsilon} +  32(\Crme/20)^{2}\epsilon\right) \fnorm{U_{t} - \projX(U_{t})}^{2} \nonumber \\
  &  \le \left(1 - \left(O\left(\frac{1}{\kappa} \right) - O(\sqrt \epsilon)\right)\right)\fnorm{U_{t} - \projX(U_{t})}^{2}.
\end{align*}
We conclude that the distance to \(\cXstar\) decreases geometrically in this regime until \(\fnorm{U_{t} U_{t}^{\top} - M^{*}} \le \sigma\), which takes a total of \(\bigO{\kappa \log\left( \frac{\sigstarr}{\sigma}\right )}\) iterations, which is very few if \(\sigma = \Theta\left(r \Gamma \right)\).

\textbf{Case 2:} \(\fnorm{U_{t} U_{t}^{\top} - M^{*}}  \le \sigma\). Then \(\fnorm{e_{t}} \le 4 \Crme \sigma \Gamma^{1/2} \sqrt\epsilon\). It follows that
\begin{align*}
  &\fnorm{U_{t+1} - \projX(U_{t+1})}^{2}  \\
  & \le (1 - \eta \alpha) \fnorm{U_{t} - \projX(U_{t})}^{2}  + 2\eta^{2}\fnorm{e_{t}}^{2} + 2\eta \fnorm{e_{t}} \fnorm{U_{t} - \projX(U_{t})} \\
  & \le  (1 - \eta \alpha) \fnorm{U_{t} - \projX(U_{t})}^{2}  + 32 \eta^{2}  \Crme^{2} \sigma^{2} \Gamma \epsilon  + 8 \eta  \Crme \sigma \Gamma^{1/2} \sqrt\epsilon \fnorm{U_{t} - \projX(U_{t})}.
\end{align*}

Write \(\C = 4 \sqrt{2}\eta  \Crme \sigma \Gamma^{1/2} \sqrt\epsilon\), \(\B_{t} = \fnorm{U_{t+1} - \projX(U_{t+1})} / \C\), we have
\[\B_{t+1} \le \sqrt{(1 - \eta\alpha) \B_{t}^{2} + \B_{t} + 1}. \]

It is straightforward to compute that the function \( b \mapsto \sqrt{(1 - \eta\alpha) b^{2} + b + 1}\) has Lipschitz constant \[ \rho =
\begin{cases}
  \frac12 & \mbox{if } 0 < 1-\eta\alpha \le \frac14\\
  \sqrt{1-\eta\alpha } & \mbox{if } \frac14 \le 1-\eta\alpha < 1.
\end{cases}
\]

Since \(\rho < 1\), Banach fixed-point theorem implies that \(\B_{t}\) converges to some \(\B^{*}\) with rate
\[\abs{\B_{t+1} - \B^{*}} \le {\rho}\abs{\B_{t} - \B^{*}}.\]  We can compute an upper bound on the fixed point \(\B^{*}\) by solving \[ \B^* =\sqrt{(1 - \eta\alpha) {\B^*}^{2} + \B^* + 1},\] which gives \(\B^{*} = \bigO{\frac1{\eta\alpha}} = \bigO{\kappa}\).

Denote the first iteration of Case 2 regime (\(\fnorm{U_{t} U_{t}^{\top} - M^{*}} \le \sigma \)) by \(t_{0}\). By~\cite[Lemma 6]{ge2017no-spurious-loc}, we have \(\fnorm{U_{t_{0}} - \projX{U_{t_{0}}}} \le \sigma\sqrt{\frac2{\sigstarr}}\) and therefore \(\B_{t_{0}} = \sqrt{\kappa/\epsilon}\). Therefore, Case 2 takes a total of \(\bigO{\kappa\log{1/\epsilon}}\) iterations to reach the error
\[ \fnorm{U_{t} - \projX(U_{t})} \le \bigO{ \kappa  \Gamma^{-1/2}\sigma \sqrt\epsilon}.\]
This translates to the error in the matrix space
\[ \fnorm{U_{t}U_{t}^{\top} - M^*} \le \Gamma^{1/2}\fnorm{U_{t} - \projX(U_{t})} = \bigO{ \kappa \sigma \sqrt\epsilon}.\]
\end{proof}

Assumption~\ref{assump:noise_level} might fail to hold; a notable example is our construction in Section~\ref{sec:sq-lower-bound}, where \(\fnorm{M^{*}} = \bigO{\epsilon^{1/2}}\) but \(\sigma = O(1)\). In this case, we consider a simple spectral method, relying on the observation that with clean samples \(\E_{(A_{i}, y_{i}) \in \G_{\sigma}}[y_{i}A_{i}] = M^{*}\).  So we run \(\mathsf{RobustMeanEstimation}\) on \(\{y_{i}A_{i}\}_{i = 1}^{n}\) to get \(\inexact M\) with its singular decomposition \(\inexact M  = \sum_{i=1}^{n}s_{i}u_{i}v_{i}^{\top}\), where singular values \(s_{1} \ge s_{2} \ge \dots \ge s_{n}\) are descending. We then form a new matrix \(\hat M\) from the \(r\) leading singular vectors, i.e., \(\hat M = \sum_{i=1}^{r}s_{i}u_{i}v_{i}^{\top}\).

This simple algorithm has the following guarantee:
\begin{theorem}
  \label{thm:together_high_noise}
  Consider the same setting as in Theorem~\ref{thm:low-rank-mtrx-sensing-intro} with $\sigma > r \Gamma$, i.e., Assumption~\ref{assump:noise_level} fails. There exists a sample size \(n = \bigtO[\big]{{d^{2}}/{\epsilon}}\) such that with high probability, there exists an algorithm that outputs a solution \(\hat M\) in one call to robust mean estimation routine~\ref{alg:robust_mean_estimation}, with error \(\fnorm[\big]{\hat M - M^{*}} = \bigO{\sigma\sqrt\epsilon}\).
\end{theorem}
\begin{proof}
  We compute the mean and variance of \(\<y_{i}A_{i}, Z\>\) for some unit Frobenius norm \(Z\).
  \begin{align*}
    \E \<y_{i}A_{i}, Z\>
    & = \E \<A_{i}, M^{*}\> \<A_{i}, Z\> + \E \zeta \<A_{i}, Z\> \\
    & = \E \<A_{i}, M^{*}\> \<A_{i}, Z\> = \<M^{*}, Z\>.
  \end{align*}
  Hence \(\E[y_{i}A_{i}] = M^{*}\).

  \begin{align*}
    & \Var\left( \<A_{i}, M^{*}\> \<A_{i}, Z\> +  \zeta \<A_{i}, Z\>\right) \\
    & \le 2 \Var \<A_{i}, M^{*}\> \<A_{i}, Z\> + 2 \Var( \zeta \<A_{i}, Z\>) \\
    & \le 4 \fnorm{M^{*}}^{2} + 4 \sigma^{2} \le 8 \sigma^{2},
  \end{align*}
  where the last inequality is because \(\fnorm{M^{*}} \le r \Gamma\). Hence running robust mean estimation algorithm on \(\{y_{i}A_{i}\}\) outputs \(\inexact M\) with \(\fnorm{\inexact M - M^{*}} = \bigO{\sigma\sqrt\epsilon}\).

  Since \(\hat M\) is formed by the \(r\) leading singular vectors of \(\inexact M\), it is the best rank \(r\) approximation to \(\inexact M\) by Eckart–Young–Mirsky's theorem, i.e.,\ \(\fnorm{\hat M - \inexact M} \le \fnorm{M^{*} - \inexact M} = \bigO{\sigma\sqrt\epsilon}\). We conclude that \(\fnorm{\hat M - M^{*}} \le \fnorm{\hat M - \inexact M} + \fnorm{\inexact M - M^{*}} =  \bigO{\sigma\sqrt\epsilon}\) by triangle inequality.
\end{proof}

\section{SQ Lower bound---Omitted Proofs}
\label{sec:sq-lower-bound-appendix}
We consider a weaker corruption model, known as Huber's \(\epsilon\)-contamination model~\cite{huber1964}, than Definition~\ref{def:corruption} in the sense that any corruptions in Huber's contamination model can be emulated by the adversary in Definition~\ref{def:corruption}. We show a lower bound of the sample complexity in the presence of Huber's \(\epsilon\)-contamination, and the sample complexity under Definition~\ref{def:corruption} can therefore be no better.

\begin{assumption}[Huber's Contamination Model]
  Let \(Q\) be the joint distribution of \((A, y)\) with \(\vect(A) \sim \normal(0, I_{d^{2}})\) and \(y | A \sim \normal(\<M^{*}, u u^{\top}\>, \sigma^{2})\). The input samples are contaminated according to Huber's \(\epsilon\)-contamination model in the following way: we observe samples from a mixture \(Q'\)  of the clean joint distribution \(Q\) and an arbitrary noise distribution \(N\), i.e., \(Q' = (1 - \epsilon) Q + \epsilon N\).
\end{assumption}
The adversary in the $\epsilon$-strong contamination model can emulate the Huber's adversary: on input clean samples, the $\epsilon$-strong contamination adversary removes a randomly chosen $\epsilon$-fraction of samples and replaces them with samples drawn i.i.d.\ from $N$.

Recall that SQ algorithms are a class of algorithms
that, instead of access to samples, 
are allowed to query expectations of bounded functions of the distribution. 

\begin{definition}[SQ Algorithms and \(\mathsf{STAT}\) Oracle~\cite{Kearns:98}]
Let \(\G\) be a distribution on \(\R^{d^{2} + 1}\).  A Statistical Query (SQ) is a bounded function \(q: \R^{d^{2} + 1} \ra [-1, 1]\). For \(\tau > 0\), the \(\mathsf{STAT}(\tau)\) oracle responds to the query \(q\) with a value \(v\) such that \(\abs{v - \E_{X \sim D}[q(X)]} \le \tau\). An SQ algorithm is an algorithm whose objective is to learn some information about an unknown distribution \(\G\) by making adaptive calls to the corresponding \(\mathsf{STAT}(\tau)\) oracle.
\end{definition}

\new{A statistical query with accuracy \(\tau\) can be implemented with error probability $\delta$ by taking $O(\log(1/\delta)/\tau^2)$ samples and evaluating the empirical mean of the query function $q(\cdot)$ evaluated at those samples~\cite[Chapter 8.2.1]{diakonikolas2022algorithmic}. This bound is tight in general for a single query.}

\begin{lemma}[Clean marginal and conditional distribution]
  Let \(A \in \R^{d \times d}\) with \(\vect(A) \sim \normal(0, I_{d^{2}})\) and \(y | A \sim \normal(\<u u^{\top}, M^{*}\>, \sigma^{2})\). Then
\begin{align*}
  y & \sim \mathcal{N}(0, \sigma^{2} + \norm{u}^4) \\
  \vect(A) | y & \sim \mathcal{N}\left(\frac{\vect(u u^{\top}) y}{\sigma^{2} + \norm{u}^4}, I_{d^{2}} - \frac{\vect(u u^{\top})\vect(u u^{\top})^{\top}}{\sigma^{2} + \norm{u}^4}\right).
\end{align*}
\end{lemma}
\begin{proof}
  Let \(Q\) be the joint distribution of \((\vect(A), y)\). We can write down the covariance of \(Q\):\[
  \begin{bmatrix}
    I_{d^{2}} & \vect(u u^{\top}) \\
    \vect(u u^{\top})^{\top} & \sigma^{2} + \norm{u}^{4}
  \end{bmatrix}.\]
The mean and variance of \(\vect(A) | y\) follow from the conditional Gaussian formula.
\end{proof}

We will construct a family of contaminated joint distributions of $(A, y) $ that are hard to
learn with SQ access. We first construct a family of clean distributions and then construct the contaminations.

\begin{lemma}[Defining Clean Distributions]
  \label{lemma:clean_distribution}
  Let \(v\) be a unit vector and set \(u = c_{1}^{1/2}\epsilon^{1/4} v\) for some small constant \(c_{1}\). Let \(\vect(A) \sim \normal(0, I_{d^{2}})\), and choose \(\sigma =  1 - c_{1}^{2}\epsilon\). Then we have the marginal distribution \(y \sim \normal(0,1)\) and
  \[ \vect(A) | y \sim \normal\left( c_{1}\sqrt\epsilon \vect(v v^{\top}) y, I_{d^{2}} - c_{1}^{2}\epsilon \vect(v v^{\top})\vect(v v^{\top})^{\top}  \right). \]
\end{lemma}

Let \(Q_{v}\) denote the joint distribution from the construction in
Lemma~\ref{lemma:clean_distribution}. The family of contaminated joint distributions
that are hard to learn will be denoted by \(Q_{v}'\), and we proceed to construct them. The technique for the construction follows a standard non-Gaussian component analysis argument, which starts from mixing a one-dimensional distribution to match moments with the standard Gaussian. We note that the conditional distribution \(\vect(A) | y\)  in all directions perpendicular to \(\vect(v v^{\top})\) is already standard Gaussian, so we proceed to consider the one-dimensional distribution along \(\vect(v v^{\top})\) direction, which is
\begin{align*}
   \normal \left(c_{1} \sqrt\epsilon y, 1 -c_{1}^2 \epsilon\right).
\end{align*}

Let \(\mu(y) = c_{1} \sqrt\epsilon y\) be a shortand for the mean of this one dimensional distribution. We proceed to mix it with some noise distribution \(B_{\mu}\). It is important to note here that the fraction of the noise depends on \(\mu\) and thus on \(y\)---the larger \(\mu\) is, the higher the noise level is needed.

\begin{definition}
  Let \(P, Q\) be two distributions with probability density functions \(P(x), Q(x)\); we obscure the line between a distribution and its density function. The \(\chi^{2}\)-divergence of \(P\) and \(Q\) is
  \[\chi^{2}(P, Q) = \int \frac{(P(x) - Q(x))^{2}}{Q(x)} dx = \int \frac{P^{2}(x)}{Q(x)} dx - 1.\]

  Let \(N\) be a distribution whose support contains the supports of \(P\) and \(Q\). We define
  \[\chi_{N}(P, Q) = \int \frac{P(x) Q(x)}{N(x)} dx. \]
  Then \(\abs{\chi_{N}(P, Q) - 1} = \int \frac{(P(x) - N(x))(Q(x) - N(x))}{N(x)} dx \) defines an inner product of \(P\) and \(Q\).
\end{definition}

\begin{lemma}[One-Dimensional Moment Matching]
  \label{lemma:mixing_one_dimensional}
  For any \(\epsilon > 0, \mu \in \R\), there is a distribution \(D_{\mu}\) such that \(D_{\mu}\) agrees with the mean of \(\mathcal N(0,1)\) and \[D_{\mu} = (1 - \epsilon_{\mu}) \mathcal N(\mu,  1 -c_{1}^2 \epsilon) + \epsilon_{\mu} B_{\mu},\]
  for some distribution \(B_{\mu}\) and \(\epsilon_{\mu}\) satisfying
  \begin{enumerate}
    \item If \(\abs{\mu} < c_{1}^{2}\sqrt{ \epsilon}\), then \(\epsilon_{\mu} = \epsilon\) and \(\chi^{2}(D_{\mu}, \mathcal N(0,1)) = e^{O(1/\epsilon)}\).

    \item If \(\abs{\mu} \ge c_{1}^{2}\sqrt{ \epsilon}\), then we take \(\epsilon_{\mu}\) such that \(\epsilon_{\mu} / ( 1 - \epsilon_{\mu} ) = \mu^{2}\), which simplies to
  \begin{equation}
    \label{eq:noise_y}
    \epsilon_{\mu} = \frac{\mu^{2}}{1 + \mu^{2}},
  \end{equation}
  and \(\chi^{2}(D_{\mu}, \mathcal N(0,1)) = e^{O(\max(1/\mu^{2}, \mu^{2}))}\).
  \end{enumerate}
\end{lemma}

\begin{proof}
  It is straightforward to verify that
  \(D_{\mu} = (1 - \epsilon_{\mu}) \mathcal N(\mu,  1 -c_{1}^2 \epsilon) + \epsilon_{\mu} \normal(a, b)\) has mean 0 and second moment 1, where \(a\) and \(b\) are defined by
  \[a = - \frac{\mu(1 - \epsilon_{\mu})}{\epsilon_{\mu}}, \;
    b = 1.\]

  Facts~\ref{fact:chi_square_mixture} and~\ref{fact:chi_square_gaussians} imply that if \(\sigma_{1} = O(1), \sigma_{2} = O(1)\), then
  \[\chi^{2}(\epsilon_{\mu} \normal(\mu_{1}, \sigma_{1}^{2}) + (1-\epsilon_{\mu})\normal(\mu_{2}, \sigma_{2}^{2}),\, \normal(0,1)) = e^{\bigO{\mu_{1}^{2} + \mu_{2}^{2}}}.\]
  \begin{itemize}
    \item \textbf{Case 1}: \(\abs{\mu} <c_{1}^{2} \sqrt{ \epsilon}\). Then taking \(\epsilon_{\mu} = \epsilon\) implies that \(|a| \le \mu / \epsilon < c_{1} /\sqrt{\epsilon}\), and \(\chi^{2}(D_{\mu}, \normal(0,1)) = e^{O(1/\epsilon)}.\)
    \item \textbf{Case 2}: \(\abs{\mu} \ge c_{1}^{2}\sqrt{ \epsilon}\). Then taking \(\epsilon_{\mu} / ( 1 - \epsilon_{\mu} ) = \mu^{2}\) implies that \(\abs{a} = \frac{1}{|\mu|}\),  and \(\chi^{2}(D_{\mu}, \mathcal N(0,1)) = e^{O(\max(1/\mu^{2}, \mu^{2}))}\).\qedhere
  \end{itemize}
\end{proof}
We follow the non-Gaussian component analysis construction and define the conditional distribution of the corrupted \(\vect(A) | y\) to have a hidden direction \(v\).
\begin{definition}[{High-Dimensional Hidden Direction Distribution~\cite[Definition 8.9]{diakonikolas2022algorithmic}}]
  For a one-dimensional distribution \(D\) that admits density \(D(x)\) and a unit vector \(v \in \R^{d^{2}}\), the distribution \(\mathbf{P}_{v}^{A}\) is defined to be the product distribution of \(D(x)\) in the \(v\)-direction and standard normal distribution in the subspace perpendicular to \(v\), i.e.,
  \[ \mathbf{P}_{v}^{D}(a) = \frac{1}{(2\pi)^{\frac{d^{2}- 1}2}} D(v^{\top} a) \exp\left(-\frac12 \norm{a - (v^{\top}a) v}^{2}\right ). \]
\end{definition}

\begin{lemma}[Construction of Corrupted Conditional Distribution]\label{lemma:corrupted-conditional}
  Define \(P_{\mu(y), v} = \mathbf{P}_{\vect(v v^{\top})}^{D_{\mu(y)}}\). Then for all unit vectors \(v, v' \in \R^{d}\), it holds that
  \begin{equation}
    \label{eq:correlation_lemma}
    \abs{\chi_{\normal(0, I)}(P_{\mu(y), v}, P_{\mu(y), v'}) - 1} \le {(v^{\top} v')}^{4}\, \chi^{2}(D_{\mu(y)}, \mathcal N(0,1)).
  \end{equation}
\end{lemma}
\begin{proof}
  By Lemma 8.12 in the textbook~\cite{diakonikolas2022algorithmic}, we have
\begin{align*}
  \abs{\chi_{\normal(0, I)}(P_{\mu(y), v}, P_{\mu(y), v'}) - 1}
  & = \abs{\chi_{\normal(0, I)}(\mathbf{P}_{\vect(v v^{\top})}^{D_{\mu(y)}}, \mathbf{P}_{\vect(v' {v'}^{\top})}^{D_{\mu(y)}}) - 1} \nonumber \\
  & \le (\vect(v v^{\top})^{\top} \vect(v' {v'}^{\top}))^{2}\, \chi^{2}(D_{\mu(y)}, \mathcal N(0,1)) \nonumber \\
  &  = {(v^{\top} v')}^{4}\, \chi^{2}(D_{\mu(y)}, \mathcal N(0,1)),
\end{align*}
where the inequality is a direct application of~\cite[Lemma 8.12]{diakonikolas2022algorithmic}, and the last equality is because
\begin{align*}
  \vect(v v^{\top})^{\top} \vect(v' {v'}^{\top})
   = \< v v^{\top}, v' {v'}^{\top}\>
   = \tr(v v^{\top}  v' {v'}^{\top}) = v^{\top}  v' \tr(v {v'}^{\top}) = (v^{\top}  v')^{2}.
\end{align*}
\end{proof}

Now we define \(Q_{v}'(A, y)\) to be a contaminated version of the joint distribution of \(A\) and \(y\) such that \(\vect(A) | y \sim P_{\mu(y), v}\) in the following lemma. Recall that, according to Lemma~\ref{lemma:mixing_one_dimensional}, different level of noise \(\epsilon_{\mu(y)}\) is needed for different \(y\), and the implication of the following lemma is that the total noise is less than \(\epsilon\). This is done via integration with respect to \(y\).

\begin{lemma}[Construction of Corrupted Joint Distribution: Controlling Large \(\epsilon_{\mu(y)}\)]
  Recall \(\mu(y) = c_{1} \sqrt\epsilon y\) and \(Q_{v}\) is the joint distribution of \(A\) and \(y\) for clean samples. Define \(Q_{v}'(A, y) = P_{\mu(y), v}(A) R(y)\), where
  \begin{align*}
    R(y) &= \frac{G(y)}{(1 - \epsilon_{\mu(y)}) \int G(y')/(1-\epsilon_{\mu(y')}) d y'}, \\
    G(y) & \mbox{ is the marginal distribution of \(y\) under \(Q_{v}\), which is standard Gaussian}.
  \end{align*}
  Then \(Q_{v}'(A, y)\) is the contaminated  joint distribution of \(A, y\) i.e.,\ \(Q_{v}' = (1 - \epsilon) Q_{v} + \epsilon N_{v}\) for some noise distribution \(N_{v}\). Under \(Q_{v}'\), it holds that \(A | y \sim P_{\mu(y), v}\). 
\end{lemma}

\begin{proof}
  We proceed to bound \(\int G(y')/(1-\epsilon_{\mu(y')}) d y' \le 1/(1-\epsilon)\). This shows that \(R(y)\) is well-defined, and we will use that to show \(Q_{v}' \ge (1 - \epsilon) Q_{v}\) (recall that we obscure the line between the distribution and its probability density function).
  \begin{align*}
    \int G(y')/(1-\epsilon_{\mu(y')}) d y'
    & =  \int G(y') \left( 1+ \frac{\epsilon_{\mu(y')}}{1 - \epsilon_{\mu(y')}}\right) d y'  \\
    & \le 1 +  \int_{\abs{y} \ge c_{1}} G(y') \left( \frac{\epsilon_{\mu(y')}}{1 - \epsilon_{\mu(y')}}\right) d y' + \int_{\abs{y} < c_{1}} G(y') \frac{\epsilon}{1 - \epsilon} d y' \\
    & \le 1 + 2 c_{1} \epsilon +  \int G(y') {\mu(y')^{2}} d y'  \\
    & \le 1 + 2 c_{1} \epsilon +  \int G(y') {c_{1}^{2}\epsilon{y'}^{2}} d y'  \\
    & \le 1 + 2 c_{1} \epsilon +  {c_{1}^{2}\epsilon}  \\
    & \le 1 / (1 - \epsilon),
  \end{align*}
  where the last step holds by choosing \(c_{1}\) to be sufficiently small.

  Since for any \(y\) and \(A\), \(D_{\mu(y)} \ge (1 - \epsilon_{\mu(y)})\normal(\mu(y), 1-c_{2}\epsilon^{2})\) by the construction of \(D_{\mu}\), so \(P_{y, v} \ge (1 - \epsilon_{\mu(y)})\normal(\mu(y) \vect(v v^{\top}), I -c_{1}^2\epsilon^{2}\vect(v v^{\top}) \vect(v v^{\top})^{\top}) \), since \(P_{y, v}\) agrees with Gaussian in all directions except \(\vect(v v^{\top})\). We proved that \(R(y)\ge (1 - \epsilon) G(y)/(1 - \epsilon_{\mu(y)})\), so \(Q_{v}' = P_{\mu(y), v}R(y) \ge (1 - \epsilon) G(y) \normal(\mu(y) \vect(v v^{\top}), I -c_{1}^2\epsilon^{2}\vect(v v^{\top}) \vect(v v^{\top})^{\top})  = (1 - \epsilon) Q_{v} \).
\end{proof}

Next we show that \(Q_{v}'\)  are near orthogonal to each other for distinct \(v\), if we view \(\chi_{S}(\cdot, \cdot) - 1\) as an inner product on the space of \(d^{2}+1\)-dimensional distributions.
\begin{lemma}[Near Orthogonality of Corrupted Joint Distributions]
  \label{lemma:near_orthogonal_distribution}
  Let \(S\) be the joint distribution of \(A\) and \(y\) when they are independent and entries of \(A\) are i.i.d.\ standard Gaussian and \(y \sim R\). Then we have \[\chi_{S}(Q_{v}', Q_{v'}') - 1= e^{O(1/\epsilon^{3})}(v^{\top} v')^{4}.\]
\end{lemma}
\begin{proof}
  By the definition of \(\chi_{S}(Q_{v}', Q_{v'}')\), we have
  \begin{align*}
    \chi_{S}(Q_{v}', Q_{v'}')
    & = \int Q_{v}'(a, y) Q_{v'}'(a,y) / S(x, y) dx dy \\
    & = \int P_{\mu(y), v}(a) P_{\mu(y), v'}(a) R(y)^{2} / (G(a)R(y)) dx dy \\
    & = \int \chi_{\normal{(0, I_{d^{2}})}}(P_{\mu(y), v}, P_{\mu(y), v'})R(y) dy \\
    & \le 1 + \bigO{(v^{\top} v')^{4}} \int \chi^{2}(D_{\mu(y), \epsilon}, \normal(0,1)) R(y) d y,
  \end{align*}
where the last inequality follows from Lemma~\ref{lemma:corrupted-conditional}.

  Recall \(\mu(y) = c_{1} \sqrt\epsilon y\). The main technical part is to bound the following quantity:
  \begin{align*}
    \int \chi^{2}(D_{\mu(y), \epsilon}, \normal(0,1)) R(y) d y
    & \le e^{O(1/\epsilon)} + \int_{\abs{y} \ge c_{1}} e^{\bigO{\max\{1/\mu(y)^{2}, \mu(y)^{2}\}}} R(y) d y \\
    & \le e^{O(1/\epsilon)} + \int_{\abs{y} \ge c_{1}} e^{\bigO{\max\{1/(c_{1}^{4}\epsilon), \mu(y)^{2}\}}} R(y) d y \\
    & \le e^{O(1/\epsilon)} +  e^{1/(c_{1}^{4}\epsilon)} \int_{\abs{y} \ge c_{1}} e^{c_{1}^{2} \epsilon y^{2}} R(y) d y \\
    & \le e^{O(1/\epsilon)} + e^{1/(c_{1}^{4}\epsilon)} \int_{\abs{y} \ge c_{1}} e^{c_{1}^{2} \epsilon y^{2}} (1 - \epsilon) G(y) / (1 - \epsilon_{\mu(y)})d y \\
    &  \le e^{O(1/\epsilon)} + e^{1/(c_{1}^{4}\epsilon)} \int_{\abs{y} \ge c_{1}} e^{c_{1}^{2} \epsilon y^{2}} c_{1}^{2} y^{2} G(y) d y \\
    & = e^{\bigO{1/\epsilon}},
  \end{align*}
  where the first inequality splits into two cases as in Lemma~\ref{lemma:mixing_one_dimensional}, and the last inequality follows by choosing $c_1$ to be a sufficiently small constant so that $c_1^2 \eps \le 1/4$ and $\int e^{c_{1}^{2} \epsilon y^{2}} y^{2} G(y) d y $ is bounded above by a constant.
\end{proof}

\begin{proof}[Proof of Theorem~\ref{thm:sq_lower_bound}]
  The proof follows a standard non-Gaussian component analysis argument. For any \( 0 < c < 1/2\), there is a set \(S\) of at least \(2^{d^{c}}\) unit vectors in \(\R^{d}\) such that for each pair of distinct \(v, v' \in S\), \(\abs{v^{\top} v'} = O(d^{c - 1/2})\). By Lemma~\ref{lemma:near_orthogonal_distribution}, we have
  \[
    \chi_{S}(Q_{v}', Q_{v'}') - 1 =
    \begin{cases}
       e^{O(1/\epsilon)} O(d^{4c-2}) =: \gamma & \mbox{ if } v \neq v' \\
      e^{O(1/\epsilon)} =: \beta & \mbox{ if } v = v'
    \end{cases}.
  \]
  Therefore, by Lemma 8.8 in~\cite{diakonikolas2022algorithmic}, any SQ algoritm requires at least \(2^{d^{c}} \gamma / \beta = 2^{d^{c}}/ d^{2-4c}\) calls to \(\mathsf{STAT}(e^{1/\sqrt\epsilon} /d^{1-2c})\) to find \(v\) and therefore \(u\) within better than \(\bigO{\epsilon^{1/4}}\).
\end{proof}

We state the auxiliary facts that we used in the above analysis:
\begin{fact}[{\cite[Lemma G.3]{diakonikolas2021statistical}}]
  \label{fact:chi_square_mixture}
	        Let $k \in \Z_+$, distributions $P_i$ and $\lambda_i\geq 0$, for $i \in [k]$ such that  $\sum_{i=1}^k \lambda_i=1$. We have that $\chi^2\left(\sum_{i=1}^k \lambda_i P_i,D  \right)= \sum_{i=1}^k \sum_{j=1}^k \lambda_i \lambda_j \chi_D(P_i,P_j)$.
\end{fact}

\begin{fact}[{\cite[Lemma F.8]{diakonikolas2019efficient-algor}}]
  \label{fact:chi_square_gaussians}
\begin{align*}
	    \chi_{\normal(0,1)}\left(\normal(\mu_1,\sigma_1^2), \normal(\mu_2,\sigma_2^2)\right) = \frac{\exp\left(-\frac{\mu_1^2(\sigma_2^2-1) +2\mu_1 \mu_2 + \mu_2^2(\sigma_1^2-1)}{2\sigma_1^2(\sigma_2^2-1) - 2\sigma_2^2} \right)}{\sqrt{\sigma_1^2 + \sigma_2^2 - \sigma_1^2\sigma_2^2}}.
\end{align*}
\end{fact}

\section{\new{A Counterexample: Why Simple Algorithms Do Not Work}}
\label{sec:count-why-simple}
While Section~\ref{sec:sq-lower-bound-appendix} formally showed that a broad class of algorithms with sample complexity proportional to the dimension \(d\) cannot achieve \(O(\epsilon^{1/4})\) error without an exponential amount of computation, we informally discuss a concrete construction of adversarial corruptions where some straightforward algorithms fail. This discussion is intended only to provide intuition; the SQ lower bound provides more rigorous evidence that improving sample complexity beyond quadratic is impossible. The \(O(\cdot)\) notation in this section only displays dependence on the dimension \(d\).

Let \((A_{i}, y_{i}) \sim \G\) be sampled according to the noiseless Gaussian design in Definition~\ref{def:iid_Gaussian_with_noise}, i.e.\ entries of \(A_{i}\) are i.i.d.\ standard Gaussians and \(y_{i} = \<M^{*}, A_{i} \> \).

\begin{itemize}
  \item  The adversary first inspects all \(n\) samples and computes \(M^{*}\) up to some small error with out-of-shelf low rank matrix sensing algorithms. In the following discussion, we simply assume the adversary finds \(M^{*}\) exactly.

  \item The adversary discards a random \(\eps\)-fraction of samples and let \(\mathcal{S}\) denote the set of remaining samples. The adversary computes \(P = -\frac1n\sum_{i \in \mathcal{S}}y_{i}A_{i}\).
  \item For \(j = 1, \dots, \epsilon n\), the adversary independently samples  \(z_{j}\) uniformly between the spheres of radius \(\fnorm{M^{*}}/2\) and \(2 \fnorm{M^{*}}\)  and computes \(E_{j} = \frac{P}{\epsilon z_{j}} + A_{j}'\), where entries of \(A_{j}'\) are i.i.d.\ sampled from standard Gaussian.
  \item The adversary adds \((z_{j}, E_{j})\) as corruptions. Let \(\mathcal{B}\) denote the points added by the adversary.
\end{itemize}

Recall that with the objective function \(f_{i}\) defined in Equation~\eqref{eq:obj_u_i} and \(B_{i} = A_{i} + A_{i}^{T}\), we have the gradient and Hessian
\begin{align*}
  \grad f_{i}(U) &= (\< U U^{\top}, A_{i} \> - y_{i})  (A_{i} + A_{i}^{\top}) U \\
  \grad^{2} f_{i}(U) &= (\< U U ^{\top}, A_{i} \> - y_{i}) I_{r} \otimes B_{i}  + \vect \left(B_{i}U  \right) \vect(B_{i}U)^{\top}.
\end{align*}

Prior works on first-order robust stochastic
optimization~\cite{prasad2020robust-estimati, diakonikolas2019sever} required
\(O(d)\) samples, because they rely on the robust estimation of the gradients.
They would accept \(0\) as the solution, because when \(U = 0\),
\(\grad f_{i}(U) = 0\) for all \(i\), i.e., all sample gradients are zero at
\(U = 0\) whether or not they are corruptions, so gradient-based filtering does
not work. On the other hand,
\(\grad^{2}\bar f(0) = - \E [\<M^{*}, A_{i}\> I_{r} \otimes B_{i} ]= - I_{r} \otimes \E[ \<M^{*}, A_{i}\> (A_{i} + A_{i}^{T})] = -2 I_{r}\otimes M^{*}\),
where the last equality follows from
Lemma~\ref{lemma:Gaussian_quadratic_form_rank1}. Hence the error measured by the
negative curvature of the solution is
\(\lambda_{\min}(\grad^{2}\bar f (0)) = 2 \sigstarl\), which does not even scale
with \(\epsilon\).

This adversary can also help illustrate why spectral initialization similar to the algorithm
discussed in Theorem~\ref{thm:together_high_noise} cannot succeed with \(O(d)\)
samples. Spectral methods rely on the observation that with clean samples,
\(\E_{(A_{i}, y_{i}) \sim \G}[y_{i} A_{i}] = M^{*}\). The hope is that the matrix formed by the leading \(r\) singular vectors of some robust mean estimation of \(\{y_{i}A_{i}\}\) is close enough to \(M^{*}\). If so, it can be a good initialization to Algorithm~\ref{alg:local_linear_convergence}, which only requires robust estimation of sample gradients and \(O(d)\) samples suffice.

However, with \(O(d)\) samples, it is unclear how to estimate the mean of \(\{y_{i}A_{i}\}\). Robust mean estimation algorithms (e.g. Algorithm~\ref{alg:robust_mean_estimation}) discussed in Proposition~\ref{prop:robust-mean-estimation-summary} do not work because they require \(O(d^{2})\) samples. Here we show that the naive filter based on the norm of \(\{y_{i}A_{i}\}\) does not work, either.

Observe that \(P \approx -M^{*}\) by construction. For each \(j\), \[\fnorm{E_{j}} \le \fnorm{A_{j}'} +  \fnorm{P} / (z_{i}\epsilon) \approx \fnorm{A_{j}'} +  \fnorm{M^{*}} / (z_{i}\epsilon) \le \fnorm{A_{j}'} + 2/\eps.\] Since \(\fnorm{A_{j}'}\) scales with the dimension \(d\) but \(\eps\) is a fixed constant, as we move to a high-dimensional regime, \(\fnorm{A_{j}'}\) dominates and the effects of the outliers diminish if we only look at the norm of data.

Without an effective filter, the mean of \(\{y_{i}A_{i}\}\) across both clean samples and outliers becomes
\[\frac1n\sum_{i \in \mathcal{S}}y_{i}A_{i} + \frac1n\sum_{j \in \mathcal{B}}z_{j}E_{j} = - P + \frac1n\sum_{j\in\mathcal{B}} P / \epsilon + A_{j}' = \frac1n\sum_{j \in \mathcal{B}}A_{j}',\]
which contains no signals --- only noise --- and is close to 0.

In summary, we constructed an adversary under which robust first-order methods and spectral methods with naive filter fail. Note that this construction relies on the ability of the adversary to corrupt the input matrices \(A_{i}\); if the adversary could only corrupt \(y_{i}\), then the results in \cite{li2020non,li2020nonconvex} would apply and \(\tilde O(d)\) samples would suffice.
\end{document}